\documentclass[10pt]{amsart}
\usepackage{amsfonts,amsmath,amssymb,color,amscd,amsthm}
\usepackage[T1]{fontenc}  
\usepackage[francais,english]{babel} 
\usepackage[all]{xy} 
\usepackage{tikz-cd}
\usepackage{pgf,tikz}
\usepackage{pgfplots}
\usepgfplotslibrary{patchplots} 
\usepackage{mathtools}
 
\pgfplotsset{
 compat=1.11,
 My Line Style/.style={
  smooth,
  thick,
  samples=400,
 },
 }

\usepackage{stmaryrd}
\usepackage[normalem]{ulem}

\newcommand\redsout{\bgroup\markoverwith{\textcolor{red}{\rule[1.5ex]{2pt}{0.4pt}}}\ULon}

\usepackage[shortlabels]{enumitem}
\setlist[1]{wide}
\setlist[2]{leftmargin=15mm}
\setlist[enumerate]{label=\rm{(\arabic*)}}
\setlist[enumerate,2]{label=\rm({\it\roman*}), }
\setlist[itemize]{label=\raisebox{0.25ex}{\tiny$\bullet$}}

\usepackage[backref, colorlinks, linktocpage, citecolor = blue, linkcolor = blue]{hyperref}

\newcommand\iso{\stackrel{\simeq}{\longrightarrow}}

\newcommand\Jac{\mathrm{Jac}}
\newcommand\car{\mathrm{char}}
\newcommand\NN{{\mathbb N}}

\newcommand\RR{{\mathbb R}}
\renewcommand\AA{{\mathbb A}}

\newcommand\CC{{\mathbb C}}

\newcommand\PP{{\mathbb P}}

\newcommand\Pic{\mathrm{ Pic}}

\newcommand\tr{\hbox to 1mm {${}^t \! $} }

\renewcommand\k{\mathbf{ \textbf k}}

\newcommand\kany{\mathrm{k}}

\DeclareMathOperator{\Aut}{Aut}
\DeclareMathOperator{\Bir}{Bir}
\DeclareMathOperator{\Aff}{Aff}
\DeclareMathOperator{\GL}{GL}
\DeclareMathOperator{\PGL}{PGL}

\DeclareMathOperator{\End}{End}

\DeclareMathOperator{\Spec}{Spec}

\DeclareMathOperator{\Span}{span}

\DeclareMathOperator{\Triang}{Triang}
\DeclareMathOperator{\Tame}{Tame}

\newtheorem{theorem}{Theorem}
\newtheorem*{maintheorem*}{Main Theorem}
\newtheorem{lemma}{Lemma}[subsection]
\newtheorem{corollary}[lemma]{Corollary}
\newtheorem{theoremsimple}[lemma]{Theorem}
\newtheorem{proposition}[lemma]{Proposition}

\newtheorem*{corollary*}{Corollary}
\newtheorem*{corollary**}{Corollary}

\theoremstyle{definition}
\newtheorem{definition}[lemma]{Definition}

\theoremstyle{remark}
\newtheorem{remark}[lemma]{Remark}
\newtheorem{example}[lemma]{Example}

\newcommand{\Z}{\mathbb{Z}}
\newcommand{\p}{\mathbb{P}}
\newcommand{\A}{\mathbb{A}}

\newcommand{\F}{\mathbb{F}}

\newcommand{\bir}{\dashrightarrow}

\makeatletter
\ifcsname phantomsection\endcsname
\newcommand*{\qrr@gobblenexttocentry}[5]{}
\else
\newcommand*{\qrr@gobblenexttocentry}[4]{}
\fi
\newcommand*{\addsubsection}{
	\addtocontents{toc}{\protect\qrr@gobblenexttocentry}
	\subsection}
\makeatother

\newcommand{\set}[2]{\left\{\,#1 \ | \ #2\,\right\}}
\newcommand{\Bigset}[2]{\left\{\,#1 \ \Big| \ #2\,\right\}}

\title[Automorphisms of the affine $3$-space of degree $3$]{Automorphisms of the affine $3$-space of degree $3$} 


\author{J\'er\'emy Blanc}
\address{J\'er\'emy Blanc, Universit\"{a}t Basel, Departement Mathematik und Informatik, Spiegelgasse $1$, CH-$4051$ Basel, Switzerland}
\email{jeremy.blanc@unibas.ch}

\author{Immanuel van Santen}
\address{Immanuel van Santen, Universit\"{a}t Basel, Departement Mathematik und Informatik, Spiegelgasse $1$, CH-$4051$ Basel, Switzerland}
\email{immanuel.van.santen@math.ch}

\begin{document}
\begin{abstract}
		In this article we give two explicit families of automorphisms of degree $\leq 3$ 
		of the affine $3$-space $\mathbb{A}^3$
		such that each automorphism of degree $\leq 3$ of $\mathbb{A}^3$
		is a member of one of these families up to composition of affine automorphisms at the source and target; this shows in particular that all of them are tame.
		As an application, we give the list of all 
		dynamical degrees of automorphisms of degree $\leq 3$
		of $\mathbb{A}^3$; this is a set of $3$ integers and $9$ quadratic integers. 
		Moreover, we also describe up to compositions with affine automorphisms for $n\geq 1$ 
		all morphisms $\mathbb{A}^3 \to \mathbb{A}^n$ of degree $\leq 3$ with the property that 
		the preimage of every affine hyperplane in $\mathbb{A}^n$ is isomorphic to~$\mathbb{A}^2$.
\end{abstract}

\maketitle

\tableofcontents

	\section{Introduction}
	\subsection{The results}
	In this text, we fix an algebraically closed field $\k$ of any characteristic. 
	We denote by $\AA^n$ or sometimes $\AA^n_{\k}$ the \emph{affine $n$-space}
	$\Spec(\k[x_1, \ldots, x_n])$ over $\k$ for a specified choice of coordinates $x_1,  \ldots , x_n$.
	Every morphism $f \colon \AA^n \to \AA^m$ is given by
	\[	
		\begin{array}{rcl}
			\AA^n &\stackrel{f}{\longrightarrow}& \AA^m \\
			(x_1, \ldots, x_m) &\longmapsto& 
			(f_1(x_1, \ldots, x_n), \ldots, f_m(x_1, \ldots, x_n))
		\end{array}
	\]
	for polynomials $f_1, \ldots, f_m \in \k[x_1, \ldots, x_n]$. 
	If $n=3$, we often use $x,y,z$ instead of $x_1,x_2,x_3$ as coordinates.
	For simplicity we denote the above morphism sometimes by $f = (f_1, \ldots, f_m)$.
	For a morphism $f = (f_1, \ldots, f_n) \colon \AA^n \to \AA^m$ we denote by $\deg(f)$
	its degree which is by definition equal to the maximum of the degrees $\deg(f_1), \ldots, \deg(f_n)$.
	
	Let $\Aut_\k(\AA^n)$ be the group of all automorphisms of $\AA^n$ over $\k$. 
	In the last decades, there has been done
	a lot of research on this group $\Aut_\k(\AA^n)$, see e.g. the survey \cite{Es2000Polynomial-automor}.
	There are two prominent subgroups of $\Aut_{\k}(\AA^n)$, namely the group of \emph{affine automorphisms}
	\[
			\Aff_{\k}(\AA^n) = \Bigset{(f_1, \ldots, f_n) \in \Aut_{\k}(\AA^n)}{ 
					\begin{array}{l}
						\textrm{$f_i \in \k[x_1, \ldots, x_n]$ and $\deg(f_i) = 1$} \\ 
						\textrm{for all $i = 1, \ldots, n$}
					\end{array}
				} \\
	\]
	and the group of \emph{triangular automorphisms}
	\[
			\Triang_{\k}(\AA^n) = \Bigset{(f_1, \ldots, f_n) \in \Aut_{\k}(\AA^n)}{ 
				\begin{array}{l}
					\textrm{$f_i \in \k[x_i, \ldots, x_n]$} \\
					\textrm{for all $i = 1, \ldots, n$}
				\end{array}
				} \, .
	\]
	The subgroup generated by $\Aff_{\k}(\AA^n)$ and $\Triang_{\k}(\AA^n)$ inside $\Aut_{\k}(\AA^n)$
	is called the group of \emph{tame automorphisms} and we denote it by $\Tame_{\k}(\AA^n)$.
	In case $n = 1$, all automorphisms of $\AA^1$ are tame (in fact they are affine) and
	for $n = 2$ it is proven by Jung and van der Kulk \cite{Ju1942Uber-ganze-biratio, Ku1953On-polynomial-ring} 
	that all automorphisms of $\AA^2$ are tame.
	Since a long time it was conjectured that the famous Nagata-automorphism
	\[
		(x-2 y(zx+y^2)-z(zx+y^2)^2, y+z(zx+y^2), z) \in \Aut_{\k}(\AA^3)
	\]
	is non-tame, until Shestakov and Umirbaev gave fifteen years ago 
	an affirmative answer if $\car(\k) = 0$, see \cite{ShUm2004The-tame-and-the-w}. 
	It is still an open problem whether 
	$\Tame_{\k}(\AA^n) \neq \Aut_{\k}(\AA^n)$ for $n \geq 4$ and when $\car(\k) \neq 0$ also for $n = 3$.	
	
	It is conjectured by Rusek \cite{Ru1988Two-dimensional-ja} 
	that all automorphisms of $\AA^n$ of degree $2$ are tame. 
	If $n = 3$ and $\k = \CC$, Fornaes and Wu \cite{FoWu1998Classification-of-} classified all
	automorphisms of $\AA_{\CC}^3$ of degree $2$ up to conjugation by 
	affine automorphisms and it turned out that all of them
	are triangular up to composition of affine automorphisms at the source and target. For $n = 4$
	and $\k = \RR$, Meisters and Olech \cite{MeOl1991Strong-nilpotence-} and for $n = 5$ and $\k = \CC$, 
	Sun \cite{Su2014Classification-of-} gave affirmative answers to Rusek's conjecture.
	
	Motivated by these investigations of the tame automorphisms in $\Aut_{\k}(\AA^n)$, 
	we study in this paper
	automorphisms of $\AA^3$ of degree $3$. For this let us introduce
	the following equivalence relation: $f, g \in \Aut_{\k}(\AA^n)$ are \emph{equivalent}
	if there exist $\alpha, \beta \in \Aff_{\k}(\AA^n)$ such that $f = \alpha \circ g \circ \beta$. 
	The main theorem of this article is the following description of degree $3$ automorphisms of $\AA^3$:
	
	\begin{theorem}[{see Theorem~\ref{thm:Classification}}]\label{Thm:Main}
			Each automorphism of $\AA^3$ 
			of degree $\leq 3$ is either equivalent to a triangular automorphism 
			or to an automorphism of the form
			\[
				\label{eq.non_triangular}
				\tag{$\ast$}
				(x + yz + za(x, z), y + a(x, z) + r(z), z) \in \Aut_{\k}(\AA^3)
			\]
			where $a \in \k[x, z] \setminus \k[z]$ is homogeneous of degree $2$
			and $r \in \k[z]$ is of degree $\leq 3$.
	\end{theorem}
	
	In fact we prove that none of the automorphisms of~\eqref{eq.non_triangular} is
	equivalent to a triangular automorphism, see Proposition~\ref{prop.Cases_distinct}.
	
	Theorem~\ref{Thm:Main} implies in particular that all automorphisms of degree $\leq 3$ of $\AA^3$ are tame,
	see Corollary~\ref{cor.deg3_autos_tame}.
	
	As an other application of Theorem~\ref{Thm:Main} we 
	compute all dynamical degrees
	of automorphisms of degree $\leq 3$.
	Recall, that the dynamical degree of an automorphism $f \in \Aut(\AA^n)$ is defined by
	\[
		\lambda(f) = \lim_{r \to \infty} \deg(f^r)^{\frac{1}{r}} \in \RR_{\geq 1},
	\]
	satisfies $1\le \lambda(f)\le \deg(f)$ and is invariant under conjugation (in $\Aut(\AA^n)$ but also in the bigger group $\Bir(\AA^n)$ of birational transformations of $\AA^n$). It gives information about the iteration of the automorphism $f$. 
	The dynamical degree of an automorphism of $\A^2$ is always an integer, and all possible integers are possible, by simply taking $(x,y)\mapsto (y,x+y^d)$, for each $d\ge 1$. 
	The set of dynamical degrees of automorphisms of $\A^3$ is still quite mysterious. In 2001, K. Maegawa proved that the set of dynamical degrees of all automorphisms of $\A^3_\CC$ of degree $2$ is equal to $\{1,\sqrt{2},(1+\sqrt{5})/2,2\}$ \cite[Theorem~3.1]{Maequadratic}. This also holds for each field (Theorem~\ref{thm.Dynamical_degrees} below). Recently, we proved that for each $d\ge 1$ and each ground field $\kany$, the set of all dynamical degrees of automorphisms of $\A^3_\kany$ of degree $\le d$ that are equivalent to a triangular automorphism is 
	\[
		\Bigset{ \frac{a+\sqrt{a^2+4bc}}{2}}{(a,b,c)\in \NN^3, a+b\le d, c\le d} \setminus \{0\} \, ,
	\] 
	see \cite[Theorem 1]{BlSa2022Dynamical-degrees-}, reproduced below as Theorem~\ref{TheoremDynAffTri}. 
	Using Theorem~\ref{Thm:Main}, we prove the following result:
	\begin{theorem}
	\label{thm.Dynamical_degrees}
	For each $d\ge 1$ and each field $\kany$, let us denote by $\Lambda_{d,\kany}\subset\mathbb{R}$ the set of dynamical degrees of all automorphisms of $\A^3_{\kany}$ of degree $d$. We then have
	\[\begin{array}{rcl}
	\Lambda_{1,\kany}&=&\{1\}\\
	\Lambda_{2,\kany}&=&\{1,\sqrt{2},(1+\sqrt{5})/2,2\}\\
	\Lambda_{3,\kany}&=&\{1 , \sqrt{2} , \ \frac{1+\sqrt{5}}{2} , \ \sqrt{3} , 2 , \ \frac{1 + \sqrt{13}}{2} , \
		1 + \sqrt{2} , \ \sqrt{6} , \ \frac{1 + \sqrt{17}}{2} , \ \frac{3+\sqrt{5}}{2} , \ 1+ \sqrt{3} , \ 3 \}
		\, .
	\end{array}\]
	\end{theorem}
Note that the automorphisms in~\eqref{eq.non_triangular} in Theorem~\ref{Thm:Main}
	all fix a linear projection $\AA^3 \to \AA^1$ and thus the dynamical degree of these automorphisms are
	integers, see e.g.~\cite[Corollary 2.4.3]{BlSa2022Dynamical-degrees-}. Thus one has to permute
	the coordinate functions of these automorphisms in order to 
	produce interesting dynamical degrees.
The most interesting number in Theorem~\ref{thm.Dynamical_degrees}
is $(3+\sqrt{5})/2$. It is the dynamical degree of 
$f=(y + xz, z,x+z(y + xz)) \in \Aut(\A^3_{\kany})$, for each field $\kany$.
 It follows from \cite[Theorem 1]{BlSa2022Dynamical-degrees-} that $\lambda(f)=(3+\sqrt{5})/2$ is not the dynamical degree of any automorphism of $\A^3$ that is equivalent (over $\kany$ or over its algebraic closure $\k=\overline{\kany}$) to a triangular automorphism, of any degree, see \cite[Example 4.4.6]{BlSa2022Dynamical-degrees-}. The fact that all  dynamical degrees above arise 
essentially follows from \cite{BlSa2022Dynamical-degrees-}, the main 
contribution of this text to Theorem~\ref{thm.Dynamical_degrees} 
is to show that we cannot get more dynamical degrees.
Theorem~\ref{thm.Dynamical_degrees} implies that every dynamical degree of an element of $\Aut(\A^3)$ of degree $2$ is also the dynamical degree of an element of $\Aut(\A^3)$ of degree $3$, contrary to the case of dimension $2$ (an element of $\Aut(\A^2)$ of degree $3$ has dynamical degree equal to either $1$ or $3$).
	
	\subsection{Outline of the article}
	In order to classify all automorphisms of degree $\leq 3$ up to equivalence we study first
	degree $3$ polynomials in $\k[x, y, z]$ that define the affine plane $\AA^2$ in $\AA^3$
	in Section~\ref{sec.cubic_hypersurfaces}.
	The closure in $\PP^3$ of such a hypersurface in $\AA^3$ is singular, 
	so the polynomial has the form $xp + q$ for some $p, q \in \k[y, z]$ up to an affine automorphism, see
	Corollary~\ref{cor.cubic_surf_iso_to_A2_singular}. These polynomials were studied by
	Sathaye~\cite{Sathaye1976} for fields with $\car(\k) = 0$ and by Russell \cite{Russell1976} for all fields
	and it turns out that all of them are variables of $\k[x, y, z]$, i.e.~there are polynomials
	$g, h \in \k[x, y, z]$ with $\k[xp+q, g, h] = \k[x, y, z]$, see also Propositions~\ref{Prop:XxpqisA2thendisjoint},~\ref{Prop:XfibratReminder} 
	and Corollary~\ref{coro:TrivialisationVariables} for more
	detailed informations. We then give a description
	of all such hypersurfaces up to affine automorphisms
	(Proposition~\ref{Proposition:XxpqisA2smalldegreelistpqDeg3}).
	As the polynomials of degree $3$ of the form $xp+q$ correspond to cubic hypersurfaces of $\A^3$ whose closures in $\p^3$ are
	singular at $[0:1:0:0]$ (Lemma~\ref{lem.hypersxpq}), it is also useful to classify them up to affine automorphisms that fix this point; this is done in Proposition~\ref{Proposition:XxpqisA2smalldegreelistpq}, where a bigger list is given. Corollary~\ref{cor:Degree3onepointoneline} then corresponds to the case where we focus on a line at infinity instead of a point.
	
	Then we investigate these hypersurfaces in families in 
	Section~\ref{sec.families_cubic_surf}. 
	The best suited notion for us is the following:
	a morphism $f \colon \AA^d \to \AA^n$ is called an \emph{affine linear system of affine spaces} if the preimage of each affine hyperplane of $\AA^n$ is isomorphic to $\AA^{d-1}$,
	see Definition~\ref{def.aff_linear_system_of_aff_spaces}.
	In case $d = 3$, we say that $f$ is \emph{in standard form} if 
	$f = (xp_1 + q_1, \ldots, x p_n + q_n)$
	for some polynomials $p_i, q_i \in \k[y, z]$. 
	An affine linear system of affine spaces $g \colon \AA^3 \to \AA^n$
	of degree $3$ is equivalent to one in standard form if and only if for general 
	affine hyperplanes $H \subset \AA^3$ the closures of
	$g^{-1}(H)$ in $\PP^3$ have a common singularity, see Lemma~\ref{lem.hypersxpq}. 
	Two affine linear systems of affine spaces $f, g \colon \AA^d \to \AA^n$ are called 
	\emph{equivalent} if there are $\alpha \in \Aff_{\k}(\AA^n)$ and $\beta \in \Aff_\k(\AA^3)$
	such that $f = \alpha \circ g \circ \beta$.
	The key point
	in the proof of Theorems~\ref{Thm:Main} and~\ref{thm:Classification} 
	is to show that each affine linear system of affine spaces
	$\AA^3 \to \AA^3$ of degree $\leq 3$ is equivalent to one in standard form, 
	see Proposition~\ref{prop:globally_xp_i+q_i}.

	In Section~\ref{sec.classific}, we give a description of all affine linear systems
	of affine spaces $\AA^3\to \AA^n$ of degree $\leq 3$
	which implies Theorem~\ref{Thm:Main}. We call a morphism $f \colon Y \to X$ an \emph{$\AA^1$-fibration} if each closed fiber is (schematically) isomorphic to $\AA^1$
	and we call $f$ a \emph{trivial $\AA^1$-fibration} if there exists an isomorphism
	$\varphi \colon X \times \AA^1 \to Y$ such that the composition 
	$f \circ \varphi \colon X \times \AA^1 \to X$ is the projection onto the first factor.
	Note that the above definition of an $\AA^1$-fibration differs
	from the notions of $\AA^1$-fibrations in 
	\cite{GuMaMi2012Bbb-A1-fibrations-}  and \cite{KaMi1978On-flat-fibrations}.
	In fact we show:
	
	\begin{theorem}
	\label{thm:Classification}
			Every affine linear system of affine spaces $\AA^3 \to \AA^n$ of degree $\leq 3$ 
			is equivalent to an element of the following eleven families. Case $\mathrm{I)}$ corresponds
			to $n=1$,  Cases $\mathrm{IIa)}$ and $\mathrm{IIb)}$ correspond to $n=2$ and Case $\mathrm{III)}$ corresponds to $n=3$. 
			\itemindent = 6pt
			\item[$\mathrm{I)}$] variables of $\k[x, y, z]$:\label{thm:Class1}
				\begin{enumerate}[leftmargin=*, label=$(\arabic*)$]
					\item \label{thm:Classification_1}
					$x+r_2(y, z)+r_3(y, z)$ where $r_i \in \k[y,z]$ is homogeneous of degree $i$;
					\item \label{thm:Classification_2}
					$xy+yr_2(y,z)+z$ where $r_2\in \k[y,z]\setminus \k[y]$ is homogeneous of degree $2$;
					\item \label{thm:Classification_3}
					$xy^2+y (z^2+az+b)+z$ where $a,b\in \k$.
				\end{enumerate}
			\item[$\mathrm{IIa)}$] trivial $\AA^1$-fibrations: 
				\begin{enumerate}[leftmargin=*, label=$(\arabic*)$]
					\setcounter{enumi}{3}
					\item \label{thm:Classification_4}
					$(x+p_2(y, z) + p_3(y, z),y+q_2 z^2 + q_3 z^3)$ where $p _i\in \k[y, z]$ 
					is homogeneous of degree $i$ and $q_2, q_3 \in \k$;
					\item \label{thm:Classification_5}
					$(yz + za_2(x, z) + x, y + a_2(x, z) + r_1 z + r_2 z^2 + r_3 z^3)$ 
					where $a_2 \in \k[x, z] \setminus \k[z]$ is homogeneous of degree $2$ and 
					$r_i \in \k$;
					\item \label{thm:Classification_6}
					$(yz + za_2(x, z) + x, z)$ where $a_2 \in \k[x, z] \setminus \k[z]$ is homogeneous of 
					degree $2$;
					\item \label{thm:Classification_7}
					$(xy^2 + y (z^2 + az + b) + z, y)$ where $a, b \in \k$.
				\end{enumerate}
				\item[$\mathrm{IIb)}$] non-trivial $\AA^1$-fibrations: 
				\begin{enumerate}[leftmargin=*, label=$(\arabic*)$]
					\setcounter{enumi}{7}
					\item \label{thm:Classification_8}
					$(x+ z^2 +y^3, y + x^2)$ where $\car(\k) = 2$;
					\item \label{thm:Classification_9}
					$(x + z^2 + y^3, z + x^3)$ where $\car(\k) = 3$.		
				\end{enumerate}
				\item[$\mathrm{III)}$] automorphisms of $\AA^3$:
				\begin{enumerate}[leftmargin=*, label=$(\arabic*)$]
					\setcounter{enumi}{9}
					\item \label{thm:Classification_10} 
					$(x+p_2(y, z) + p_3(y, z),y+q_2 z^2 + q_3 z^3,z)$ 
					where $p _i\in \k[y, z]$ is homogeneous of degree $i$ and $q_2, q_3 \in \k$;
					\item \label{thm:Classification_11} 
					$(yz + za_2(x, z) + x, y + a_2(x, z) + r_2 z^2 + r_3 z^3, z)$ where
					$a_2 \in \k[x, z] \setminus \k[z]$ is homogeneous of degree $2$ and $r_2, r_3 \in \k$.
				\end{enumerate} 
	\end{theorem}
	
	The proof of Theorem~\ref{thm:Classification} 
	is given towards the end of Section~\ref{sec.classific}. 
	All the eleven cases in our list are in fact pairwise non-equivalent, see
	Proposition~\ref{prop.Cases_distinct}. For $n = 1$ and $\k = \CC$, 
	Ohta gave in \cite[Theorem 1]{Oh1999The-structure-of-a}
	a list of all possibilities for affine linear systems of affine spaces $\AA^3 \to \AA^1$
	of degree $\leq 3$, together with a description of the curve at infinity. This corresponds then to a refined list of the items \ref{thm:Classification_1}-\ref{thm:Classification_2}-\ref{thm:Classification_3} of 
	Theorem~\ref{thm:Classification}. 
	Note that the fact that each affine linear system
	$\AA^3 \to \AA^1$ of affine spaces of degree $\leq 3$ is equivalent to one of the items~\ref{thm:Classification_1}-\ref{thm:Classification_2}-\ref{thm:Classification_3} is proven in Proposition~\ref{Proposition:XxpqisA2smalldegreelistpqDeg3} below, and is thus the very first part of our study.
	Moreover, Ohta gave in \cite[Theorem 2]{Oh2001The-structure-of-a} and 
	\cite[Theorem 2]{Oh2009The-structure-of-a}
	lists of all possible
	affine linear systems $\AA^3 \to \AA^1$ of affine spaces of degree $4$ in case
	the closure of the corresponding hypersurface in $\PP^3$ is normal. 
	In particular, he proves that all of them are variables of $\AA^3$.
	
	Let us give the connection of our results to the Jacobian conjecture.
	Recall that an endomorphism $f \in \End_{\k}(\AA^n)$ has a constant non-zero Jacobian determinant
	$\det(\Jac(f)) \in \k^\ast$ if and only if for all affine hyperplanes $H \subset \AA^n$
	the preimage $f^{-1}(H)$ is a smooth hypersurface
	of $\AA^n$, see Lemma~\ref{lem.Derksen}. 
	Thus for all $f \in \End_{\k}(\AA^n)$ we have the following implications
	\[
		f \in \Aut_{\k}(\AA^n) \ \implies \
		\begin{array}{l}
			\textrm{$f$ is an affine linear system} \\
			\textrm{of affine spaces}
		\end{array}
		\ \implies \
		\det(\Jac(f)) \in \k^\ast \, .
	\]
	For fields with $\car(\k) = 0$, the Jacobian conjecture asserts that the
	implications are equivalences. For $n = 3$, Vistoli proved the Jacobian conjecture in case $f \in \End_{\k}(\AA^3)$
	has degree $3$, see \cite{Vi1999The-Jacobian-conje}. For fields with $\car(\k) = p > 0$, the 
	last implication is certainly not an equivalence, 
	take e.g. $(x_1 + x_1^p, x_2, \ldots, x_n) \in \End_{\k}(\AA^n)$.
	However, Theorem~\ref{thm:Classification} shows that in case 
	$n = 3$ and $f \in \End_{\k}(\AA^3)$ is of degree $\leq 3$, the
	first implication is an equivalence. 
	
	It is also worth to mention that in case $n = 2$, there are
	affine linear systems of affine spaces $\AA^3 \to \AA^n$ of degree $\leq 3$
	that are $\AA^{3-n}$-fibrations which are not trivial $\AA^{3-n}$-fibrations, 
	contrary to the cases $n = 1$ and $n = 3$.
	In fact, an affine linear system of affine spaces $\AA^3 \to \AA^n$ of degree $\leq 3$
	is a trivial $\AA^{3-n}$-fibration if and only if it is equivalent to a linear system in standard form,
	see Corollary~\ref{cor:Trivial_bundles}. Note that there are even non-trivial $\AA^1$-fibrations
	$\AA^2 \to \AA^1$ in positive characteristic, see~\cite[Example on p.670]{KaMi1978On-flat-fibrations}.
	
	In the last Section, we then compute the dynamical degree of all automorphisms of $\AA^3$ of degree
	$\leq 3$ by using the technique introduced in~\cite{BlSa2022Dynamical-degrees-} and we prove Theorem~\ref{thm.Dynamical_degrees} at the end of this section.

	\addsubsection*{Acknowledgements}
	The authors would like to thank Pierre-Marie Poloni for many fruitful discussions
	and the indication of the references~\cite{Oh1999The-structure-of-a, Oh2001The-structure-of-a, Oh2009The-structure-of-a}. Moreover, we would like to thank the anonymous referee for the very helpful comments, especially 
	for pointing us an elementary argument in Proposition~\ref{Prop:XfibratReminder} 
	that gives \ref{XfibratReminderTechnical}$\Rightarrow$\ref{XfibratReminderXisoA2}.
	
	\subsection{Conventions}
	All schemes, varieties, rational maps and morphisms between them are defined over $\k$. Points of varieties refer to closed points of the associated scheme.
	If $f \colon X \to Y$ is a morphism of varieties, then the fibre over a point $y \in Y$ 
	refers to the \emph{schematic} fibre of $f$ over $y$, i.e.~$f^{-1}(y) = \Spec(\kappa(y)) \times_Y X$
	where $\Spec(\kappa(y)) \to Y$ corresponds to the embedding of the point $y$ in $Y$. More generally,
	the preimage of a closed subvariety $Y'$ of $Y$ corresponds to the \emph{schematic} preimage of $Y'$ under $f$, i.e.~$f^{-1}(Y') = Y' \times_Y X$. If we speak of an $n$-dimensional scheme $X$, then
	we mean that every irreducible component of $X$ has dimension $n$.	
	
	We denote for each $d \geq 0$ by $\k[x_1, \ldots, x_n]_d$ the vector space
	of homogeneous polynomials of degree $d$ in the variables $x_1,\ldots, x_n$.
	By convention, the zero polynomial will be assumed to be homogeneous of any degree $d\ge 0$ (even if it has degree $-\infty$).
	
	\section{Hypersurfaces of $\AA^3$ that are isomorphic to $\AA^2$}
	\label{sec.cubic_hypersurfaces}
	\subsection{Existence of singularities at infinity}
	In the sequel, we always see $\A^3$ as an open subvariety of $\p^3$ 
	via the open embedding $\AA^3 \hookrightarrow \PP^3$, $(x, y, z) \mapsto [1:x:y:z]$ and denote by
	$[w:x:y:z]$ the homogeneous coordiantes of $\PP^3$.

Recall that the multiplicity $m$ of a hypersurface $Y \subseteq \PP^n$ at a given point $p \in Y$ is the multiplicity of the equation at this point, that can be computed locally, or is equivalently the multiplicity at $p$ of the polynomial obtained by restriction of $Y$ to a general line through $p$.

\begin{lemma}
	\label{lem.hypersxpq}
	Let $F \in \k[w, x,y,z]$ be a homogeneous polynomial of degree $d$, let $f = F(1, x, y, z) \in \k[x, y, z]$ 
	and let $X=\Spec(\k[x,y,z]/(f)) \subset \AA^3$ be the corresponding hypersurface. 
	The following conditions are equivalent:
	\begin{enumerate}[leftmargin=*]
		\item \label{degree1x} $f=xp+q$ for some polynomials $p,q\in \k[y,z]$.
		\item \label{degree1xmult} 
		The closure $\overline{X}$ in $\PP^3$ has multiplicity $\ge d-1$ at the point $[0:1:0:0]$.
	\end{enumerate}
\end{lemma}
\begin{proof}
	We write $f=\sum_{i=0}^d x^{d-i} f_i(y,z)$ where $f_i\in \k[y,z]$ is of degree $\le i$ for $i=0,\ldots,d$. For each $i$, we denote by $F_i\in \k[w,y,z]$ the homogeneous polynomial of degree $i$ such that $F_i(1,y,z)=f_i$. This implies that
	$F=\sum_{i=0}x^{d-i} F_i$. Note that $\deg(F) = d$ and that $\overline{X}$ is given by $F$ in $\PP^3$.
	Note that the multiplicity of $\overline{X}$, 
	or equivalently of $F$,
	at the point $[0:1:0:0]$ is the smallest integer $m\ge 0$ such that $F_m$ is not zero. Hence, this multiplicity $m$ satisfies $m\ge d-1$ if and only $F=xF_{d-1}+F_d$, which corresponds to ask that $f=xf_{d-1}+f_d$.
\end{proof}

\begin{corollary}
	\label{cor.cubic_surf_iso_to_A2_singular}
	Let $X \subset \AA^3$ be a hypersurface of degree $d\leq 3$ with $X \simeq \AA^2$.
	\begin{enumerate}[leftmargin=*]
		\item \label{eq1} If $d= 3$, then the closure $\overline{X}$ in $\PP^3$ is singular.
		\item \label{eq2} 
		Up to an affine coordinate change, $X$ is given by $xp + q = 0$
		for polynomials $p, q \in \k[y, z]$ with $\max(\deg(p)+1,\deg(q))= d$.
	\end{enumerate}
\end{corollary}

\begin{proof}
	\ref{eq1}:
	If $\overline{X}$ is a smooth cubic hypersurface of $\PP^3$, then 
	$\Pic(\overline{X}) \simeq \Z^7$, see \cite[Chp.~V, Proposition~4.8(a)]{Ha1977Algebraic-geometry}. However,
	since $\overline{X} \setminus X$ has at most $3$ irreducible components $\Pic(X)$ is not trivial,
	so $X$ cannot be isomorphic to $\A^2$.
	
	\ref{eq2}: There exists a point in $\overline{X}\subset \p^3$ having multiplicity $\ge d-1$: this is clear if $d\le 2$ and follows from \ref{eq1} if $d=3$. Applying an affine automorphism of $\A^3$, we can assume that this point is $[0:1:0:0]$, and the result then follows from Lemma~\ref{lem.hypersxpq}.
\end{proof}
\begin{remark}
Corollary~\ref{cor.cubic_surf_iso_to_A2_singular}\ref{eq1} is also true for $d \geq 4$: 
If $\bar{X}$ is smooth,
then it is a K3-surface in case $d = 4$ and of general type in case $d > 4$. In both situations 
$\overline{X}$ is not rational. 

Corollary~\ref{cor.cubic_surf_iso_to_A2_singular}\ref{eq2} is false for $d \geq 4$: Consider
the hypersurface $X$ in $\AA^3$ which is given by $f := z + (x+ yz)^2\cdot y^{d-4} = 0$.
Note that $X$ is isomorphic to $\AA^2$, since $f$ is the first component of 
the composition $\varphi_2 \circ \varphi_1$ of the automorphisms
\[
	\begin{array}{rcl}
		\AA^3 &\stackrel{\varphi_1}{\longrightarrow}& \AA^3 \\
		(x, y, z) &\longmapsto& (x + yz, y, z)
	\end{array}
	\quad
	\textrm{and}
	\quad
	\begin{array}{rcl}
		\AA^3 &\stackrel{\varphi_2}{\longrightarrow}& \AA^3 \\
		(x, y, z) &\longmapsto& (x, y, z + x^2y^{d-4}) \, .
	\end{array}
\]
Note that the closure $\overline{X}$ in $\PP^3$ is singular only along the lines $w=y=0$ and $w=z=0$
and that the multiplicity at each of these points is $\leq d-2$. In particular, by Lemma~\ref{lem.hypersxpq}
there is no affine coordinate change of $\AA^3$ such that $X$ is given by $xp+ q = 0$ for 
$p, q \in \k[y, z]$.
\end{remark}

\subsection{Hypersurfaces of $\AA^3$ of degree $1$ in one variable}
Motivated by Corollary~\ref{cor.cubic_surf_iso_to_A2_singular}, this section is devoted to the study of hypersurfaces $X\subset \A^3$ given by 
\[
	xp(y,z)+q(y,z) = 0
\] 
for some polynomials $p,q\in \k[y,z]$ where $p \neq 0$. We start with the following result which is due to Russell~\cite[Theorem~2.3]{Russell1976}
\begin{proposition}\label{Prop:XxpqisA2thendisjoint}
	Let $p,q\in \k[y,z]$ be such that 
	\[X=\Spec(\k[x,y,z]/(xp+q))\]
	is isomorphic to $\A^2$ and such that $p\not\in \k$. 
	Then there is an automorphism of $\k[y,z]$ that sends $p$ onto an element of $\k[y]$.
	In particular,  the irreducible components of the scheme
	$\Spec(\k[y,z]/(p))$ are disjoint and isomorphic to $\A^1$.
\end{proposition}

By Proposition~\ref{Prop:XxpqisA2thendisjoint} 
we are led to study the case of hypersurfaces in $\AA^3$ of the form
$xp(y) + q(y, z)$. This is done in the next result.

\begin{proposition}\label{Prop:XfibratReminder}
	Let $p\in \k[y]\setminus \k$, $q\in \k[y,z]$ and consider the polynomial 
	\[
		f=xp(y)+q(y,z)\in \k[x,y,z] \, .
	\]
	Write $\tilde{p}=\prod_{i=1}^r (y-a_i)$ where $a_1,\ldots,a_r\in \k$ are the $r$ distinct roots of $p$. Then
	the following statements are equivalent:
	\begin{enumerate}[leftmargin=*]
	\item\label{XfibratReminderXisoA2}
	$X=\Spec(\k[x,y,z]/(f))$ is isomorphic to $\A^2$;
	\item\label{XfibratReminderPvariable}
	There exists $\varphi\in \Aut_\k(\k[x,y,z])$ such that $\varphi(x)=f$ and $\varphi(y)=y$; 
	\item\label{XfibratReminderTechnical}
	There exist $a\in \k[y,z], r_0,r_1\in \k[y]$ with $\deg(r_i)<r$ for $i=0,1$ such that
	$r_1(a_i)\not=0$ for each $i\in \{1,\ldots,r\}$ and
	\[
		q(y,z)=a\tilde{p}+ zr_1+r_0 \, .
	\]
	\end{enumerate}
\end{proposition}
\begin{proof}
$\ref{XfibratReminderXisoA2}\Rightarrow \ref{XfibratReminderPvariable}$:
	This is done in~\cite[Theorem~2.3]{Russell1976}, see also~\cite{Sathaye1976} for the case $\car(\k) = 0$.
	
$\ref{XfibratReminderPvariable}\Rightarrow \ref{XfibratReminderXisoA2}$: The automorphism $\varphi$ corresponds to an automorphism of $\A^3$ that sends $X$ onto $\Spec(\k[x,y,z]/(x))\simeq \A^2$.	
	
$\ref{XfibratReminderXisoA2}\Rightarrow \ref{XfibratReminderTechnical}$:
	We consider the morphism $\pi\colon X\to \A^1$ given by $(x,y,z)\mapsto y$. Then, outside of $\{a_1,\ldots,a_r\}$, $\pi$
	is a trivial $\A^1$-bundle. If $X$ is isomorphic to $\AA^2$, then
	each fibre of $\pi$ needs to be isomorphic to $\A^1$ (this follows for instance from \cite[Theorem~4.12]{Ganong2011}). We write $q(y,z)=\sum_{j=0}^d z^j (q_j\tilde{p}+r_j),$ with $q_j,r_j\in \k[y]$ and $\deg(r_j)<r=\deg(\tilde{p})$ for each $j$.
	
	For each $i\in \{1,\ldots,r\}$, the fibre of $\pi$ over $a_i$ is $\Spec(\k[x,z]/(q(a_i,z)))$, so $q(a_i,z)$ is a polynomial of degree $1$ in $z$ (as each fibre of $\pi$ is isomorphic to $\AA^1$). 
	This implies that $r_j(a_i)=0$ for each $j\ge 2$ and that $r_1(a_i)\not=0$. As $\deg(r_j)<r$, we obtain that $r_j=0$ for $j\ge 2$.
	This gives \ref{XfibratReminderTechnical} with $a=\sum_{j=0}^d z^j q_j$.
	
	$\ref{XfibratReminderTechnical}\Rightarrow \ref{XfibratReminderXisoA2}$:
Let $R=\k[x,y,z]/(f)$ be the ring of regular functions on $X$. For each $i\in \{1,\ldots,r\}$, Assertion~\ref{XfibratReminderTechnical} gives $f(x,a_i,z)=q(a_i,z)=zr_1(a_i)+r_0(a_i)$, so $R/(y-a_i)\simeq \k[\A^1]$, which implies that $(y-a_i)$ is a prime ideal of $R$ and that $\pi^{-1}(a_i)=X\cap \{y=a_i\}$ is isomorphic to $\A^1$. Hence, every (closed) fibre of $\pi$ is isomorphic to $\A^1$.
	
	We consider $h_0=z$ and construct inductively a finite sequence $h_0,h_1,\ldots,h_{N_1}$ of regular functions on $X$ such that $(\pi, h_i)\colon X\to \A^2$ restricts to an isomorphism $\pi^{-1}(U)\iso U\times \A^1$, where $U = \AA^1 \setminus \{ a_1, \ldots, a_r \}$.
	
	If $h_i$ is constant on $\pi^{-1}(a_1)$, 
	then there is a $c_i\in \k$ such that $h_i-c_i$ is a multiple of $y-a_1$. We then choose $h_{i+1} \in R$ such that $h_i-c_i = (y-a_1)\cdot h_{i+1}$. 
	This sequence ends up at some point, i.e.~that there exists $N_1\ge 0$ such that $h_{N_1}$ is not constant on $\pi^{-1}(a_1)$. Indeed, this is a direct application of \cite[Lemma~1.1]{KaWr1985Flat-families-of-a} 
	where we use that $R$ is a Noetherian integral domain.
	
	Now, we start with $h_{N_1} \in R$. With the same argument as above, there exists now
	$h_{N_2} \in R$ that is not constant on $\pi^{-1}(a_1)$, not constant on $\pi^{-1}(a_2)$
	and $(\pi, h_{N_2})$ restricts to an isomorphism $\pi^{-1}(U)\iso U\times \A^1$. Proceeding
	the same way with $i =3, \ldots, r$ we find $h \in R$ that is not constant on each
	$\pi^{-1}(a_j)$ for $j =1, \ldots, r$ and such that $(\pi, h)$ restricts to an isomorphism
	$\pi^{-1}(U)\iso U\times \A^1$. 
	 
	We observe that $(\pi,h)\colon X\to \AA^2$ is birational, quasi-finite and surjective.
	By Zariski's Main Theorem~\cite[Corollaire~(4.4.9)]{Gr1961Elements-de-geomet} 
	it is thus an isomorphism.
\end{proof}

Remark that the implication $\ref{XfibratReminderTechnical}\Rightarrow \ref{XfibratReminderXisoA2}$ of Proposition~\ref{Prop:XfibratReminder} also follows from~\cite[Lemma~3.10]{BlSa2019Embeddings-of-affi} (the argument is essentially due to Asanuma~\cite[Corollary 3.2]{As1987Polynomial-fibre-r}), but the argument given above is much simpler and goes back to \cite{KaWr1985Flat-families-of-a}.

\begin{corollary}\label{coro:TrivialisationVariables}
Let $f\in \k[x,y,z]$ be a polynomial of degree $\le 3$. Then $f$ is a variable of $\k[x,y,z]$ if and only if $\Spec(\k[x,y,z]/(f))\simeq \A^2$. In particular, if this holds, then $\Spec(\k[x,y,z]/(f-\lambda))\simeq \A^2$ for each $\lambda\in \k$.
\end{corollary}
\begin{proof}
If $f$ is a variable of $\k[x,y,z]$, then $\Spec(\k[x,y,z]/(f-\lambda))\simeq \A^2$ for each $\lambda\in \k$, and thus in particular for $\lambda=0$. Conversely, we suppose that $\Spec(\k[x,y,z]/(f))$ is isomorphic to $\A^2$, and prove that $f$ is a variable.

After an affine coordinate change we may assume $f = x p(y) + q(y, z)$ with
$p \in \k[y] \setminus \{0\}$ and $q \in \k[y, z]$ (Proposition~\ref{Proposition:XxpqisA2smalldegreelistpqDeg3}). If $p\in \k^*$, then $f$ is a variable as $(f,y,z)\in \Aut(\A^3)$. If $p\in \k[y]\setminus \k$, then Proposition~\ref{Prop:XfibratReminder}\ref{XfibratReminderPvariable} implies that
$f$ is a variable.
\end{proof}
\subsection{Hypersurfaces of $\AA^3$ of small degree that are isomorphic to $\AA^2$}

\begin{lemma}\label{Lemm:EmbeddingA2degree}
Let $p,q\in \k[t]$ be two polynomials such that 
\[
	\k[t]=\k[p,q] \text{ and }\deg(p)<\deg(q).
\] 
Then, either $1\in \{\deg(p),\deg(q)\}$ or $2\le \deg(p)\le \deg(q)-2$.
\end{lemma}
\begin{proof}
Suppose first that $\deg(p)\le 0$, which is equivalent to $p\in \k$. We obtain $\k[t]=\k[q]$, which implies that $\deg(q)=1$. Indeed, $\deg(q)\ge 1$ since $q\not\in \k$ and $\deg(q)>1$ is impossible, 
as the degree of any element of $\k[q]$ is a multiple of $\deg(q)$.

If $\deg(p)=1$, the result holds, so we may assume that $\deg(p)\ge 2$. It remains to see that $\deg(p)<\deg(q)-1$.
We then consider the closed embedding 
$f\colon \A^1 \hookrightarrow \A^2$ given by $t\mapsto (p(t),q(t))$, which extends to a morphism 
$\hat{f}\colon \p^1\to \p^2$ given by $[t:u]\mapsto [u^d:P(t,u):Q(t,u)]$, where $d=\deg(q)$ and where 
$P(t,u)=u^d\cdot p(\frac{t}{u})$, $Q(t,u)=u^d \cdot q(\frac{t}{u})$ are homogeneous polynomials of degree $d$. The image $\Gamma=\hat{f}(\p^1)$ is a closed curve of $\p^2$ that is rational and smooth outside of $[0:0:1]=\hat{f}([1:0])$. The degree of $\Gamma$ is the intersection of $\Gamma$ with a general line, which is then equal to $d=\deg(q)\ge 3$. The multiplicity $m$ of $\Gamma$ at the point $[0:0:1]$ satisfies then $m>1$, as a smooth curve of degree $d\ge 3$ has genus $\frac{(d-1)(d-2)}{2} \geq 1$. It remains to observe that $m=\deg(q)-\deg(p)$. This can be checked in coordinates, or simply seen geometrically: a general line of $\p^2$ passing through $[0:0:1]$ intersects the curve $\Gamma \setminus \{ [0:0:1] \}$ 
in $\deg(q)-m$ points and these points correspond to the roots of $p-\lambda$ for some general $\lambda$.
\end{proof}

\begin{corollary}\label{Coro:A1inA2smalldegree}
Let $C\subset \A^2=\Spec(\k[x,y])$ be a closed curve isomorphic to $\A^1$, of degree $\le 3$. Then, 
up to applying an element of $\Aff(\A^2)$, 
the curve $C$ is given by
$x+p(y) = 0$ for some $p\in \k[y]$ of degree $\le 3$ with no constant or linear term.
\end{corollary}
\begin{proof}
Let $p,q\in \k[t]$ be such that $t\mapsto (p(t),q(t))$ is an isomorphism $\A^1\to C$ defined over $\k$. The polynomials $p,q$ satisfy then $\k[p,q]=\k[t]$. After applying an affine automorphism of $\A^2$, we may assume that $\deg(p)<\deg(q)$. By Lemma~\ref{Lemm:EmbeddingA2degree}, we obtain $1\in \{\deg(p),\deg(q)\}$. 

We first assume that $\deg(q)=1$, which implies that $\deg(p)<1$, so $p\in \k$. After applying an affine automorphism of $\A^2$, we get $p=0$ and $q=t$, so the curve $C$ is given by $x=0$.

We then assume that $\deg(p)=1$. After applying an automorphism of $\A^1$, we may assume that $p=t$. Hence, $C$ is
given by $y-q(x)=0$. After applying the automorphism $(x,y)\mapsto (y,x)$, the equation is $x-q(y) = 0$. By using
an automorphism of the form 
$(x,y)\mapsto (x+ay+b,y)$ for some $a,b\in \k$, we may assume that $q$ has no constant or linear term.
\end{proof}

\begin{lemma}\label{Lemm:XxpqisA2smalldegreelistp}
Let $f \in \k[x,y,z]$ be a polynomial of the form
\[
	f=xp(y, z) + q(y, z) \, ,
\]
for some $p,q\in \k[y,z]$ with $p\not=0$ and $\deg(p)\le 3$. If the surface $\Spec(\k[x,y,z]/(f))$ is isomorphic to $\A^2$, then after applying an affine automorphism on $y$ and $z$, one of the following cases hold:
\begin{enumerate}[leftmargin=*]
\item\label{FirstCaseky_}
$p\in \k[y]$ has degree $\le 3$;
\item
$p=y+r(z)$ for some $r\in \k[z]$ of degree $2$ or $3$.
\end{enumerate}
\end{lemma}
\begin{proof}
If $p\in \k$, then we are in case \ref{FirstCaseky_}. We may thus assume that $p\not\in \k$.
By Proposition~\ref{Prop:XxpqisA2thendisjoint}, the irreducible components of
$F_p=\Spec(\k[y,z]/(p))$ are disjoint and isomorphic to $\A^1$.

We use the embedding $\A^2\hookrightarrow \p^2$, $(y,z)\mapsto [1:y:z]$ and denote by $L_\infty = \p^2\setminus \A^2$
the line at infinity.

If the irreducible components of $F_p$ are lines, then their closures in $\PP^2$ have
to pass through the same point in $L_\infty$. After applying an affine automorphism, we may assume that the point is $[0:0:1]$, which implies that $p\in \k[y]$.

It remains to study the case where at least one irreducible component has degree $\ge 2$. This component
corresponds to an irreducible curve $C\subset \A^2$ of degree $d\in \{2,3\}$ whose closure $\overline{C}$ in $\p^2$ is again an irreducible curve of degree $d$.

By Corollary~\ref{Coro:A1inA2smalldegree}, we may apply an affine automorphism and assume that $C$ is the zero locus of $y+r(z)$ for some polynomial $r$ of degree $d$. If $F_p$ is equal to $C$, then $p=y+r(z)$ (up to some constant which can be removed by an affine automorphism). Otherwise, as $F_p$ has degree $\le 3$, we get that $F_p$ is reduced,
and it is the disjoint union of the degree $2$ curve $C$ with some line. But there is no such line in $\A^2$: by B\'ezout's theorem, the closure of the line in $\p^2$ would be tangent to the conic $\overline{C}$
at the point at infinity of $\overline{C}$, impossible as already $L_\infty$ is tangent to $\overline{C}$ at that point.
\end{proof}

\begin{proposition} 
\label{Proposition:XxpqisA2smalldegreelistpq}
Let $f\in \k[x,y,z]$ be a polynomial of degree $\le 3$ of the form
\[
	f=xp(y, z) + q(y, z) \, ,
\]
for some $p,q\in \k[y,z]$. If the surface $\Spec(\k[x,y,z]/(f))$ is isomorphic to $\A^2$, then after 
applying an affine automorphism that fixes the point $[0:1:0:0]$, one of the following cases occurs:

	\begin{enumerate}[leftmargin=*]
	\item \label{pqCaseOnlyyz}
	 $f=y+s(z)$ for some polynomial $s \in \k[z]$ of degree $\le 3$;
		\item \label{pqPofdegree2}
		$f=x(y+z^2)+z$;
		\item \label{pqFirstCasekyDeg3}
		$f=x+r_2(y, z)+r_3(y, z)$ for some homogeneous $r_i \in \k[y,z]$ of degree $i$;
		\item \label{pqSecondCasexyDeg3}
		$f=xy+yr_2(y,z)+z$ for a homogeneous polynomial $r_2\in \k[y,z]$ of degree $2$;
		\item \label{pqSecondCasexy2Deg3}
		$f=xy^2+y s(z)+z$ for a polynomial $s\in \k[z]$ of degree $\le 2$;
		\item \label{pqTwoyDeg3}
		$f=xy(y+1)+s(y)z+t(y)$ for some polynomials $s,t\in \k[y]$ of degree $\le 1$ with $s(0)s(-1)\not=0$.
	\end{enumerate}
\end{proposition}
\begin{proof}
If $p=0$, then $f=q\in \k[y,z]$, so $\Spec(\k[x,y,z]/(f))=\A^1 \times \Spec(\k[y,z]/(f))$, which implies that $\Spec(\k[y,z]/(f))\simeq \A^1$. By Corollary~\ref{Coro:A1inA2smalldegree}, we may apply an affine automorphism on $y$ and $z$ in order to be in case~\ref{pqCaseOnlyyz}. We may thus assume in the sequel that $p\not=0$.

According to Lemma~\ref{Lemm:XxpqisA2smalldegreelistp}, we only need to consider the following two cases: either $p\in \k[y]$ or $p=y+r(z)$ for some $r\in \k[z]$ of degree $2$.

Suppose first that $p=y+r(z)$ for some $r\in \k[z]$ of degree $2$. By using the (non-affine) automorphism $(x,y,z)\mapsto (x,y-r(z),z)$ of $\AA^3$, we get
\[
	\Spec(\k[x,y,z]/(f))\simeq \Spec(\k[x,y,z]/(xy+q(y-r(z),z)).
\]
Then, Proposition~\ref{Prop:XfibratReminder} shows that $q(y-r(z),z)=ay+\lambda z+\mu$ for some $a\in \k[y,z]$, $\lambda \in \k^*$ and $\mu\in \k$. This gives 
\[
	f=xp + q=(x+s)(y+r)+\lambda z +\mu,
\]
where $s=a(y+r,z)\in \k[y,z]$. As $\deg(r)=2$, we obtain that $\deg(s)\le 1$. 
Hence, after applying the affine automorphism $(x, y, z) \mapsto (x-s(y, z), y, z)$,
we may assume that $f$ is equal to $x(y+r(z)) + \lambda z+\mu$. Using the affine automorphism 
$(x,y,z)\mapsto (x,y,\lambda^{-1}(z-\mu))$, we obtain $x(y+r'(z)) + z$ for some $r'=\sum_{i=0}^2 \mu_i z^2\in \k[z]$ of degree $2$.
After replacing $y$ with $y-\mu_0-\mu_1 z$ we get $x(y+\mu_2 z^2)+z$. We then apply $(x,y,z)\mapsto (\mu_2^{-1} x,\mu_2 y,z)$ 
in order to be in case~\ref{pqPofdegree2}.

It remains to consider the case where $p\in \k[y]$. We distinguish the different cases:

If $p\in \k^*$, we may assume that $p=1$ and after applying $(x, y, z) \mapsto (x - q_0 - q_1(y, z))$ we are in case~\ref{pqFirstCasekyDeg3}, where $q_0, q_1 \in \k[y, z]$ are the constant and linear part of $q$, respectively.

If $p$ has one single root, we may assume that $p=y^i$ for some $i\in \{1,2\}$. Then, Proposition~\ref{Prop:XfibratReminder} shows that $q(y,z)=ay+\lambda z+\mu$ for some $a\in \k[y,z]$, $\lambda\in \k^*$ and $\mu\in \k$. 
After applying the affine automorphism $(x,y,z)\mapsto (x,y,\lambda^{-1}(z-\mu))$ we may assume that $\lambda=1$ and $\mu=0$. 

If $i = 1$, then $f=xy+yr(y,z)+z$ for some polynomial $r$ of degree $\leq 2$.
		Let $r_1, r_0 \in \k[y, z]$ be the homogeneous parts of degree $1$ and degree $0$ of $r$, respectively.
		We may apply the affine automorphism $(x, y, z) \mapsto (x-r_1(y, z) -r_0, y, z)$
		and thus we may assume that $r$ is homogeneous of degree $2$. Hence, we are in case~\ref{pqSecondCasexyDeg3}.
		
		 If $i = 2$, then $f=xy^2+yr(y,z)+z$ for some polynomial $r$ of degree $\leq 2$. Now, after
		applying a suitable affine automorphism of the form $(x, y, z) \mapsto (x-b(y, z), y, z)$ we may assume
		that $r \in \k[z]$ and thus we are in case~\ref{pqSecondCasexy2Deg3}. 
	
We then assume that $p$ has two distinct roots. We may assume that $p=y(y+1)$. Proposition~\ref{Prop:XfibratReminder} shows that $q(y,z)=ay(y+1)+s z+ t$ for some $a\in \k[y,z]$ of degree $\leq 1$, and some $s, t \in \k[y]$ of degree $\leq 1$ with $s(0) \neq 0$, $s(-1) \neq 0$. After applying $(x, y, z) \mapsto (x-a(y, z), y, z)$ we are in case~\ref{pqTwoyDeg3}.
\end{proof}

\begin{proposition}[Hypersurfaces isomorphic to $\AA^2$ of degree $\leq 3$]
	\label{Proposition:XxpqisA2smalldegreelistpqDeg3}
	Let $f \in \k[x, y, z]$ be an irreducible polynomial of degree $\leq 3$.
	If the surface $\Spec(\k[x, y, z] / (f))$ is isomorphic to $\A^2$, then 
	there is $\alpha \in \Aff(\AA^3)$, such that one of the following cases occur:
	\begin{enumerate}[leftmargin=*, label=\Alph*$)$]
		\item \label{FirstCasekyDeg3}
		$\alpha^\ast(f)=x+r_2(y, z)+r_3(y, z)$ for some homogeneous $r_i \in \k[y,z]$ of degree $i$;
		\item \label{SecondCasexyDeg3}
		$\alpha^\ast(f)=xy+yr_2(y,z)+z$ for a homogeneous $r_2\in \k[y,z]\setminus \k[y]$ of degree $2$;
		\item \label{SecondCasexy2Deg3}
		$\alpha^\ast(f)=xy^2+y (z^2+az+b)+z$ for some $a,b\in \k$.
	\end{enumerate}
	Moreover, if $f \in \k[x, y, z]$ is one of the polynomials
	from cases~\ref{pqFirstCasekyDeg3}-\ref{pqTwoyDeg3} 
	of Proposition~$\ref{Proposition:XxpqisA2smalldegreelistpq}$, then
	we may in addition assume that $\alpha^\ast(y) \in \k[y]$.
\end{proposition}

\begin{proof}By Corollary~\ref{cor.cubic_surf_iso_to_A2_singular} we may assume that
	\[
		f=xp +q
	\]
	for some $p,q\in \k[y,z]$ with $\deg(p) \leq 2$ and $\deg(q) \leq 3$.
	We go through the different cases of Proposition~\ref{Proposition:XxpqisA2smalldegreelistpq}.
	
	\ref{pqCaseOnlyyz}: We exchange $x$, $y$ and get $f=x+s(z)$ and then we replace $x$ with 
	$x+a+bz$ for some $a, b \in \k$ in order to be in case~\ref{FirstCasekyDeg3}.
	
 	\ref{pqPofdegree2}: We exchange $x$, $y$ and get $f=y(x+z^2)+z=xy+yz^2+z$ which is a subcase of~\ref{SecondCasexyDeg3}.
	
	\ref{pqFirstCasekyDeg3} and \ref{pqSecondCasexyDeg3} directly give \ref{FirstCasekyDeg3} and \ref{SecondCasexyDeg3}, except if we are in case~\ref{pqSecondCasexyDeg3} with $r_2 \in \k[y]$, 
	in which case we exchange $x$, $z$ in order to be in case~\ref{FirstCasekyDeg3}.
	
	\ref{pqSecondCasexy2Deg3}: 
	We have $f=xy^2+y s(z)+z$ for some polynomial $s$ of degree $\le 2$. We distinguish three cases:
	
	If $\deg(s)\le 0$, we have $s\in \k$. After the coordinate change $(x, y, z) \mapsto (x, y, z- sy)$ and the exchange of $x$, $z$ we are in case~\ref{FirstCasekyDeg3}. 
	
	If $\deg(s)=1$, we have $f=xy^2+y(az+b)+z$ for some $a\in \k^*$ and $b\in \k$. We replace $x,y,z$ with $a(az+b),(y-1)/a,x$ and obtain $xy+yr_2(y, z)+z$ where 
	$r_2 = yz + uy + vz + w$ for some $u,v,w \in \k$. 
	After replacing $x$ with $x-uy- vz-w$, 
	we may assume that $r_2$ is homogeneous and still not in $\k[y]$; this gives \ref{SecondCasexyDeg3}.
		
	If $\deg(s)=2$ we apply a homothety in $x$ and $y$, and obtain \ref{SecondCasexy2Deg3}.
	
	\ref{pqTwoyDeg3}: We exchange $x$ and $z$ and get $f=xs(y)+y(y+1)z+t(y)$ for some polynomials $s,t\in \k[y]$ of degree $\le 1$ with $s(0)s(-1)\not=0$. If $s\in \k$, then $s \neq 0$ and after
	applying $(x, y, z) \mapsto (s^{-1}(x-t(y)), y, z)$ we are in case~\ref{FirstCasekyDeg3}. 
	Otherwise, we replace $s(y)$ with $y$ and get $xy+u(y)z+v(y)$ where $u, v \in \k[y]$, $\deg(u) = 2$, $\deg(v) \leq 1$
	and $u(0)\not=0$. Hence, we get $xy+ya(y,z)+\lambda z+\mu$ with $a\in \k[y,z]$, $\lambda\in \k^*$ and $\mu\in \k$.
	After replacing $\lambda z+\mu$ with $z$ we get $f = xy + y b(y, z) + z$ for some $b \in \k[y, z]$.
	When we write $b$ as $b_0 + b_1 + b_2$, where each $b_i \in \k[y,z]$ is homogeneous of degree $i$, we may replace $x$ with $x-b_0-b_1$ and obtain \ref{SecondCasexyDeg3}, except when $b_2 \in \k[y]$:
	then we exchange $x$ and $z$ in order to be in case~\ref{FirstCasekyDeg3}.
	
	Moreover, in cases~\ref{pqFirstCasekyDeg3}-\ref{pqTwoyDeg3} we see that the constructed affine 
	coordinate change maps $\k[y]$ onto itself. This shows the last statement.
	\end{proof}

	In the next corollary, we list several properties of the closure in $\PP^3$ of a 
	hypersurface in $\AA^3$ of degree $3$ which is isomorphic to $\AA^2$.
		
	\begin{corollary}
		\label{cor.Properties_of_hypersurfaces_at_infinity}
		Let $f \in \k[x, y, z]$ be a polynomial of degree $3$ such that 
		$X = \Spec(\k[x, y, z] / (f)) \simeq \AA^2$ and write $f = f_0 + f_1 + f_2 + f_3$
		where $f_i \in \k[x, y, z]$ is homogeneous of degree $i$.
		\begin{enumerate}[leftmargin=*]
			\item \label{cor.PHI_Conic_tangent} 
					If $f_3$ defines a conic $\Gamma$ and a tangent line $L$ in $\PP^2$,
					then the singular locus of $\overline{X} \subset \PP^3$ equals the point
					$(\Gamma \cap L)_{\textrm{red}}$. 
			\item \label{cor.PHI_line} 
					If $f_3$ defines one line $($with multiplicity $3)$ in $\PP^2$, then $f_2$ is either zero or defines 
					some lines in $\PP^2$ and all the lines given by $f_3$ and $f_2$ have a point
					in $\PP^2$ in common. 
					Moreover, the singular locus of $\overline{X} \subset \PP^3$ is given by $w = f_2 = f_3 = 0$.
			\item \label{cor.PHI_distinct_lines} If $f_3$ neither defines
					a conic and a tangent line in $\PP^2$, nor one line in $\PP^2$, 
					then $f_3$ defines several lines
					in $\PP^2$ and all these lines pass through
					the same point $q \in \PP^2$. Moreover, $q$ lies in the singular
					locus of $\overline{X} \subset \PP^3$.
		\end{enumerate}
	\end{corollary}
	\begin{proof}Applying an affine automorphism, we are in one of the three cases \ref{FirstCasekyDeg3}-\ref{SecondCasexyDeg3}-\ref{SecondCasexy2Deg3} of Proposition~\ref{Proposition:XxpqisA2smalldegreelistpqDeg3}. 
	The affine automorphism induces an automorphisms 
	of the plane at infinity and thus an isomorphism between the curve in $\p^2$ 
	given by $f_3 = 0$ and respectively $r_3(y,z)=0$, $yr_2(y,z)=0$ and $y(xy+z^2)=0$
	where $r_i \in \k[y, z]$ is homogeneous of degree $i$ for $i = 1, 2$.
	We thus obtain two cases for $f_3=0$, namely a conic and a tangent line \ref{cor.PHI_Conic_tangent}, or a set of lines through the same point: \ref{cor.PHI_line}-\ref{cor.PHI_distinct_lines}. The distinction between \ref{cor.PHI_line} and \ref{cor.PHI_distinct_lines} corresponds to ask whether the lines are all the same or not. We study the three cases separately.
	
		\ref{cor.PHI_Conic_tangent}: Here we are in Case \ref{SecondCasexy2Deg3} of Proposition~\ref{Proposition:XxpqisA2smalldegreelistpqDeg3}. There exist thus $\psi \in \Aff(\AA^3)$
		and $a, b \in \k$ with $f = \psi^\ast(g)$ where $g = xy^2 + y(z^2 + az + b) + z$. 
		Let $G \in \k[w, x, y, z]$ be the homogeneous 
		polynomial of degree $3$ such that $G(1, x, y, z) = g$. The gradient of $G$ 
		\begin{align*}
			&\left(\frac{\partial G}{\partial w}, \frac{\partial G}{\partial x}, \frac{\partial G}{\partial y}, \frac{\partial G}{\partial z}\right) \\
			&= \left(y(az + 2bw) + 2zw, y^2, 2xy + z^2 + azw + bw^2, y(2z + aw) + w^2\right)
		\end{align*}
		is equal to zero if and only if $w=y=z=0$ and thus $[0:1:0:0]$ is the only singularity
		of the hypersurface $G=0$ in $\PP^3$. 
		

		\ref{cor.PHI_line}: Here we are in Case \ref{FirstCasekyDeg3} of Proposition~\ref{Proposition:XxpqisA2smalldegreelistpqDeg3}. There exist thus $\psi \in \Aff(\AA^3)$ and a homogeneous $r_2 \in \k[y, z]$ of degree $2$ such that
		$f = \psi^\ast(h)$ where $h = x + r_2(y, z) + y^3$. 
		Let $\varphi \in \GL_3(\k)$ be the linear part of $\psi$.
		Then $f_3 = \varphi^\ast(y)^3$ and 
		$f_2 = r_2(\varphi^\ast(y), \varphi^\ast(z)) + 3 \delta \varphi^\ast(y)^2$ where 
		$\psi^\ast(y) = \varphi^\ast(y) + \delta$. Thus $f_2, f_3 \in \k[s, t]$ for 
		$s = \varphi^\ast(y)$, $t = \varphi^\ast(z)$ and the first claim follows.
		Let $H \in \k[w, x, y, z]$ such that $H(1, x, y, z) = h$. The gradient
		of $H$
		\[
			\left(\frac{\partial H}{\partial w}, \frac{\partial H}{\partial x}, \frac{\partial H}{\partial y}, \frac{\partial H}{\partial z}\right) = \left(2xw + r_2(y, z), w^2, w \frac{\partial r_2}{\partial y}(y, z) + 3y^2,
											 w \frac{\partial r_2}{\partial z}(y, z)\right)
		\]
		is equal to zero if and only if
		\[
			\left\{
				\begin{array}{rl}
					w = y = r_2(y,z) = 0 & \textrm{if $\car(\k) \neq 3$} \\
					w = r_2(y, z) = 0 & \textrm{if $\car(\k) = 3$}
				\end{array}
			\right. \, .
		\]
		Since the intersection of $H=0$ with the plane $w=0$ at infinity only consists of the line $w=y=0$, the singular locus of $H=0$ is equal to $w = y = r_2(y, z) = 0$
		(where $\k$ has any characteristic). 
		Note that this singular locus is mapped via $\psi^{-1}$
		onto $w = s = r_2(s, t) + 3 \delta s^2 = 0$ and thus the second claim follows.
		
		\ref{cor.PHI_distinct_lines}: The first claim
		directly follows from Proposition~\ref{Proposition:XxpqisA2smalldegreelistpqDeg3} and
		we may assume (after an affine automorphism) that $f$ is as in case~\ref{FirstCasekyDeg3} or
		in case~\ref{SecondCasexyDeg3}. In both cases the common intersection
		point of the lines defined by $f_3$ is $[0:1:0:0]$ which is a singularity of 
		$\overline{X} \subset \PP^3$ by Lemma~\ref{lem.hypersxpq}.
	\end{proof}
	
	\begin{corollary}\label{cor:Degree3onepointoneline}
	Let $f \in \k[x, y, z]$ be an irreducible polynomial of degree $3$ such that the hypersurface $X=V_{\AA^3}(f)$ is isomorphic to $\A^2$ and such that the closure of $X$ in $\p^3$ contains the line $w=y=0$. After applying an affine automorphism of $\AA^3$ that preserves the line $w=y=0$, we obtain one of the following cases:

	\begin{enumerate}[leftmargin=*, label=\alph*{$)$}]
		\item \label{LFirstCasekyDeg3}
		$f=x+r_2(y, z)+ys_2(y, z)$ for some homogeneous $r_2,s_2 \in \k[y,z]$ of degree $2$, with $s_2\not=0$;
		\item \label{LSecondCasexyDeg3y}
		$f=xy+yr_2(y,z)+z$ for a homogeneous $r_2\in \k[y,z]\setminus \k[y]$ of degree $2$;
		\item\label{LSecondCasexyDeg3z}
		$f=xz+yzr_1(y,z)+y+\delta z$ for some homogeneous $r_1\in \k[y,z]\setminus \{0\}$ of degree $1$ and $\delta\in \k$;
		\item \label{LSecondCasexy2Deg3}
		$f=xy^2+y (z^2 + az + b)+z$ for some $a, b \in \k$;
	\end{enumerate}
	\end{corollary}
	
\begin{proof}
There exists an affine automorphism that sends $f$ onto a $g\in \k[x,y,z]$ wich is one of the polynomials
from Proposition~\ref{Proposition:XxpqisA2smalldegreelistpqDeg3}. We then look at the image $\ell$ of the line $w = y = 0$ in the plane at infinity $H_\infty=\{[w:x:y:z]\in \p^3 \mid w=0\}$ and apply an affine automorphism to send it back to $w = y = 0$.

In case~\ref{FirstCasekyDeg3}, $g=x+r_2(y, z)+r_3(y, z)$ for some homogeneous $r_i \in \k[y,z]$ of degree~$i$. As $\deg(g)=3$, we get $r_3\not=0$, and the line $\ell$ is given by $p_1(y,z)=0$ for some homogeneous polynomial $p_1\in\k[y,z]$ of degree $1$ that divides $r_3$. We apply an element of $\GL_2(\k)$ acting on $y,z$ and obtain \ref{LFirstCasekyDeg3}.

In case~\ref{SecondCasexyDeg3}, $g=xy+yr_2(y,z)+z$ for a homogeneous polynomial $r_2\in \k[y,z]\setminus \k[y]$ of degree $2$. The line $\ell$ is given by $p_1(y,z)=0$ for some homogeneous polynomial $p_1\in\k[y,z]$ of degree $1$ that divides $yr_2(y,z)$. If $\ell$ is the line $y=0$ we get \ref{LSecondCasexyDeg3y}. 
Otherwise, the line is $\alpha y +\beta z$ with $\beta\not=0$ and $g=xy+y(\alpha y +\beta z)s_1(y,z)+z$ for some homogeneous degree $1$ polynomial $s_1\in \k[y,z] \setminus \{0\}$. We apply a linear coordinate change and send $\alpha y +\beta z$ and $y$ respectively to $y$ and $z$; this sends $z$ onto $\gamma y+\delta z$ with $\gamma\in \k^*$, $\delta\in \k$, and sends $g$ onto $xz+yzs_1'(y,z)+\gamma y+\delta z$ for some homogeneous degree $1$ polynomial 
$s_1'\in \k[y,z]\setminus \{0\}$. 
We replace $y$ with $\gamma^{-1}y$ and get~\ref{LSecondCasexyDeg3z}.

In case~\ref{SecondCasexy2Deg3}, $g=xy^2+y (z^2 + az + b)+z$ for some
$a, b, \in \k$ and thus the line $\ell$ is $y = 0$. Hence we obtain~\ref{LSecondCasexy2Deg3}.
\end{proof}

\begin{corollary}[Hypersurfaces isomorphic to $\AA^2$ of degree $2$]
	\label{Cor:Hypersurf_degree_2}
	Let $f \in \k[x, y, z]$ be an irreducible polynomial of degree $2$ 
	and assume that $X = \Spec(\k[x, y, z]/(f))$ is isomorphic to $\AA^2$. Then, after applying an affine automorphism, one
	of the following cases occur:
	\begin{enumerate}[$(1)$]
		\item \label{Cor:Hypersurf_degree_2_1} $f = x + y^2$;
		\item \label{Cor:Hypersurf_degree_2_2} $f = x + yz$.
	\end{enumerate}
\end{corollary}

\begin{proof}
	Since $f$ is of degree $2$, it follows from Proposition~\ref{Proposition:XxpqisA2smalldegreelistpqDeg3} that 
	$f$ is equal to $x + r_2(y, z)$ for a non-zero homogeneous polynomial of degree $2$ up to an affine automorphism.
	Depending whether $r_2(y, z) = 0$ has one ore two zeros in $\PP^1$ we are in 
	case~\ref{Cor:Hypersurf_degree_2_1} and case~\ref{Cor:Hypersurf_degree_2_2}, respectively.
\end{proof}

\section{Families of cubic hypersurfaces of $\AA^3$, all isomorphic to $\A^2$}
\label{sec.families_cubic_surf}
In this section, we study families of cubic hypersurfaces of $\AA^3$ that are isomorphic to $\AA^2$.
In order to to this we begin with linear systems on $\PP^2$.

\subsection{Linear systems on $\p^2$}
To study families of hypersurfaces of $\AA^3$, it is natural too look at the behaviour at infinity. 
In the following, for $d \geq 0$, we denote by $\k[x, y, z]_d$ 
the vector space of homogeneous polynomials of degree $d$ in $\k[x, y, z]$ 
and we consider it as an affine space (of dimension $d +2 \choose 2$).
In particular, $\k[x, y, z]_d$ carries the Zariski topology. Moreover, for any vector space $V$,
we let $\PP(V) = \textrm{Proj}_{\k}(\textrm{Sym}V^\ast)$ be the projectivisation of the symmetric algebra
$\textrm{Sym}V^\ast$ of the dual vector space $V^\ast$.

\begin{lemma}\label{Lemm:PencilLinearFactor}
Let $f,g\in \k[x,y,z]$ be two homogeneous polynomials of degree $d\ge 1$ without common factor. The following are equivalent:
\begin{enumerate}[leftmargin=*]
\item\label{AlwaysLinFact1}
The polynomial $\lambda f+\mu g$ is divisible by a linear factor, for all $\lambda,\mu\in \k$. 
\item\label{AlwaysLinFact2}
The polynomial $\lambda f+g$ is divisible by a linear factor, for infinitely many $\lambda\in \k$. 
\item\label{AlwaysLinFact3}
There are two linear polynomials $s,t\in \k[x,y, z]_1$ such that $f,g\in \k[s,t]$.
\end{enumerate}
\end{lemma}
\begin{proof}
Observe that the subset $R_d\subset\k[x,y,z]_d$ of elements that are divisible by a linear factor is closed. Indeed, $\p(R_d)$ is the image of the morphism $\p(\k[x,y,z]_1)\times \p(\k[x,y,z]_{d-1})$, $(p,q) \mapsto pq$. Hence, the set 
\[
	\set{ [\lambda:\mu]\in \p^1}{\lambda f+\mu g\text{ is divisible by a linear factor}}
\] 
is a closed subset of $\p^1$. Thus it is infinite if and only if it is the whole $\p^1$. This gives the equivalence $\ref{AlwaysLinFact1}\Leftrightarrow\ref{AlwaysLinFact2}$.

Let us prove $\ref{AlwaysLinFact3}\Rightarrow\ref{AlwaysLinFact1}$. As $f$ and $g$ have no common factor, $s, t$ are linearly independent. We apply a linear coordinate change 
and may assume that $s=x$ and $t=y$. Now, it is enough to remark that every 
homogeneous polynomial of $\k[x,y]$ is a product of linear factors.

It remains to prove $\ref{AlwaysLinFact1}\Rightarrow\ref{AlwaysLinFact3}$. We prove this by induction on $d = \deg(f) = \deg(g)$. The case where $d = 1$ 
holds by choosing $s=f$ and $t=g$.
We consider the dominant rational map $\eta\colon \p^2\dasharrow \p^1$, $[x:y:z]\mapsto [f(x,y,z):g(x,y,z)]$. If $\k(\frac{f}{g})$ is separably closed in $\k(\frac{x}{z},\frac{y}{z})$, then a general fibre of $\eta$ is irreducible \cite[Theorem 3.3.17, page~105]{FOV} (but not necessarily reduced). After
replacing $f,g$ with another basis of $\k f\oplus \k g$, we may thus assume that the zero locus of 
$f$ and $g$ are irreducible curves in $\PP^2$. 
The assumption \ref{AlwaysLinFact1} implies that two linear factors $s,t\in \k[x,y,z]$ exist such that $f=s^d$ and $g=t^d$. This gives $\ref{AlwaysLinFact3}$. If $\k(\frac{f}{g})$ is not separably closed in $\k(\frac{x}{z},\frac{y}{z})$, then there is a rational map 
$\frac{a}{b}$ (where $a,b\in \k[x,y,z]$ are homogeneous of the same degree without common factor) such that $\k(\frac{f}{g}) \subsetneq \k(\frac{a}{b})$ is a proper algebraic field 
extension, by the Primitive Element Theorem. 
Hence, we may decompose $\eta$ as $\eta=\nu\circ \eta'$, where $\nu\colon \p^1\to \p^1$ is a finite morphism 
which is not an isomorphism and $\eta'\colon \p^2\dasharrow \p^1$ is given by $[x:y:z]\mapsto [a(x,y,z):b(x,y,z)]$. Note that $\deg(a) = \deg(b) < d$, since $\nu$ is not an isomorphism.
As infinitely many fibres of $\eta$ contain lines, the same holds for $\eta'$, so \ref{AlwaysLinFact2} holds for $a$ and $b$. By induction, we find two homogeneous 
linear polynomials $s,t\in \k[x,y, z]$ such that $a,b\in \k[s,t]$ and hence $f,g\in \k[s,t]$ too.
\end{proof}

\begin{lemma}\label{LinearSystemLinearFactors}
Let $d\ge 2$ and let $V\subseteq \k[x,y,z]_d$ be a vector subspace such that the gcd of all elements of $V$ is $1$, and such that each element of $V$ is divisible by a linear factor. 
Then, one of the following holds:
\begin{enumerate}[leftmargin=*]
\item\label{Vlinearpencil}
There are two linear polynomials $s,t\in \k[x,y,z]_1$ such that $V\subseteq \k[s,t]$.
\item\label{Vlinearsquare_or_power3}
The degree $d$ is a power of $\car(\k)=p>0$, and $V=\k x^d\oplus \k y^d\oplus \k z^d$.
\end{enumerate}
\end{lemma}

\begin{proof}
Since the gcd of all elements in $V$ is $1$, we get $\dim V \geq 2$.
Suppose first that every element of $V$ is a $d$-th power in $\k[x, y, z]$. 
Then up to a linear coordinate change we may assume that $x^d, y^d \in V$. Since $x^d-y^d$ is a $d$-th power and is divisible by $x-y$, we get $x^d-y^d=(x-y)^d$. As $d\ge 2$, this implies that $\car(\k)=p>0$ and that $d$ is a power of $p$.
We get \ref{Vlinearpencil} if $V$ is generated by $x^d$ and $y^d$ and \ref{Vlinearsquare_or_power3} otherwise.

Suppose now that some element $f\in V$ is not a $d$-th power. By 
Lemma~\ref{Lemm:PencilLinearFactor}, we may apply a linear coordinate change and may assume that $f\in \k[x,y]$. For each element $g\in V$ that has no common factor with $f$, there exist two linear polynomials $s,t\in \k[x,y,z]_1$ such that $f,g\in \k[s,t]$ by Lemma~\ref{Lemm:PencilLinearFactor}. 
As $f\in \k[x,y]$ is not a power of an element of $\k[x, y, z]_1$ and as $f \in \k[x, y]$ is homogeneous,
there are linearly independent $p_1, q_1 \in \k[x, y]_1$ such that $f$ is divisible by the product $p_1 q_1$.
Since $\k[s, t]$ is factorially closed in $\k[x, y, z]$ and as $f \in \k[s, t]$, we get $p_1, q_1 \in \k[s, t]$
and thus $x, y \in \k[s, t]$, i.e. $\k[x, y] = \k[s, t]$. In particular, $g \in \k[x, y]$.
Since the set of elements $g \in V$ that have no common factor with $f$
is Zariski open in $V$, this set spans $V$ as a $\k$ vector space and so $V\subset \k[x,y]$.
\end{proof}

\begin{lemma}
	\label{lem.2nd_power}
	Assume that $\car(\k) = 2$ and let $g_1, \ldots, g_n \in \k[x, y]_2$,
	such that $\k g_1+\cdots +\k g_n = \k x^2 \oplus \k y^2$. If $s \geq 0$
	and $h_1, \ldots, h_n \in \k[x, y]_s$ are such that
	$\sum_i \lambda_i g_i$ and $\sum_i \lambda_i h_i$ have a common non-zero linear
	factor for all $(\lambda_1, \ldots, \lambda_n) \in \k^n$, then either $h_i = 0$ for all $i$
	or $s \geq 2$ and there exists $h \in \k[x, y]_{s-2} \setminus \{ 0 \}$ with
	$h_i = h g_i$ for all $i$.
\end{lemma}

\begin{proof}
	Note that $n\geq 2$. 
	After a linear coordinate change
	in $x, y$ and after replacing $h_1, \ldots, h_n$ and $g_1, \ldots, g_n$ with certain linear combinations
	we may assume that $g_1 = x^2$ and $g_i = y^2$ for all $i = 2, \ldots, n$. 
	For each $i\in \{2,\ldots,n\}$ and each
	$\alpha,\beta\in \k$, $(\alpha x + \beta y)^2 = \alpha^2 g_1 + \beta^2 g_i$ and 
	$\alpha^2 h_1 + \beta^2 h_i$ have a common non-zero linear factor, so $\alpha x + \beta y$ divides $\alpha^2 h_1 + \beta^2 h_i$, which means that $\alpha^2 h_1(\beta, \alpha)+\beta^2 h_i(\beta,\alpha)=0$. As this last equation is true for all $\alpha,\beta\in \k$, the polynomial $y^2h_1+x^2h_i$ is zero. We get a polynomial $\tilde{h}_i$ such that $h_1=\tilde{h}_ix^2$ and $h_i=\tilde{h}_iy^2$. The equality $h_1=\tilde{h}_ix^2$ yields that $\tilde{h}_i$ is independent of $i$, so writing $h=h_i$ gives the result.
\end{proof}

	\begin{lemma}
		\label{lem.almost_eigenvectors}
		Assume that $\car(\k) > 0$ and denote by $\phi \colon \PP^2 \to \PP^2$ the Frobenius endomorphism.
		\begin{enumerate}
		\item\label{lem.almost_eigenvectors1}
		For each $A\in \PGL_3$, there exists $v \in \PP^2$ such that $A \phi(v) = v$.
		\item\label{lem.almost_eigenvectors2}
		For each $B\in \PGL_3$, there exists $v \in \PP^2$ such that $\phi (B v) = v$.
		\end{enumerate}
	\end{lemma}
\begin{proof}
		 We denote by $\theta \colon \PGL_3	\to \PGL_3$ the endomorphism that sends
		 a matrix $C$ to the matrix obtained from $C$ by taking the $p$-th power of each entry.
			
		We will only prove \ref{lem.almost_eigenvectors1}, as \ref{lem.almost_eigenvectors2} follows from it by choosing $A = \theta(B)$. We then have to show that
		\[
		\Gamma = \set{A \in \PGL_3}{A\phi(v) = v \ \textrm{for some $v \in \PP^2$}}
		\] 
		is equal to $\PGL_3$. We consider 
		\[
		M = \set{(A,v) \in \PGL_3\times \p^2 }{A\phi(v) = v}
		\] 
		and obtain $\Gamma=\pi_1(M)$, where $\pi_1 \colon M \to \PGL_3$ is the first projection. As $\pi_1$ is proper,
		we get that $\Gamma$ is closed in $\PGL_3$ and thus we only have to show that $\dim \Gamma = 8$. We observe that the identity matrix $I\in \PGL_3$ belongs to $\Gamma$ and that $\pi_1^{-1}(I)=\p^2(\F_p)$ is finite. By Chevalley's Upper Semi-continuity Theorem for the dimension of fibres \cite[Corollaire 13.1.5]{Gr1966Elements-de-geomet}, the set 
		$\set{A \in \Gamma}{\dim \pi_1^{-1}(\{A\})\ge 1}$ is closed in $\Gamma$. It then suffices to show that $M$ is irreducible and of dimension $8$.

		To show this, we will prove that the second projection $\pi_2\colon M\to \p^2$ is a locally trivial $P$-bundle, where  $P$ is the parabolic subgroup of $\PGL_3$ that fixes $[1:0:0]$. Note that $\pi_2 \colon M \to \PP^2$ is $\PGL_3$-equivariant with 
		respect to the natural action on $\PP^2$ and the $\PGL_3$-action on $M$ given by
		$B \cdot (A, v) \coloneqq (B A \theta(B)^{-1}, Bv)$. We then only need to show that $\pi_2$ is a trivial $P$-bundle over the open subset $U=\set{[p_0:p_1:p_2]\in \p^2}{p_0\not=0}$. We
		consider the morphism $h \colon U \to \PGL_3$ given by
		\[
				[p_0:p_1:p_2]  \mapsto  
				\begin{pmatrix}
					p_0 & 0 & 0 \\
					p_1 & p_0 & 0 \\
					p_2 & 0 & p_0
				\end{pmatrix} \, ,
		\]
		which satisfies $h(p)([1:0:0])=p$ for each $p\in U$. We get a $V$-isomorphism
		\[
				\begin{array}{ccc}
				P \times V & \stackrel{\simeq}{\longrightarrow} & \pi_2^{-1}(V) \\
				 (A, p) & \longmapsto & (h(p) A \theta(h(p)^{-1}), p),
				\end{array}
		\]
		whose inverse sends $(A, p)$ onto $ (h(p)^{-1} A \theta(h(p)), p)$.
	\end{proof}

\subsection{Affine linear systems of affine spaces} It turns out that the following definition is very useful for us: 
\begin{definition}
	\label{def.aff_linear_system_of_aff_spaces}
	Let $f_1, \ldots, f_n \in \k[x_1, \ldots, x_d]$. We say that a morphism 
	\[
		\begin{array}{rcl}
			\AA^d &\longrightarrow& \AA^n \\
			(x_1, \ldots, x_d) &\longmapsto& (f_1(x_1, \ldots, x_d), \ldots, f_n(x_1, \ldots, x_d))
		\end{array}
	\]
	is an \emph{affine linear system of affine spaces} if for each $\lambda_0 \in \k$ and each
	$(\lambda_1, \ldots, \lambda_n) \in \k^n \setminus \{0\}$ the
	polynomial $\lambda_0 + \lambda_1 f_1 + \ldots + \lambda_n f_n$ is not constant
	and the corresponding hypersurface in $\AA^d$ is isomorphic to $\AA^{d-1}$. 
	This is equivalent
	to say that the preimage of every 
	affine linear hypersurface in $\AA^n$ 
	under the morphism $(f_1, \ldots, f_n) \colon \AA^d \to \AA^n$ is isomorphic to $\AA^{d-1}$.
	
	We call two affine linear systems of affine spaces 
	$(f_1, \ldots, f_n), (g_1, \ldots, g_n) \colon \AA^d \to \AA^n$ \emph{equivalent} if there exist
	affine automorphisms $\alpha \in \Aff(\AA^d)$, $\beta \in \Aff(\AA^n)$ such that
	\[
		(g_1, \ldots, g_n) = \beta \circ (f_1, \ldots, f_n) \circ \alpha \, .
	\]
	
	If the preimage of every linear hypersurface in $\AA^n$ under the morphism 
	$f = (f_1, \ldots, f_n) \colon \AA^d \to \AA^n$ is isomorphic to $\AA^{d-1}$, then we say that
	$f$ is a \emph{linear system of affine spaces}. Hence, every affine linear system of affine spaces
	is a linear system of affine spaces.
\end{definition}

\begin{remark}
	Every automorphism $f \colon \AA^n \to \AA^n$ is an affine linear system of affine spaces
	and two automorphisms $f, g \colon \AA^n \to \AA^n$ are equivalent, if they are the same
	up to affine automorphisms at the source and target.
\end{remark}

\begin{remark}
	Note that the notions ``affine linear hypersurface'' and ``affine linear system of affine spaces'' 
	are not intrinsic notions of the affine space and of morphisms between them.
	They depend on the choice of coordinate systems of the affine spaces 
	(up to affine automorphisms). Therefore, as mentioned in the introduction, we always make 
	a particular choice of the coordinates of the affine spaces involved.
\end{remark}

\begin{example}
	Let $f_1, \ldots, f_n \in \k[x_1 \ldots, x_d]$.
	If $\deg(f_i) \leq 1$ for each $i$, then
	$f \coloneqq (f_1, \ldots, f_n) \colon \AA^d \to \AA^n$ is called an \emph{affine linear morphism}.
	In case $f$ is surjective, it is an affine linear system of affine spaces.
\end{example}

Next, we list some basic properties of affine linear systems of affine spaces.

\begin{lemma}
	\label{lem.properties_lin_systems_of_aff_spaces}
	Let $f_1, \ldots, f_n \in \k[x_1, \ldots, x_d]$ be polynomials and let
	$f = (f_1, \ldots, f_n)$ be the corresponding morphism $\AA^d \to \AA^n$.
	\begin{enumerate}[leftmargin=*]
		\item \label{lem.properties_lin_systems_of_aff_spaces-1}
		If $f = (f_1, \ldots, f_n)$ is an affine linear system of affine spaces and if 
		$f_{i, 1}$ denotes the homogeneous part of $f_i$ of degree $1$ for $i=1, \ldots, n$, then 
		$f_{1, 1}, \ldots, f_{n, 1}$ are linearly independent over $\k$ in $\k[x_1, \ldots, x_d]_1$. In particular,
		$n \leq d$.
		\item \label{lem.properties_lin_systems_of_aff_spaces0}
		Assume that $f$ is an affine linear system of affine spaces. Then for all automorphisms
		$\varphi \in \Aut(\AA^d)$ and all $\alpha \in \Aff(\AA^n)$, the composition
		$\alpha \circ f \circ \varphi \colon \AA^d \to \AA^n$ is an affine linear system of affine spaces.
		\item \label{lem.properties_lin_systems_of_aff_spaces1} 
		Assume that $\deg(f) = \max_{1 \leq i \leq n} \deg(f_i) = 1$. Then $f$
		is an affine linear system of affine spaces if and only if $f \colon \AA^d \to \AA^n$ is surjective.
		In particular, if $d \geq n$, then up to equivalence there is exactly
		one affine linear system of affine spaces $\AA^d\to \AA^n$ of degree $1$.	
		\item \label{lem.properties_lin_systems_of_aff_spaces1.5}
		If $f_1, \ldots, f_n \in \k[x, y, z]$ are of degree $\le3$, then 
		$(f_1, \ldots, f_n)\colon \AA^3 \to \AA^n$ defines a linear system of affine spaces
		if and only if it defines an \emph{affine} linear system of affine spaces.
		\item 
		\label{lem.properties_lin_systems_of_aff_spaces2}
		Let $\pi \colon \AA^n \to \AA^l$ be a surjective affine linear morphism. If $f$ is an
		affine linear system
		of affine spaces, then the composition $\pi \circ f \colon \AA^d \to \AA^l$ as well.
		\item 
		\label{lem.properties_lin_systems_of_aff_spaces3}
		Let $\rho \colon \AA^r \to \AA^d$ be a surjective affine linear morphism. If 
		$f$ is an affine linear system of affine spaces, then $f \circ \rho$ as well. If $d \leq 3$ and if
		$f \circ \rho$ is an affine linear system of affine spaces, then $f$ as well.
		\item
		\label{lem.properties_lin_systems_of_aff_spaces4}
		Assume that $d = n$. If $f = (f_1, \ldots, f_n) \colon \AA^n \to \AA^n$ 
		is an affine linear system of affine spaces, then the determinant of the Jacobian
		of $f$ lies in $\k^\ast$.
	\end{enumerate}
\end{lemma}

\begin{proof}
	\ref{lem.properties_lin_systems_of_aff_spaces-1}: If
		there exists 
		$(\lambda_1, \ldots, \lambda_n) \in \k^n\setminus \{0\}$ such that 
		$\sum_{i=1}^{n} \lambda_i f_{i, 1} = 0$, we write 
		$\lambda_0=\sum_{i=1}^{n} \lambda_i f_{i}(0)\in \k$ and obtain that the polynomial 
		$\sum_{i=1}^{n}\lambda_i f_i-\lambda_0$ is either $0$ or defines a singular hypersurface of $\A^n$.
		In both cases $\sum_{i=1}^{n} \lambda_i f_i - \lambda_0$ does not define an $\AA^{d-1}$ in $\AA^d$.

	\ref{lem.properties_lin_systems_of_aff_spaces0}: This follows directly from the definition.
	
	\ref{lem.properties_lin_systems_of_aff_spaces1}: If $f$ is surjective, then the statement is clear.
	If $f$ is not surjective, then the image of $f$ is contained in an affine 
	linear hypersurface in $\AA^n$ and thus $f$ is not an affine linear system of affine spaces.
	
	\ref{lem.properties_lin_systems_of_aff_spaces1.5}: This
	follows from Corollary~\ref{coro:TrivialisationVariables}.
	
	\ref{lem.properties_lin_systems_of_aff_spaces2}: This follows, since the preimage of an affine linear
	hypersurface under $\pi$ is again an affine linear hypersurface.
	
	\ref{lem.properties_lin_systems_of_aff_spaces3}: Let $H \subset \AA^n$ be an affine 
	linear hypersurface.
	Then the preimage $(f \circ \rho)^{-1}(H)$ is isomorphic to $f^{-1}(H) \times \AA^{r-d}$. Hence, 
	the first claim follows. On the other hand, 
	as $f^{-1}(H)$ has dimension $d-1$ and 
	since Zariski's Cancellation Problem has an affirmative answer for the affine line
	(see \cite[Corollary~2.8]{AbHeEa1972On-the-uniqueness-}) and the affine plane
	(see \cite{Fu1979On-Zariski-problem, MiSu1980Affine-surfaces-co} and
	 \cite[Theorem~4]{Ru1981On-affine-ruled-ra}), 
	the second claim follows.
	
	\ref{lem.properties_lin_systems_of_aff_spaces4}: This follows from Lemma~\ref{lem.Derksen} below.
\end{proof}

The next Lemma is essentially due to Derksen, see~\cite[Lemma~2.3]{EsSh1997Some-combinatorial}:

\begin{lemma}
	\label{lem.Derksen}
	Let $f_1, \ldots, f_n \in \k[x_1, \ldots, x_n]$ and let $f = (f_1, \ldots, f_n) \colon \AA^n \to \AA^n$.
	Then the determinant of the Jacobian of $f$ lies in $\k^\ast$ if and only if 
	the preimage of each affine linear hypersurface under $f$ is a smooth hypersurface in $\AA^n$.
\end{lemma}

\begin{proof}
	The determinant of the Jacobian of $f$ does not lie in $\k^\ast$ if and only if there exist
	$\lambda_1, \ldots, \lambda_n \in \k$, not all equal to zero, and there is a point $a \in \AA^n$ such that
	\[
		\sum_{i=0}^n \lambda_i \frac{\partial f_i}{\partial x_j}(a) = 0 \quad \textrm{for each $j = 1, \ldots, n$} \, .
	\]
	However, this last condition is equivalent to the existence of some
	$\lambda_0 \in \k$ and some $(\lambda_1, \ldots, \lambda_n) \in \k^n \setminus \{ 0 \}$ such that
	either $\lambda_0 + \lambda_1 f_1 + \ldots + \lambda_n f_n$ is zero
	or defines a singular hypersurface in $\AA^n$.
\end{proof} 

In the next Proposition, we study affine linear systems of affine spaces $\AA^2\to \AA^2$ of
degree $\leq 3$ up to affine automorphisms at the source and target.

\begin{proposition}
	\label{cor.Autos_of_A2_deg3}
	Let $f_1, f_2 \in \k[x, y]$ of degree $\leq 3$ 
	such that $f = (f_1, f_2) \colon \AA^2 \to \AA^2$ is a linear system of affine spaces.
	Then, up to affine coordinate changes at the source and target, we get
	$f = (x + q(y), y)$ where $q \in \k[y]$.
\end{proposition}

\begin{proof}
	By Corollary~\ref{Coro:A1inA2smalldegree}, we may assume after an affine coordinate change
	in $(x, y)$ that $f_1 = x + q(y)$ for some $q \in \k[y]$ of degree $\leq 3$. 
	Set $\psi = (x - q, y) \in \Aut(\AA^2)$. The determinant of the Jacobian
	of $(x, f_2(x-q, y)) = f \circ \psi$ is a non-zero constant (due to Lemma~\ref{lem.properties_lin_systems_of_aff_spaces}\ref{lem.properties_lin_systems_of_aff_spaces4})
	and it is equal to the $y$-derivative of $f_2(x-q, y)$. Hence, 
	$f_2(x - q, y) = a y + p(x)$ for some $a \in \k^\ast$ and 
	$p \in \k[x]$, i.e.~$f_2 = a y + p(x + q)$. After scaling
	$f_2$ we may assume $a = 1$. If $\deg(q) \leq 1$, then $\psi \in \Aff(\AA^2)$
	and since $f \circ \psi = (x, y + p(x))$, the result follows after conjugation with $(x, y) \mapsto (y, x)$.
	If $\deg(q) \geq 2$, then $\deg(p) \leq 1$, since otherwise $\deg(f_2) = \deg(p) \deg(q) \geq 4$.
	Thus $\varphi = (x, y - p(x)) \in \Aff(\AA^2)$ and since 
	$\varphi \circ f = (x + q(y), y)$, the result holds.
\end{proof}

\subsection{Linear systems of affine spaces of degree $3$ 
with a conic in the base locus}

In this subsection we study linear systems $f \colon \AA^3 \to \AA^n$ of degree $3$ such that the
rational map $\PP^3 \bir \PP^n$ which extends $f$ contains a conic in the base locus. In fact,
this study will be important in order to prove that every automorphism of degree $3$ of $\AA^3$
can be brought into standard form (Proposition~\ref{prop:globally_xp_i+q_i} below). 
As explained in the introduction, we say that an affine linear system of affine spaces $f\colon \AA^3\to \AA^n$ is in \emph{standard form} if $f = (xp_1 + q_1, \ldots, x p_n + q_n)$
for some polynomials $p_i, q_i \in \k[y, z]$.

\begin{proposition}
	\label{prop.standard_form_if_conic_at_infinity}
	Let $f_1, \ldots, f_n \in \k[x, y, z]$ be polynomials and 
	assume that $f = (f_1, \ldots, f_n) \colon \AA^3 \to \AA^n$ is
	a linear system of affine spaces of degree $3$ such that there
	is a homogeneous irreducible polynomial of degree $2$
	that divides the homogeneous parts of degree $3$ of $f_1, \ldots, f_n$. Then $f$
	is equivalent to a linear system of affine spaces in standard form.
\end{proposition}

\begin{proof}
For $i=1,\ldots,n$, we write $f_i=\sum_{j=0}^3 f_{i,j}$ where $f_{i,j}\in \k[x,y,z]_j$. Applying an automorphism of $\A^n$ we may assume that $f_{i,3}\not=0$ for each $i$. By assumption, there is an irreducible conic $\Gamma \subset \PP^2$ that is contained in the zero locus of $f_{i,3}$, for each $i\in \{1,\ldots,n\}$. Moreover, for each $i$, $f_i$ defines an $\A^2$ inside $\A^3$, so the polynomial $f_{i,3}$ defines in $\PP^2$ the conic $\Gamma$ and a tangent line to that conic in a point $q_i$
and the closure in $\PP^2$ of the hypersurface given by $f_i$ is singular at $q_i$ 
(see Corollary~\ref{cor.Properties_of_hypersurfaces_at_infinity}). If all the points $q_1,\ldots,q_n$ are the same, we can assume that these are $[1:0:0]$, and obtain the result by Lemma~\ref{lem.hypersxpq}. We thus assume that two of the $q_i$'s are distinct and derive a contradiction. 
	We may assume that $q_1\not=q_2$ by applying a permutation of $\A^n$.
	Applying automorphisms of $\A^3$, we may moreover assume that 
	$f_1 = x y^2 + y(z^2 + az + b) + z$ for some $a, b \in \k$ (see
	Proposition~\ref{Proposition:XxpqisA2smalldegreelistpqDeg3}). Hence, $q_1=[1:0:0]$, $\Gamma$ is the conic $xy+z^2=0$ and $q_2\in \Gamma\setminus \{q_1\}$, so $q_2=[-\xi^2:1:\xi]$ for some $\xi\in \k$. Replacing $f_2$ with $f_2\lambda$ for some $\lambda\in \k^*$, we obtain
	\[
		f_{1,3}= y(xy+z^2) \, , \quad f_{2,3}=(x-\xi^2y+2\xi z)(xy+z^2) \, .
	\]
	
	For each $\mu\in \k$, the polynomial $f_2+\mu^2 f_1$ defines a hypersurface $X_\mu\subset \A^3$ and its homogeneous part of degree $3$ is $(x-\xi^2y+\mu^2y+2\xi z)(xy+z^2)$. By Corollary~\ref{cor.Properties_of_hypersurfaces_at_infinity}\ref{cor.PHI_Conic_tangent}, the line $\ell_\mu$ given by $x-\xi^2y+\mu^2y+2\xi z$ is tangent to $\Gamma$. Choosing 
	$\mu=\xi$ when $\xi \neq 0$ and choosing $\mu = 1$ when $\xi = 0$	 gives $\car(\k)=2$.
 	We may then replace $f_2$ with $f_2+\xi^2f_1$ and assume that $\xi=0$. The point of tangency of $\Gamma$ and $\ell_\mu$ is then 
 	$p_\mu= [\mu^2: 1: \mu]$.
	
	Suppose first that $f_{1,2}=f_{2,2}=0$. We obtain
	 \[
	 	f_1 = y(xy+z^2)+ by+z \, , \quad f_2=x(xy+z^2)+\alpha x+\beta y +\gamma^2 z +\delta
	 \]
	 for some $\alpha,\beta,\gamma,\delta\in \k$. The polynomial $f_2+\gamma^2 f_1= (x+\gamma^2 y)(xy+z^2)+\alpha x+(\beta +b\gamma^2)y+\delta$ defines an $\A^2$, so the same holds 
	 when we replace $x$ and $z$ with $x+\gamma^2y,z+\gamma y$ respectively, hence for the polynomial
	 $x(xy+z^2)+\alpha x +(\beta+(b+\alpha)\gamma^2)y+\delta$, impossible by Proposition~\ref{Prop:XfibratReminder} (applied to the polynomial obtained by exchanging $x$ and $y$).	 
	 
	 We now assume that $f_{1,2}$ and $f_{2,2}$ are not both zero.
	There is an affine automorphism of $\A^3$ that sends $f_2+\mu^2 f_1$ onto $h=xy^2+y (z^2+cz+d)+z$ for some $c,d\in \k$ (Proposition~\ref{Proposition:XxpqisA2smalldegreelistpqDeg3}). Thus, $f_2+\mu^2 f_1$ is obtained by applying an element of $\GL_3(\k)$ to $h'=h(x+\varepsilon_1,y+\varepsilon_2,z+\varepsilon_3)$ for some $\varepsilon_1,\varepsilon_2,\varepsilon_3\in \k$. As $h'= h'_0+h'_1+h'_2+h'_3$ where $h'_i \in\k[x,y,z]_i$ and $h'_3=y(xy+z^2)$, 
	$h'_2=\varepsilon_1 y^2+cyz+\varepsilon_2 z^2$ are both singular at $[1:0:0]$, the homogeneous part of degree $2$ of $f_2+\mu^2 f_1$ is singular at $p_\mu$. 
	
	As $f_{1,2}$ and $f_{2,2}$ are not both zero and the set $\{p_\mu \mid \mu \in \k\}$ is not contained in a line, there is no linear factor that divides both $f_{1,2}$ and $f_{2,2}$.
	However, as $f_{2,2}+\mu^2 f_{1,2}$ is divisible by a linear factor for each $\mu\in \k$, there exist $s,t\in \k[x,y, z]_1$ such that $f_{1,2},f_{2,2}\in \k[s,t]$(Lemma~\ref{Lemm:PencilLinearFactor}).
	Remembering that $f_{1,2}=ayz$, we prove first that $a=0$. Indeed, otherwise $\k[s,t]=\k[y,z]$ and $f_{2,2}+\mu^2 f_{1,2} \in \k[y,z]$ is singular at $p_\mu$ so is a multiple of 
	$(\mu y+ z)^2=\mu^2 y^2 + z^2$, impossible as it contains $yz$ for infinitely many $\mu$. 
	Now that $a=0$ is proven, the polynomial $f_{2,2}+\mu^2 f_{1,2}=f_{2,2}$ is singular at each point $p_\mu$, so $f_{2,2}=0$, 
	in contradiction with the above assumption.
\end{proof}

\subsection{Affine linear systems in characteristic $2$ and $3$}
We call a morphism $f \colon Y \to X$  an \emph{$\AA^1$-fibration} if each closed fiber is (schematically) isomorphic to $\AA^1$. We moreover say that the $\AA^1$-fibration $f$ is locally trivial in the Zariski $($respectively \'etale$)$ topology if for each $x \in X$
	there is an open neighbourhood $U \subset X$ of $x$ (respectively 
	 an \'etale morphism $U \to U'$ onto an open neighbourhood $U'$ of $x$ in $X$) 
	such that
	the fiber product $U \times_X Y \to U$ is isomorphic to $U\times \AA^1$ over $U$. 
	
	Recall from the introduction, that an 
	$\AA^1$-fibration $f \colon Y \to X$ is called \emph{trivial} if there
	exists an isomorphism $\varphi \colon X \times \AA^1 \to Y$ such that the
	composition $f \circ \varphi \colon X \times \AA^1 \to X$ is the projection onto the first factor.
	
	An $\A^1$-bundle is then simply an $\A^1$-fibration that is locally trivial in the Zariski topology.
	
We now give two examples of linear systems of affine spaces of degree $3$
that are not equivalent to linear systems in standard form.

\begin{lemma}
	\label{lem.Lin_system_char_2}
	Assume that $\car(\k) = 2$ and let 
	\[
		f = x + z^2 + y^3 \quad \textrm{and} \quad
		g = y + x^2
	\] Then, $\pi=(f, g) \colon \AA^3 \to \AA^2$ is an affine linear system of affine spaces, which is not equivalent
	to an affine linear system in standard form. 
	Moreover, $\pi$ is an $\AA^1$-fibration that is not locally trivial in the \'etale topology.
\end{lemma}
\begin{proof}
	If $\lambda \neq 0$, then 
	$\lambda^2 f + g = \lambda^2 x + y + ( x + \lambda z)^2 + \lambda^2 y^3$ 
	defines an $\AA^2$ in $\AA^3$, since the linear polynomials
	$\lambda^2 x + y$, $ x + \lambda z$ and $y$ are linearly independent
	in $\k[x, y, z]_1$. On the other hand, both $f$ and $g$ define an $\AA^2$
	in $\AA^3$ as well. This implies that $\pi=(f, g) \colon \AA^3 \to \AA^2$ is a linear system of
	affine spaces and thus an affine linear system of affine spaces by Lemma~\ref{lem.properties_lin_systems_of_aff_spaces}\ref{lem.properties_lin_systems_of_aff_spaces1.5}.
	
	Let $X, Y \subset \PP^3$ be the closures of the hypersurfaces in $\AA^3$ which are given by $f$ and $f + g$,
	respectively. By Corollary~\ref{cor.Properties_of_hypersurfaces_at_infinity}\ref{cor.PHI_line} the singular
	locus of $X$ is equal to $[0:1:0:0]$ and the singular locus of $Y$ is equal to $[0:1:0:1]$.
	In particular, $X$, $Y$ have no common singularity and thus, 
	$\pi$ is not equivalent to an affine linear system in standard form by Lemma~\ref{lem.hypersxpq}.
	
	It remains to see that all closed fibres of $\pi$ are isomorphic to $\A^1$ but that $\pi$ is not 
	 locally trivial in the \'etale topology. 
	To simplify the situation, we apply some non-affine automorphisms 
	at the source and the target. We first apply $(x,y+x^2,z)$ (at the source) to get $(x+z^2+(y+x^2)^3,y)$. Applying $(x+y^3,y)$ at the target and $(x,y,z+x^3+xy)$ at the source gives
	\[
		\phi = (x +x^4y+z^2 , y) \colon \AA^3 \to \AA^2 \, .
	\]
	The fibre over a point $(x_0,y_0)$ with $y_0=0$ is isomorphic to $\A^1$, via its projection onto $z$. The fibre over a point $(x_0,y_0)\in \A^2$ with $y_0\not=0$ is isomorphic to $\A^1$, as one can apply $z\mapsto z+\sqrt{y_0}x^2$ to reduce to the previous case.
	
	It remains to see that $\phi$ is not  locally trivial in the \'etale topology. 
	The fibre $F$ of $\phi$ over the (non-closed) generic point of $\{x = 0\}$ is the scheme
	given by $x +x^4y+z^2$ inside $\AA^2_{\k(y)} = \Spec(\k(y)[x, z])$. By~\cite[Corollary~2.3.1 and Lemma~1.2]{Ru1970Forms-of-the-affin}, $F$ is non-isomorphic to the affine line $\AA^1_{\k(y)}$ over $\k(y)$,
	however after extending the scalars to $\k(\sqrt{y})$ we get 
	\[
		F \times_{\Spec(\k(y))} {\Spec(\k(\sqrt{2}))} \simeq \AA^1_{\k(\sqrt{y})} \, .
	\]
	By \cite[Lemma~1.1]{Russell1976} there doesn't exist any separable field extension $k(y) \subseteq K$ such that
	$F \times_{\Spec(\k(y))} \Spec(K) \simeq \AA^1_K$. Hence, $\phi$ and thus $\pi$ are not locally trivial in
	the \'etale topology.
	
\end{proof}

\begin{lemma}
	\label{lem.Lin_system_char_3}
	Assume that $\car(\k) = 3$ and let 
	\[
		f = x + z^2 + y^3 \quad \textrm{and} \quad
		g = z+x^3
	\]Then, $\pi=(f, g) \colon \AA^3 \to \AA^2$ is an affine linear system of affine spaces, which is not equivalent
	to an affine linear system in standard form. 
	Moreover, $\pi$ is an $\AA^1$-fibration that is not locally trivial in the \'etale topology.
\end{lemma}
\begin{proof}
	For each $\lambda \in \k$, the polynomial $f + \lambda^3 g = \lambda^3 z + x + z^2 + (y + \lambda x)^3$ defines an $\AA^2$ in $\AA^3$: replacing $y$ with $y-\lambda x$ 
	and $x$ with $x- \lambda^3 z$ gives $x + z^2 + y^3$.
	On the other hand, $g$ also defines an $\AA^2$
	in $\AA^3$. This implies that $\pi=(f, g) \colon \AA^3 \to \AA^2$ is a linear system of
	affine spaces and thus an affine linear system of affine spaces by Lemma~\ref{lem.properties_lin_systems_of_aff_spaces}\ref{lem.properties_lin_systems_of_aff_spaces1.5}.

	Let $X, Y \subset \PP^3$ be the closures of the hypersurfaces
	of $\AA^3$ which are given by $f$ and $g$, respectively. Then the singular locus of $X$
	is only the point $[0: 1: 0: 0]$ and the singular locus of $Y$ is the line $w = x = 0$,
	by Corollary~\ref{cor.Properties_of_hypersurfaces_at_infinity}\ref{cor.PHI_line}).
	Hence, $(f, g)$ is not equivalent to an affine linear system in standard form (see Lemma~\ref{lem.hypersxpq}).
	
	It remains to see that all closed fibres of $\pi$ are isomorphic to $\A^1$ but that $\pi$ is 
	not a trivial $\AA^1$-fibration. To simplify the situation, we apply some non-affine automorphisms 
	at the source and the target. We first apply $(x,y-x^2,z-x^3)$ (at the source) to get $(x+y^3+z^2+x^3z,z)$, then apply $(x-y^2,y)$ at the target to obtain
	\[
		\phi = (x+y^3+x^3z,z) \colon \AA^3 \to \AA^2 \, .
	\]
	The fibre over a point $(x_0,y_0)$ with $y_0=0$ is isomorphic to $\A^1$, via its projection onto $y$. The fibre over a point $(x_0,y_0)\in \A^2$ with $y_0\not=0$ is isomorphic to  $\AA^1$, 
	as one can apply $y\mapsto y-\sqrt[3]{y_0}x$ to reduce to the previous case. 
	
	Now, the fibre $F$ of $\phi$ over the generic point of $\{z = 0\}$
	is the scheme given by $x+y^3+x^3z$ inside $\AA^2_{\k(z)} = \Spec(\k(z)[x, y])$.
	Using again~\cite{Ru1970Forms-of-the-affin}, we find the
	same way as in the proof of Lemma~\ref{lem.Lin_system_char_2}, that there exists
	no separable field extension $\k(z) \subseteq K$ such that $F \times_{\Spec(\k(z))} \Spec(K) \simeq \AA^1_K$,
	however
	\[
		F \times_{\Spec(\k(z))} {\Spec(\k(\sqrt[3]{z}))} \simeq \AA^1_{\k(\sqrt[3]{z})} \, .
	\]
	This implies again, that neither $\phi$ nor $\pi$ is locally trivial in the \'etale topology.
\end{proof}

We now prove that these two examples of linear systems are unique in some sense (see Lemma~\ref{lem.char2} and \ref{lem.char3} below).

\begin{lemma}\label{A2A3linindep}
Let $\ell_1,\ell_2,\ell_3\in \k[x,y,z]_1$ be three linear polynomials such that $\ell_2$ and $\ell_3$ are linearly independent. Then, $\sum_{i=1}^3 (\ell_i)^i$ defines an $\A^2$ in $\A^3$ if and only if $\ell_1,\ell_2,\ell_3$ are linearly independent.
\end{lemma}
\begin{proof}
If $\ell_1,\ell_2,\ell_3$ are linearly independent, we may apply an element of $\GL_3(\k)$ and assume that $\ell_1=x$, $\ell_2=y$, $\ell_3=z$. 
Thus, $\sum_{i=1}^3 (\ell_i)^i=x+y^2+z^3$ defines an $\A^2$ in $\A^3$. Otherwise, we may assume that $\ell_1=ax+by$, $\ell_2=x$, $\ell_3=y$, so the hypersurface of $\A^3$ given by $\sum_{i=1}^3 (\ell_i)^i=0$ is isomorphic to $\Gamma \times \A^1$, where $\Gamma\subset \A^2$ is the curve given by $ax+by +x^2+y^3=0$. It remains to see that $\Gamma$ is not isomorphic to $\A^1$ (by the positive answer to Zariski's Cancellation Problem,
see \cite[Corollary~2.8]{AbHeEa1972On-the-uniqueness-}). Indeed, the closure of $\Gamma$ in $\PP^2$ 
would otherwise be an irreducible curve singular at infinity, which is here not the case.
\end{proof}
\begin{lemma}
	\label{lem.char2}
	Assume that $\car(\k) = 2$ and let $f = (f_1, \ldots, f_n) \colon \AA^3 \to \AA^n$ be
	an affine linear system of affine spaces.
	Suppose that $f_i = \sum_{j=0}^3 f_{i, j} \in \k[x, y, z]$ for each $i \in \{1, \ldots, n\}$, where 
	$f_{i,j}\in \k[x, y, z]_j$ and that 
	\[
		\Span_\k(f_{1,3},\ldots,f_{n,3})=\k y^3\text{ and }\Span_\k(f_{1,2},\ldots,f_{n,2},y^2)= 
		\k x^2 +\k y^2 +\k z^2.
	\] 
	Then, $n=2$ and $f$ is equivalent to the linear system $(x + z^2 + y^3,y + x^2)$ of Lemma~$\ref{lem.Lin_system_char_2}$.
\end{lemma}
\begin{proof}
As $\Span_\k(f_{1,2},\ldots,f_{n,2},y^2)= \k x^2 +\k y^2 +\k z^2$, we have $n\ge 2$. Applying a linear automorphism of $\A^n$, we may assume that $f_{1,3}=y^3$ and that $f_{i,3}=0$ for $i\ge 2$. We may moreover assume that $\Span_\k(f_{1,2},f_{2,2},y^2)= \k x^2 +\k y^2 +\k z^2$ by possibly adding multiples of 
$f_i$, $i \geq 2$ to $f_1$ and then permuting the $f_i$, $i\ge 2$.
Hence, $f_{1,2}=\ell_1^2+\alpha y^2$ and $f_{2,2}=\ell_2^2+\beta y^2$, where $\ell_1,\ell_2\in \k[x,z]_1$ are linearly independent and $\alpha,\beta\in \k$. Applying a linear automorphism at the source that fixes $y$, we may reduce to the case where $f_{1,2}=z^2$ and $f_{2,2}=x^2$. We may moreover assume that $f_{i,0}=0$ for each $i$, by applying a translation at the target.

We then choose $a,b,c,d\in \k$ such that $f_{1,1}=ax+bz \mod \k y$ and $f_{2,1}=cx+dz\mod \k y$. For each $\lambda\in \k$, the polynomial 
\[f_1+\lambda^2 f_2=((a+\lambda^2c)x+(b+\lambda^2 d)z+ \zeta y)+ (z+\lambda x)^2+y^3\]
defines an $\A^2$ in $\A^3$ (where $\zeta\in \k$ depends on $\lambda$). This implies that $((a+\lambda^2c)x+(b+\lambda^2 d)z+ \zeta y)$, $y$ and $z+\lambda x$ are linearly independent (Lemma~\ref{A2A3linindep}), and thus that $(a+\lambda^2c)+(b+\lambda^2 d)\lambda\not=0$. As this is true for all $\lambda$, we obtain $a\not=0$ and $b=c=d=0$, so $f_1=ax+\xi y+z^2+y^3$ and $f_2=\nu y+x^2$ for some $\xi,\nu\in \k$. As $f_2$ defines an $\A^2$ in $\A^3$, we have $\nu\not=0$. 
Applying $x \mapsto \sqrt{\nu} x$ at the source
and replacing $f_2$ by $\nu^{-1} f_2$, we may assume that $\nu=1$. We then replace $f_1$ with $f_1 + \xi f_2$ and $z$ with $z + \sqrt{\xi} x$ to assume $\xi=0$. This gives $(f_1,f_2)=(a x + z^2 + y^3,y + x^2)$. 
After replacing $x, y, z$ with $\mu x, \mu^2 y, \mu^3 z$ 
at the source where $\mu \in \k$
is chosen with $\mu^5 = a$ and after replacing $f_1, f_2$ with 
$f_1 / \mu^6$, $f_2 / \mu^2$, respectively, we may assume further that $a = 1$. This
achieves the proof if $n=2$. 

It remains to see that $n\ge 3$ leads to a contradiction. We add a multiple of $f_2$ to $f_3$ and may assume that $f_{3,2}$ is equal to $\varepsilon^2y^2+\tau^2z^2=(\varepsilon y+\tau z)^2$ for some $\varepsilon,\tau\in\k$. For each $\lambda\in \k$, the polynomial $\lambda^2f_1+f_2+ f_3=
(\lambda^2 x+ y+ f_{3,1})+( x+\varepsilon y+ (\lambda+\tau) z)^2+\lambda^2 y^3$ defines an $\A^2$ in $\A^3$. Hence, for each $\lambda\in \k^*$, the polynomials $\lambda^2 x+ y+ f_{3,1}$, $x+\varepsilon y+ (\lambda+\tau) z$ and $y$ are linearly independent (Lemma~\ref{A2A3linindep}). Writing $f_{3,1}=\alpha x+\beta z +\gamma y$, with $\alpha,\beta,\gamma \in \k$, the polynomials
\[ (\lambda^2+\alpha)x+\beta z\text{ and } x+(\lambda+ \tau) z\] are linearly independent, so $0\not=(\lambda^2+\alpha)(\lambda+ \tau)+\beta=\lambda^3+\lambda^2\tau+\lambda\alpha+(\alpha\tau+\beta)$, for each $\lambda\in \k^*$. Hence, $\alpha=\tau=\beta=0$, which yields $f_3\in \k[y]$. As $f_3$ defines an $\A^2$, we obtain $f_3=\gamma y$ with $\gamma\in \k^*$. But then $f_2+\gamma^{-1}f_3=x^2$ does not define an $\A^2$,
contradiction.
\end{proof}

\begin{lemma}
	\label{lem.char3}
	Assume that $\car(\k) = 3$, let $f_1, \ldots, f_n \in \k[x, y, z]$ of degree $\leq 3$
	such that $f = (f_1, \ldots, f_n) \colon \AA^3 \to \AA^n$ is an affine linear system of affine spaces
	and that the linear span of the homogeneous
	parts of degree $3$ of the $f_1, \ldots, f_n$ is a subspace of dimension $\ge2$ of
	$\k x^3 \oplus \k y^3 \oplus \k z^3$. Then either $f$ is equivalent to a linear system in standard form
	or $n = 2$ and $f$ is equivalent to the linear system
	$(x + z^2 + y^3, z+x^3)$
	in Lemma~$\ref{lem.Lin_system_char_3}$.
\end{lemma}
\begin{proof}
	Let $f_{i, j} \in \k[x, y, z]$ be the homogeneous part of degree $j$ of $f_i$ for $i=1, \ldots, n$, and let us define $V_j=\Span_\k(f_{1,j},\ldots,f_{n,j})\subseteq \k[x,y,z]_j$ for each $j$. By assumption, $V_3\subseteq \k x^3 \oplus \k y^3 \oplus \k z^3$, so $\sum \lambda_i f_{i, 3}$
	is a third power for all $(\lambda_1, \dots, \lambda_n) \in \k^n$. We may moreover assume that $V_0=0$ by applying a translation at the target. 
	
	It follows from Corollary~\ref{cor.Properties_of_hypersurfaces_at_infinity}\ref{cor.PHI_line}
	that for each $(\lambda_1, \dots, \lambda_n) \in \k^n$ such that $\sum \lambda_i f_{i, 3}\not=0$ (which is true for a general $(\lambda_1, \dots, \lambda_n)$), the polynomial 
	$\sum \lambda_i f_{i, 2}$ is either zero or defines a conic in $\mathbb{P}^2$ that is singular on a point of the triple line defined by $\sum \lambda_i f_{i, 3}$.
	
	 Suppose first that $\gcd(V_2)=1$, and thus that $\dim V_2\ge 2$. Lemma~\ref{LinearSystemLinearFactors} gives two polynomials $s,t\in \k[x,y,z]_1$ such that $V_2\subseteq \k[s,t]$. Changing coordinates on $\A^3$, we may assume that $s=y$ and $t=z$. For general $(\lambda_1,\ldots,\lambda_n)\in \k^n$,  
	 the hypersurface in $\PP^2$ given by the homogeneous polynomial $\sum \lambda_i f_{i, 2}$ is only singular at the point $p=[1:0:0]\in \p^2$ (as $\car(\k) \neq 2$), which is on the triple line defined by $\sum \lambda_i f_{i, 3}$. This implies that $V_3 \subseteq \k y^3 \oplus \k z^3$, so $f$ is a linear system 
	 in standard form.
	 
	 We may now assume that a linear polynomial $h\in \k[x,y,z]_1$ divides each element of $V_2$. Applying an element of $\GL_3$ at the source, we may thus assume that $h=z$. 
	 If a point $p\in \p^2$ is such that all elements of $V_2$ and $V_3$ vanish at $p$, we apply an element of $\GL_3$ at the source to assume $p=[1:0:0]$ and obtain that $f$ is in standard form. Hence, we may assume that the elements of $V_3$ do not share a common zero on the line $z=0$.
	 
	 We now prove that $z^2$ divides $f_{i,2}$ for each $i \in \{1,\ldots,n \}$. We suppose the converse to derive a contradiction. Applying a general element of $\GL_n$ at the target, we obtain that $f_{1,2}$ is not a multiple of $z^2$ and that $f_{1,3}$ and $f_{2,3}$ do not share a common zero on the line $z=0$. Choosing $\ell_1,\ell_2\in \k[x,y,z]_1$ such that $f_{1,3}=\ell_1^3$ and $f_{1,3}=\ell_1^3$, the elements $\ell_1,\ell_2,z$ are linearly independent. We may thus apply an element of $\GL_3$ and assume that $f_{1,3}=x^3$ and $f_{2,3}=y^3$. We write $f_{1,2}=z(ax+by+c z)$ 
	 $f_{2, 2} = z g$ for some $a,b,c\in \k$ 
	 with $a,b$ not both equal to zero and $g \in \k[x, y, z]_1$. 
	 For each $\lambda\in \k$, the polynomial $f_1+\lambda^3f_2$ defines an $\A^2$ in $\A^3$ and as $f_{1,3}+\lambda^3f_{2,3}=(x+\lambda y)^3$, the hypersurface in $\PP^2$ given by the homogeneous polynomial 
	 $f_{1,2}+\lambda^3 f_{2,2}= z (ax + by + cz + \lambda^3 g)$ is singular at a 
	 point $p_\lambda$ of the line in $\PP^2$ given by $x+\lambda y=0$ (Corollary~\ref{cor.Properties_of_hypersurfaces_at_infinity}\ref{cor.PHI_line}). This yields $p_\lambda=[-\lambda:1:0]$, and thus 
	 $-\lambda a + b + \lambda^3 g(-\lambda, 1, 0) = 0.$ 
	 This being true for each $\lambda$, we get $a=b=0$, giving the desired contradiction.
	 
	 We now show that $\dim(V_3)=2$. If $\dim(V_3)=3$, we may assume $(f_{1,3},f_{2,3},f_{3,3})=(x^3,y^3,z^3)$. By Lemma~\ref{lem.almost_eigenvectors}, there exists $(\lambda_1, \lambda_2, \lambda_3) \neq (0, 0, 0)$ and $\varepsilon \neq 0$ such that
	$\sum \lambda_i^3 f_{i, 1} = \varepsilon \ell_1$, where $ \ell_1=\lambda_1 x + \lambda_2 y + \lambda_3 z$. Hence, the polynomial $\sum \lambda_i^3 f_i$ is equal to $\varepsilon \ell_1+\nu z^2+(\ell_1)^3$ for some $\nu\in \k$ and does not define an $\A^2$ in $\A^3$: it is reducible
	if $\nu=0$ or if $z$ and $\ell_1$ are collinear, and otherwise does not define an $\A^2$ by Lemma~\ref{A2A3linindep}.
	
	Now that $\dim(V_3)=2$ and that the elements of $V_3$ do not share a common zero point on $z=0$, we may apply an element of $\GL_3$ that fixes $z$ to get $V_3=\k x^3+\k y^3$. Moreover, $V_2=\k z^2$ (as otherwise $V_2=\{0\}$ would give a linear system in standard form
	after exchanging $x$ and $z$). 
	We apply an element of $\GL_n$ at the target and assume that $f_{1,2}=z^2$ and $f_{1,3}\not=0$. We then add to $f_2$ a linear combination of the other $f_i$ and assume that $f_{2,2}=0$ and that $f_{2,3}$ is not a multiple of $f_{1,3}$. Applying again at the source an element of $\GL_3$ that fixes $z$, 
	we obtain $f_{1,3}=y^3$, $f_{2,3}=x^3$. We get $\alpha,\beta,\gamma,\delta,\varepsilon,\zeta\in \k$ such that
	\[
		f_1= (\alpha x+\beta y+\gamma z)+ z^2+y^3 \, , \quad f_2=(\delta x+\varepsilon y+\zeta z)+x^3 \, .
	\]
	For each $\lambda\in \k$, the polynomial $f_1+\lambda^3 f_3$ defines an $\A^2$ in $\A^3$. This implies that $(\alpha+\lambda^3\delta)x+(\beta+\lambda^3 \varepsilon)y+(\gamma+\lambda^3 \zeta)z$,
	$z$ and $y+\lambda x$ are linearly independent (Lemma~\ref{A2A3linindep}). Hence, $\lambda(\beta+\lambda^3 \varepsilon)-(\alpha+\lambda^3 \delta)\not=0$. This being true for each $\lambda$, we obtain $\beta=\delta=\varepsilon=0$ and $\alpha\not=0$. Hence $f_1= \alpha x+\gamma z+ z^2+y^3, f_2=\zeta z+x^3,$ with $\alpha\zeta\not=0$. 
	Replacing $f_1$ with $f_1-(\gamma/\zeta) \cdot f_2$ and replacing $y$ with $y + \kappa x$ where $\kappa^3=\gamma/\zeta$, we may assume that $\gamma=0$. 
	It remains then to choose $\xi\in \k^*$ with $\alpha^3\zeta=\xi^{15}$, to replace $x,y,z$ with $\xi^6/\alpha x, \xi^2 y, \xi^3z$ at the source and $f_1,f_2$ with $f_1/\xi^6$, $f_2\alpha^3/\xi^{18}$ at the target, to obtain 
	\[
		f_1= x+ z^2+y^3 \, , \quad f_2=z+x^3 \, .
	\]
	Thus, $f$ is the linear system of affine spaces 
	in Lemma~$\ref{lem.Lin_system_char_3}$ if $n=2$. It remains to see that $n\ge 3$ yields a contradiction.
	Adding to $f_3$ a linear combination of $f_1,f_2$ we obtain that $f_{3,3}=0$. This gives $f_3=\alpha x+\beta y +\gamma z+ \theta z^2$ with $\alpha,\beta,\gamma,\theta\in \k$. 
	Replacing $f_3$ by a multiple, we may assume that $\alpha\not=-1$ and $\theta\not=-1$. 
	For each $\lambda\in \k$, the polynomial $f_1+\lambda^3f_2+f_3=(1+\alpha)x+\beta y+
	(\gamma + \lambda^3) z+(1+\theta)z^2+(y+\lambda x)^3$ 
	defines an $\A^2$ in $\A^3$, so $y+\lambda x, z, (1+\alpha)x+\beta y+(\gamma + \lambda^3)z$ are linearly independent (Lemma~\ref{A2A3linindep}). This implies that $\beta\lambda-(1+\alpha)\not=0$. As this is true for each $\lambda$, we get $\beta=0$. But then the linear parts of $f_1,f_2,f_3$ are linearly dependent, contradicting Lemma~\ref{lem.properties_lin_systems_of_aff_spaces}\ref{lem.properties_lin_systems_of_aff_spaces-1}.
	\end{proof}

\subsection{Linear systems of affine spaces of degree $3$ 
with a line in the base locus}

In the following lemma we give 
necessary conditions for a polynomial of degree $\leq 3$ such that it defines an $\AA^2$
in $\AA^3$ and this hypersurface contains in its closure in $\PP^3$ a specific line.

\begin{lemma}\label{Lem:AlongALine}
	Let $F\in \k[w,x,y,z]$ be a homogeneous polynomial of degree $3$ such that $f=F(1,x,y,z)$ satisfies $\Spec(\k[x,y,z]/(f))\simeq \A^2$ and such that $F(0,x,0,z)=0$. Write $F$ as 
	\[
	F=wa_2(x,z)+yb_2(x,z)+w^2c_1(x,z)+wyd_1(x,z)+y^2e_1(x,z)+F_3(w,y)
	\]
	where $a_2,b_2\in \k[x,z]$ are homogeneous of degree $2$, $c_1,d_1,e_1\in \k[x,z]$ are homogeneous of degree $1$ and $F_3\in \k[w,y]$ is homogeneous of degree $3$. Then:
	\begin{enumerate}[leftmargin=*]
		\item\label{Lb2square}
		The polynomial $b_2\in \k[x,z]$ is a square;
		\item\label{La2b2}
		The polynomials $a_2, b_2\in \k[x,z]$ have a common linear factor;
		\item\label{La2e1}
		If $b_2 =0$, then $a_2, e_1\in \k[x,z]$ have a common linear factor;
		\item\label{Ld1a2_common_lin_fact}
		If $b_2 = e_1 = 0$ and $a_2$ is a square, 
		then the polynomials $a_2, d_1 \in \k[x, z]$ have a common linear factor;
		\item\label{La2b2d1e1zero_c1non-zero}
		If $a_2 = b_2 = d_1 = e_1 = 0$ and $\deg(f) \geq 2$, then $c_1 \neq 0$.
	\end{enumerate}
	Under the additional assumption that $\deg(f) = 3$, we have: 
	\begin{enumerate}[leftmargin=*]
		\setcounter{enumi}{5}
		\item\label{Lb_2e_1_zero_a2square}
		If $b_2 = e_1 = 0$, then the polynomial $a_2\in \k[x,z]$ is a square;
		\item\label{Lb_2e_1_zero_a2b1c1no_comm_div}
		If $b_2 = e_1 = 0$ and $(a_2, d_1) \neq (0, 0)$, then $\gcd(a_2, c_1, d_1) = 1$;
		\item\label{La2notsquareb2e1notzero}
		If $a_2$ is not a square, then $b_2\not=0$ or $e_1\not=0$;
	\end{enumerate}
\end{lemma}

\begin{proof}The fact that $F(0,x,0,z)=0$ implies that $F$ can be written in the above form. 
	Note that $F=F_1+F_2+F_3$, where $F_1=wa_2(x,z)+yb_2(x,z)$, $F_2=w^2c_1(x,z)+wy d_1(x,z)+y^2e_1(x,z)$ and $F_3$ are homogeneous in $w,y$ of degree 
	$1$, $2$ and $3$, respectively. 
	It remains to see that the above eight assertions hold. 
	
	First, we assume that $w$ divides $F$. Then $\deg(f)<3$ and $b_2=e_1=0$, 
	so~\ref{Lb2square},~\ref{La2b2} and~\ref{La2e1} hold. If in addition $a_2$ is a square and if
	$a_2$ and $d_1$ would
	have no common non-zero linear factor, then
	the homogeneous part of $f$ of degree $2$ would be 
	$f_2 = a_2 + y(d_1 + \lambda y)$ for some $\lambda \in \k$. 
	As $a_2$ is a square, we may apply a linear coordinate
	change in $x, z$ and assume that $a_2 = x^2$. 
	We then write
	$d_1 = d_{1, 0} x + d_{1, 1} z$ with $d_{1,0},d_{1,1}\in \k$, and obtain
	\[
	f_2 = x^2 + d_{1, 0}yx + y(\lambda y + d_{1, 1}z) \, .
	\]
	Since $d_{1, 1} \neq 0$, the polynomial $f_2 \in \k[x, y, z]$ 
	is irreducible (e.g. by the Eisenstein criterion)
	which contradicts Proposition~\ref{Proposition:XxpqisA2smalldegreelistpqDeg3}
	and therefore~\ref{Ld1a2_common_lin_fact} holds.
	If $a_2 = d_1 = 0$ and 
	$\deg(f) \geq 2$, then $c_1 \neq 0$, since otherwise $f \in \k[y]$ would not be irreducible.
	Hence,~\ref{La2b2d1e1zero_c1non-zero} holds.
	
	We may now assume that $w$ does not divide $F$, which implies that $\deg(f)=3$.
	
	We observe that the group of affine automorphisms 
	$G \subset \Aff(\AA^3) \subset \Aut(\PP^3)$ which preserve the line 
	$L = \set{[w:x:y:z] \in \PP^3}{w=y=0}$ 
	is generated by the following two subgroups:
	\[
	\begin{array}{rcl}
	G_1 &=& 
	\Bigset{\varphi_{\alpha, \beta, \gamma, \delta} \in \Aut(\PP^3)}{ 
		\begin{bmatrix} 
		\alpha & \beta \\ 
		\gamma & \delta
		\end{bmatrix} 
		\in \GL_2(\k) } \\
	G_2 &=& 
	\Bigset{\psi_{\varepsilon, \tau_1,\tau_2,\tau_3, \xi_1, \xi_3} \in \Aut(\PP^3)}{ 
		\varepsilon \in \k^*,\tau_1,\tau_2,\tau_3, \xi_1, \xi_3 \in \k}
	\end{array}
	\]
	where
	\[
	\begin{array}{ccc}
	\PP^3 &\stackrel{\varphi_{\alpha, \beta, \gamma, \delta}}{\xrightarrow{\hspace*{1.5cm}} }& \PP^3 \\
	{[w:x:y:z]} &\xmapsto{\hspace*{1.5cm}}& [w:\alpha x+\beta z: y:\gamma x+\delta z]
	\end{array}
	\]
	and
	\[
	\begin{array}{ccc}
	\PP^3 &
	\stackrel{\psi_{\varepsilon, \tau_1,\tau_2,\tau_3, \xi_1, \xi_3}}{\xrightarrow{\hspace*{2.5cm}}}
	& \PP^3 \\
	{[w:x:y:z]} &\xmapsto{\hspace*{2.5cm}}& 
	[w:x +\xi_1 y+ \tau_1w: \varepsilon y+\tau_2w:z+\xi_3 y+\tau_3w] \, .
	\end{array}
	\]
	Indeed, this follows from the facts that 
	the action of $G$ on $L$ gives a group homomorphism $G\to \Aut(L)\simeq \PGL_2(\k)$ that is surjective on $G_1$, and that the kernel is generated by $G_2$ and the homotheties of $G_1$.
	The fact that all assertions~\ref{Lb2square}-\ref{La2notsquareb2e1notzero} hold
	is preserved under elements of $G_1$ and $G_2$.
	We may thus assume that $f$ is of the form given in Corollary~\ref{cor:Degree3onepointoneline} and we check that the
	assertions~\ref{Lb2square}-\ref{La2notsquareb2e1notzero} are satisfied.
	
	In case~\ref{LFirstCasekyDeg3}, 
	$(a_2,b_2, c_1, d_1, e_1)=(\lambda z^2,\mu z^2, x, \varepsilon z, \nu z)$ for some 
	$\lambda,\mu,\nu, \varepsilon \in \k$.
	
	In case~\ref{LSecondCasexyDeg3y}, 
	$(a_2,b_2, c_1, d_1, e_1)=(0,\mu z^2, z, x, \nu z)$ for some $\mu,\nu \in \k$.
	
	In case~\ref{LSecondCasexyDeg3z}, $f=xz+yz(\lambda y+\mu z)+y+\delta z$ where 
	$\lambda,\mu,\delta\in \k$ and $(\lambda,\mu)\not=(0,0)$, so 
	$(a_2,b_2, c_1, d_1 ,e_1)=(xz,\mu z^2, \delta z, 0, \lambda z)$.
	
	In case~\ref{LSecondCasexy2Deg3}, $f = xy^2 + y(z^2 + az +b) + z$ for some $a, b \in \k$, so
	$(a_2,b_2, c_1, d_1, e_1)=(0, z^2, z, az, x)$.
	
	In each case, $b_2$ is a square, and there is a linear factor that divides $a_2$, $b_2$ 
	and a linear factor that divides $a_2$, $e_1$. 
	Moreover, $a_2$ is not a square only in case~\ref{LSecondCasexyDeg3z} 
	and thus $b_2$ or $e_1$ is non-zero. This shows that~\ref{Lb2square},~\ref{La2b2},~\ref{La2e1}
	and~\ref{La2notsquareb2e1notzero} are satisfied. 
	The equalities $a_2 =b_2=d_1=e_1=0$ are only possible in case~\ref{LFirstCasekyDeg3}, where $c_1=x\neq 0$, 
	thus~\ref{La2b2d1e1zero_c1non-zero} is satisfied. The equalities $b_2 = e_1 = 0$ are only possible in the cases~\ref{LFirstCasekyDeg3} and~\ref{LSecondCasexyDeg3y};
	and then $a_2$, $d_1$ have a common non-zero linear factor, $a_2$ is a square,
	and if $(a_2, d_1) \neq (0, 0)$, then $\gcd(a_2, c_1, d_1) = 1$. 
	Thus~\ref{Ld1a2_common_lin_fact},~\ref{Lb_2e_1_zero_a2square} and~\ref{Lb_2e_1_zero_a2b1c1no_comm_div} are satisfied. 
\end{proof}

\begin{proposition}
	\label{prop.standard_form_if_line_at_infinity}
	Let $f_1, \ldots, f_n \in \k[x, y, z]$ be polynomials and 
	assume that $f = (f_1, \ldots, f_n) \colon \AA^3 \to \AA^n$ is
	an affine linear system of affine spaces of degree $3$ such that $y$ divides the homogeneous parts of degree $3$ of $f_1, \ldots, f_n$. Then, the following hold:
	\begin{enumerate}[$(i)$]
	\item
	Either $f$ is equivalent to a linear system of affine spaces in standard form, or $\car(\k)=2$ and 
	 $f$ is equivalent to $(x+ z^2 +y^3, y + x^2) \colon \AA^3 \to \AA^2$.
	\item\label{rem.All_b_i.2_are collinear}
	Writing the homogeneous
	part of degree $3$ of $f_i$ as $y(\eta_i y^2 + y e_{i, 1}+ b_{i, 2})$ where $\eta_i \in \k$ and
	$e_{i, 1}, b_{i, 2} \in \k[x, z]$ are homogeneous of degree $1$ and $2$, the polynomials $b_{1, 2}, \ldots, b_{n, 2}$ are collinear.
	\end{enumerate}
\end{proposition}
\begin{proof}
For each $i$ we denote by $F_i\in \k[w,x,y,z]$ a homogeneous polynomial of degree $3$ such that $f_i=F_i(1,x,y,z)$ and write it as
	\[wa_{i,2}(x,z)+yb_{i,2}(x,z)+w^2c_{i,1}(x,z)+wyd_{i,1}(x,z)+y^2e_{i,1}(x,z)+F_{i,3}(w,y)\]
	where $a_{i,2},b_{i,2}\in \k[x,z]$ are homogeneous of degree $2$, $c_{i,1},d_{i,1},e_{i,1}\in \k[x,z]$ are homogeneous of degree $1$, $F_{i,3}\in \k[w,y]$ is homogeneous of degree $3$, and 
	the following hold for all $(\lambda_1, \ldots, \lambda_n) \in \k^n$ (see Lemma~\ref{Lem:AlongALine}):
	\begin{enumerate}[leftmargin=*]
		\item[\ref{Lb2square}]
		$\sum \lambda_i b_{i,2}(x,z)$ is a square;
		\item[\ref{La2b2}]
		$\sum \lambda_i a_{i,2}(x,z)$ and $\sum \lambda_i b_{i,2}(x,z)$ have a common non-zero linear factor;
		\item[\ref{La2e1}]
		If $\sum \lambda_i b_{i,2}(x,z) = 0$, then $\sum \lambda_i a_{i,2}(x,z)$ and $\sum \lambda_i e_{i,1}(x,z)$ have a common non-zero linear factor;
		\item[\ref{Ld1a2_common_lin_fact}]
		If $\sum \lambda_i b_{i, 2}(x, z) = \sum \lambda_i e_{i, 1}(x, z) = 0$ and $\sum \lambda_i a_{i, 2}(x, z)$ 
		is a square, then $\sum \lambda_i a_{i, 2}(x, z), 
		\sum \lambda_i d_{i, 1}(x, z)$ have a common non-zero linear factor;
	\end{enumerate}
	and if $\deg(\sum \lambda_i f_i)= 3$, then: 	
	\begin{enumerate}[leftmargin=*]
		\item[\ref{La2notsquareb2e1notzero}]
		If $\sum \lambda_i a_{i,2}(x,z)$ is not a square, then $\sum \lambda_i b_{i,2}(x,z)\not=0$ or $\sum \lambda_i e_{i,1}(x,z)\not=0$.
	\end{enumerate}
	We distinguish, whether all $b_{i, 2}$ are collinear (case (A)) or not (case (B)). 
	It turns out that in fact case (B) cannot occur, which proves~\ref{rem.All_b_i.2_are collinear}.

	(A): Any two $b_{i,2}$ are collinear:
	After applying an element of $\GL_2(\k)$ on $x,z$, we may assume that $z^2$ divides all $b_{i,2}$ 
	by assertion~\ref{Lb2square}. If $z$ divides each $a_{i,2}$, the point $[0:1:0:0]$ will be a 
	singular point of the hypersurface in $\PP^3$ given by $F_i$ for each $i$, 
	so $f$ is in standard form. We may thus 
	assume that there is $j$ such that $z$ does not divide $a_{j,2}$. 
	Assertion~\ref{La2b2} then implies that $b_{i,2}=0$ for each $i$.
	
	If a linear factor divides all $a_{i,2}$, we apply an element of $\GL_2$ on $x,z$ and assume that $z$ divides all $a_{i,2}$, giving again that $f$ is in a standard form. We then assume that no linear factor divides all $a_{i,2}$. In particular, $\dim \Span_{\k}(a_{1, 2}, \ldots,a_{n, 2})\geq 2$.
	
	We assume that each $\sum \lambda_i a_{i, 2}$ is a square, which implies that $\car(\k) = 2$ and
	$\Span_{\k}(a_{1, 2}, \ldots,a_{n, 2})= \k x^2 \oplus \k z^2$.
	By assertion~\ref{La2e1}, we can
	apply Lemma~\ref{lem.2nd_power} in order to get $e_{i, 1} = 0$ for each $i =1, \ldots, n$.
	Then, by assertion~\ref{Ld1a2_common_lin_fact} we can apply Lemma~\ref{lem.2nd_power} once again
	and get $d_{i, 1} = 0$ for $i=1, \ldots, n$. Hence, the result follows from Lemma~\ref{lem.char2}.
	
	We now assume that $\sum \lambda_i a_{i, 2}$ is not a square for general $(\lambda_1,\ldots,\lambda_n)\in \k^n$. Assertion~\ref{La2notsquareb2e1notzero} implies that $\sum \lambda_i e_{i, 1}$ is a non-zero linear polynomial 
	 for general $(\lambda_1, \ldots, \lambda_n) \in \k^n$, 
	 which then needs to divide $\sum \lambda_i a_{i, 2}$ by Assertion~\ref{La2e1}. As no linear factor divides all $a_{i,2}$, we may apply a general element of $\GL_n$ at the target and may assume 
	 that $a_{1,2}$ and $a_{2,2}$ have no common factor, and then the same holds for $e_{1,1}$ and $e_{2,1}$
	 (as $e_{i, 1}$ divides $a_{i, 2}$ for $i=1, 2$).
	 We then apply $\GL_2$ on $x, z$ at the source to get $e_{1,1}=x$ and $e_{2,1}=z$. We get $a_{1,2}=x(\alpha x+\beta z)$, $a_{2,2}=z(\gamma x +\delta z)$ for some $\alpha,\beta,\gamma,\delta\in \k$. For each $\lambda\in \k$, the polynomial $e_{1,1}+\lambda e_{2,1}=x+\lambda z$ divides $a_{1,2}+\lambda a_{2,2}=\alpha x^2+(\beta+\lambda \gamma) xz+\delta\lambda z^2 $, so replacing $x=\lambda$ and $z=-1$ gives $0= \lambda^2(\alpha-\gamma)+(\delta-\beta)\lambda$. This being true for all $\lambda$, we obtain $\alpha=\gamma$ and $\beta=\delta$, contradicting the fact that $a_{1,2}$ and $a_{2,2}$ 
	 have no common factor.

	(B): It remains to suppose that not all $b_{i, 2}$, $i=1, \ldots, n$ are collinear and to derive a contradiction. 
	Since by assertion~\ref{Lb2square} each $\sum \lambda_i b_{i, 2}$ is a square, 
	we get $\car(\k) = 2$.
	After applying a linear automorphism at the target, we may assume that
	$b_{1, 2} = z^2$ and $b_{2, 2} = x^2$. According to~\ref{La2b2},
	we can apply Lemma~\ref{lem.2nd_power} and get
	$a \in \k$ with $a_{1, 2} = a z^2$, $a_{2, 2} = a x^2$. Replacing $y$ with $y +a$ at the source, we may assume $a=0$. 
	This gives
	 \[
	 	f_1 = yz^2 + \alpha x +\beta z +\varepsilon \quad \textrm{and} \quad f_2 = yx^2 +\gamma x +\delta z + \nu
	 \]
	 where $\alpha,\beta,\gamma,\delta,\varepsilon,\nu\in \k[y]$ 
	 (the first four of degree $\le 2$ and the last two of degree $\le 3$). For each $\lambda\in \k$, the polynomial $f_1+\lambda^2 f_2=y(z+\lambda x)^2+(\alpha+\lambda^2\gamma) x +(\beta+\lambda^2\delta) z+\varepsilon+\lambda^2\nu$ defines an $\A^2$ in $\A^3$. Replacing $z$ with $z +\lambda x$, the polynomial 
	 \[
	 	R_\lambda=yz^2+(\alpha + \lambda \beta+\lambda^2\gamma + \lambda^3\delta) x +(\beta+\lambda^2\delta) z+\varepsilon+\lambda^2\nu
	 \]
	 defines an $\A^2$ in $\A^3$. 
	 Let us write $p_\lambda=\alpha + \lambda\beta+\lambda^2\gamma + \lambda^3\delta\in \k[y]$.
	 
	 Let us write $\alpha=\sum_{i\ge 0} \alpha_i y^i$, $\beta=\sum_{i\ge 0} \beta_i y^i$, $\gamma=\sum_{i\ge 0} \gamma_i y^i$, $\delta=\sum_{i\ge 0} \delta_i y^i$, where $\alpha_i,\beta_i,\gamma_i,\delta_i\in \k$ for each $i\ge 0$. If there is some $i\ge 0$ such that the coefficient of $y^i$ of $p_\lambda$ is zero for a general (or equivalently for all) $\lambda\in \k$, then $\alpha_i +\lambda \beta_i+\lambda^2\gamma_i+\lambda^3\delta=0$ for each $\lambda\in \k$, so $\alpha_i=\beta_i=\gamma_i=\delta_i=0$.
	 
	 Suppose first that $p_\lambda\in \k[y]\setminus \k$ for a general $\lambda\in \k$. In this case, we may apply Proposition~\ref{Prop:XfibratReminder}: writing $R_\lambda= xp_\lambda(y)+q_\lambda(y,z)$ with $q_\lambda\in \k[y,z]$, the polynomial $q_\lambda(y_0,z)\in \k[z]$ is of degree $1$
	 for each root $y_0\in\k$ of $p_\lambda$. 
	 As the coefficient of $z^2$ in $q_\lambda(y,z)$ is $y$, we find that $0$ is the only possible root of $p_\lambda(y)$, and in fact is a root for a general $\lambda$, as we assumed $p_\lambda\in \k[y]\setminus \k$. Applying the above argument with $i=0$ implies that $\alpha_0=\beta_0=\gamma_0=\delta_0=0$, but then, for each $\lambda\in \k$ the polynomial $\beta+\lambda^2\delta$ is zero at $y=0$, so $q_\lambda(0,z)\in \k[z]$ is not of degree $1$.
	 
	 The last case is when $p_\lambda\in \k$ for each $\lambda\in \k$. This implies (again by the above argument) that $\alpha_i=\beta_i=\gamma_i=\delta_i=0$ for each $i\ge 1$, so $\alpha,\beta,\gamma,\delta\in \k$. We have $\delta \neq 0$, since otherwise
	 $f_2 \in \k[x, y]$ would define in $\AA^2_{x, y}$ 
	 a curve with two points at infinity. There exists thus $\lambda\in \k$ such that $p_\lambda=0$, so $R_\lambda$ does not define an $\A^2$ (it belongs to $\k[y,z]$ and the curve that it defines in $\A^2_{y,z}$ has two points at infinity).
	\end{proof}

\subsection{Reduction to affine linear systems of affine spaces in standard form}
\begin{proposition}
	\label{prop:globally_xp_i+q_i}
	Let $n\ge 1$ and let $f_1,\ldots,f_n\in \k[x,y,z]$ be polynomials of degree $\leq 3$
	such that $f = (f_1, \ldots, f_n) \colon \AA^3 \to \AA^n$ is a linear system
	of affine spaces. Then either $f$ is equivalent to a linear system of affine spaces in standard form,
	or $f$ is equivalent to one of the following linear systems of affine spaces:
	\begin{enumerate}
		\item \label{eq.special_char_2} 
		$(x+ z^2 +y^3, y + x^2) \colon \AA^3 \to \AA^2$ where $\car(\k) = 2$,
		or
		\item \label{eq.special_char_3} 
		$(x + z^2 + y^3, z + x^3) \colon \AA^3 \to \AA^2$ where $\car(\k) = 3$.
	\end{enumerate}
\end{proposition}

\begin{remark}
	The families of linear systems of affine spaces in~\ref{eq.special_char_2} and~\ref{eq.special_char_3}
	from Proposition~\ref{prop:globally_xp_i+q_i} are the linear systems of affine spaces
	from Lemmata~\ref{lem.Lin_system_char_2} and~\ref{lem.Lin_system_char_3}. In particular, 
	the linear systems of affine spaces in~\ref{eq.special_char_2} and~\ref{eq.special_char_3} are all
	non-equivalent to linear systems of affine spaces in standard form.
\end{remark}

\begin{proof}[Proof of Proposition~$\ref{prop:globally_xp_i+q_i}$]
	If $n=1$, the result follows from Corollary~\ref{cor.cubic_surf_iso_to_A2_singular}, so we will assume that $n\ge 2$. By Lemma~\ref{lem.properties_lin_systems_of_aff_spaces}\ref{lem.properties_lin_systems_of_aff_spaces-1},
	we get $n \leq 3$.
	
	Let $d = \deg(f)$. Since the statement
	holds when $d = 1$, we assume $d \in \{2, 3\}$.
	
	Let $f_{i, j} \in \k[x, y, z]$ be the homogeneous part of degree $j$ of $f_i$ for $i=1, \ldots, n$, and let us define $V_j=\Span_\k(f_{1,j},\ldots,f_{n,j})\subseteq \k[x,y,z]_j$ for each $j\le d$.
	
	First, we consider the case $d = 2$. Due to Corollary~\ref{Cor:Hypersurf_degree_2}, each element
	in $V_2$ is reducible and due to Lemma~\ref{LinearSystemLinearFactors}
	one of the following cases occur:
	\begin{itemize}[leftmargin=*]
		\item There exists $h \in \k[x, y, z]_1$ which divides each element of $V_2$;
		\item $V_2 \subset \k[s, t]$ for linearly independent $s, t \in \k[x, y, z]_1$;
		\item $\car(\k) = 2$ and $V_2 = \k x^2 \oplus \k y^2 \oplus \k z^2$.
	\end{itemize}
	In the first case we may assume that $h = y$ and in the second case we may
	assume that $(s, t) = (y, z)$, so $f$ is in standard form in both cases. 
	If we are in the last case, then 
	$n = 3$ and we may assume that $f_{1, 2} = x^2$, $f_{2, 2} = y^2$, $f_{3, 2} = z^2$.
	Due to Lemma~\ref{lem.almost_eigenvectors} there exists 
	$(\lambda_1, \lambda_2, \lambda_3) \neq (0, 0, 0)$ and $\varepsilon \neq 0$ such that
	$\sum \lambda_i^2 f_{i, 1} = \varepsilon (\lambda_1 x + \lambda_2 y + \lambda_3 z)$ and hence
	we get a contradiction to the irreducibility of $\sum \lambda_i^2 f_i$. 
	
	It remains to do the case where $d = 3$. 
	If a linear factor or an irreducible polynomial of degree $2$ divides all elements of $V_3$, the result follows respectively from Proposition~\ref{prop.standard_form_if_line_at_infinity} (after applying an element of $\GL_3$ at the source) and 
	Proposition~\ref{prop.standard_form_if_conic_at_infinity}. 
	By Corollary~\ref{cor.Properties_of_hypersurfaces_at_infinity}, no element of $V_3$ is irreducible, so we may assume that $\gcd(V_3)=1$. In particular, $\dim V_3 \geq 2$.
	
	If each element of $V_3$ is a third power,
	then $\car(\k) = 3$ and the result follows from Lemma~\ref{lem.char3}.
	Thus we may assume that a general element in $V_3$ is not a third power. Now, Lemma~\ref{LinearSystemLinearFactors}
	implies that there exist linearly independent
	$s, t \in \k[x, y, z]_1$ such that $V_3 \subset \k[s, t]$. We may assume that $(s, t) = (y, z)$.
	As a general element of $V_3$ is not a third power, then by Corollary~\ref{cor.Properties_of_hypersurfaces_at_infinity}\ref{cor.PHI_distinct_lines} 
	the closure of the cubic $\sum \lambda_i f_i=0$ in $\p^3$ has a singularity at $[0:1:0:0]$
	for general $(\lambda_1,\ldots,\lambda_n) \in \k^n$ and thus
	$f$ is in standard form.
\end{proof}

\begin{corollary}
	\label{cor.Linear_systems_in_standard_form}
	Let $1 \leq n \leq 3$ and let 
	$f_1,\ldots,f_n\in \k[x,y,z]$ be polynomials of degree $\leq 3$
	such that $f = (f_1, \ldots, f_n) \colon \AA^3 \to \AA^n$ is a trivial $\AA^{3-n}$-bundle.
	Then $f$ is equivalent to a linear system of affine spaces in standard form.
\end{corollary}

\begin{proof}
	This follows directly from Proposition~\ref{prop:globally_xp_i+q_i}, since
	the linear systems of affine spaces from Lemma~\ref{lem.Lin_system_char_2} and
	Lemma~\ref{lem.Lin_system_char_3}
	are not trivial $\AA^1$-bundles.
\end{proof}

%

\subsection{Study of affine linear systems of affine spaces $\AA^3 \to \AA^2$ in standard form}

Towards the description of the automorphisms of degree $\leq 3$, we study 
in this subsection certain 
affine linear systems of affine spaces $(f_1, f_2) \colon \AA^3 \to \AA^2$ in standard form, i.e.~such that
$f_i = xp_i + q_i$ for $i=1,2$, with $p_i, q_i \in \k[y, z]$.

\begin{lemma}\label{Lem:Noptildenoaut}
	For $i=1, 2$, let $p_i ,q_i \in \k[y, z]$ such that 
	$(xp_1 + q_1, x p_2 + q_2)$ is a linear system of affine spaces. Then,
	$\k[p_1, p_2]\not=\k[y,z]$, i.e.~$(p_1, p_2) \colon \AA^2 \to \AA^2$ is not an automorphism.
\end{lemma}	
\begin{proof}
	If $\k[p_1,p_2]=\k[y,z]$, then
	we apply a (possibly non-affine) 
	automorphism of $\k[y,z]$ and may assume that $p_1 = y$, $p_2 = z$. We choose $\alpha,\beta,\gamma,\delta,\varepsilon,\tau\in \k$ such that 
	\[
		q_1(y,z)=\alpha y+\beta z+\varepsilon\mod (y^2,yz,z^2) \, ,\quad q_2(y,z)=\gamma y+\delta z+\tau\mod (y^2,yz,z^2)
		\, .
	\]
	Proposition~\ref{Prop:XfibratReminder} implies that 
	$q_1(0,z)\in \k[z]$ and $q_2(y,0)\in \k[y]$ have degree $1$, so $\beta,\gamma\in \k^*$. 
 	For each $\lambda \in \k$, the polynomial in 
 	$(xy+ q_1) - \lambda(xz + q_2) = x ( y - \lambda z) + ( q_1 - \lambda q_2) \in \k[x, y, z]$
	defines an $\AA^2$ in $\AA^3$. Replacing $y$ with $y+\lambda z$, the polynomial 
	\[
		R_\lambda=xy+q_1(y+\lambda z,z)- \lambda q_2(y+\lambda z,z)
	\]
	defines an $\A^2$ in $\A^3$. 
	Proposition~\ref{Prop:XfibratReminder} implies that $R_\lambda(x,0,z)=R_\lambda(0,0,z)\in \k[z]$ is of degree $1$, for each $\lambda\in \k$. However, 
	\[R_\lambda(0,0,z)=q_1(\lambda z,z)-\lambda q_2(\lambda z,z)=\alpha \lambda z+\beta z+\varepsilon-\lambda (\gamma \lambda z+\delta z+\tau)\pmod{z^2}\]
	and as $\gamma\not=0$, there is $\lambda\in \k$ such that the coefficient of $z$ of $R_\lambda(0,0,z)$ is 
	zero, contradiction.
\end{proof}
\begin{lemma}
	\label{lem.Lin_System_p_in_k[y]}
	For $i = 1, 2$, let $p_i \in \k[y]$ and $q_i \in \k[y, z]$ and assume that
	$f = (f_1, f_2) = (xp_1 + q_1, xp_2 + q_2) \colon \AA^3 \to \AA^2$ is an affine 
	linear system of affine spaces. Then the following hold:
	\begin{enumerate}[leftmargin=*]
		\item \label{lem.Lin_System_p_in_k[y]_0}
				If $p_1$ and $p_2$ have a common root, then they are linearly dependent.
		\item \label{lem.Lin_System_p_in_k[y]_1} 
				If $p_1 \not\in \k$ and $p_2 = 0$, then $q_2 \in \k[y]$ and $\deg(q_2) = 1$.
		\item \label{lem.Lin_System_p_in_k[y]_2} 
				If $p_1 = y$ and $q_1 = a y + z r_1 + r_0$ for $a \in \k[y, z]$, $r_1 \in \k^\ast$, $r_0 \in \k$
				and if $p_2 = 1$, 
				then $a-q_2 \in \k[y]$.
		\item \label{lem.Lin_System_p_in_k[y]_3} 
				If $p_1 = y^2$ and $q_1 = y s(z) + z$ for some $s \in \k[z]$ and $\deg(f) \leq 3$, then:
				\begin{enumerate}
					\item \label{lem.Lin_System_p_in_k[y]_3i} If $p_2 = 1$, then $s \in \k$ and $q_2 \in \k[y]$.
					\item \label{lem.Lin_System_p_in_k[y]_3ii} If $p_2 = y+1$, then $s = -z + b$ and $q_2 = -z + r$ for some $b \in \k$
							and $r \in \k[y]$ with $\deg(r) \leq 3$.
				\end{enumerate}
		\item \label{lem.Lin_System_p_in_k[y]_4} 
				If $p_1 = y(y+1)$ and $q_1 = s(y)z + t(y)$ for
				$s, t \in \k[y]$ of degree $\leq 1$ and $p_2 =1$, then $s \in \k^\ast$ and $q_2 \in \k[y]$.
	\end{enumerate}
\end{lemma}

\begin{proof}
	By assumption for each $(\lambda, \mu) \neq (0,0)$, the equation
	\[
		\lambda f_1 + \mu f_2 = x (\lambda p_1 + \mu p_2) + \lambda q_1 + \mu q_2 = 0
	\]
	defines an $\AA^2$ in $\AA^3$. Hence, by Proposition~\ref{Prop:XfibratReminder}, for each
	$y_0 \in \k$ the following holds:
	\begin{align*}
		\label{eq.assertionStar}
		\tag{$\ast$}
		&\textrm{if $\lambda p_1(y_0) + \mu p_2(y_0) = 0$ and $\lambda p_1 + \mu p_2 \neq 0$,} \\
		&\textrm{then the degree of $\lambda q_1(y_0, z) + \mu q_2(y_0, z) \in \k[z]$ is $1$} \, .
	\end{align*}
	We will use this fact constantly, when we consider the cases~\ref{lem.Lin_System_p_in_k[y]_0}-\ref{lem.Lin_System_p_in_k[y]_4}.
	
	\ref{lem.Lin_System_p_in_k[y]_0}: 
	After an affine coordinate change in $y$, we may assume that $y$ divides $p_1$ and $p_2$.
	By Proposition~\ref{Prop:XfibratReminder} it follows that $q_i(0, z)$ is a polynomial of degree $1$
	in $z$ for $i = 1, 2$. Hence there exists $\mu\in \k$ such that $q_1(0, z)- \mu q_2(0, z)$ is constant. 
	This, together with \eqref{eq.assertionStar}, implies that $ p_1 = \mu p_2$.
	
	\ref{lem.Lin_System_p_in_k[y]_1}: Since $p_1 \not\in \k$, there exists $\gamma \in \k$
	with $p_1(\gamma) = 0$. After applying
	an affine coordinate change in $y$, we may assume that $\gamma = 0$.
	By \eqref{eq.assertionStar}, the degree of $q_1(0,z)+\mu q_2(0,z)\in \k[z]$ is $1$ for each $\mu\in \k$, 
	so $q_2(0, z) \in \k$.
	Hence, $y$ divides $q_2 - q_2(0, 0)$ in $\k[y, z]$.
	Since $q_2 - q_2(0, 0) = 0$ defines an $\AA^2$ in $\AA^3$, the polynomial $q_2 -q_2(0, 0)$ 
	is irreducible and thus $q_2 = \alpha y + q_2(0, 0)$ for some $\alpha \in \k^\ast$.
	
	\ref{lem.Lin_System_p_in_k[y]_2}:
	Choosing $(\lambda,\mu)=(1,-\eta)$ for some $\eta\in \k$, we get
	$\lambda p_1 + \mu p_2 = y - \eta$. Thus by~\eqref{eq.assertionStar}, the degree of the polynomial
	\[
		\eta a(\eta, z) + r_1 z + r_0 - \eta q_2(\eta, z)
		= r_1 z + r_0 + \eta (a(\eta, z)- q_2(\eta, z)) \in \k[z]
	\]
	is $1$ for each $\eta\in \k$. This implies that $a(\eta, z)- q_2(\eta, z) \in \k[\eta]$.
	
	\ref{lem.Lin_System_p_in_k[y]_3}\ref{lem.Lin_System_p_in_k[y]_3i}: 
	Choosing $(\lambda,\mu)=(1,-\eta^2)$, we get $\lambda p_1 + \mu p_2 = (y- \eta)(y + \eta)$.
	By~\eqref{eq.assertionStar} it follows that for all $\eta \in \k$ the degree of
	\[
		\eta s(z) + z - \eta^2 q_2(\eta, z) \in \k[z]
	\]
	is $1$, i.e.~$\eta s(z) + z - \eta^2 q_2(\eta, z) = \alpha z + \beta$ 
	for some $\alpha \in \k^{\ast}$ and $\beta \in \k[\eta]$. 
	In order to use
	\[
		\label{eq.deomp}
		\tag{$\ast\ast$}
		z \k[z] \oplus \k \oplus \eta \k[z] \oplus \eta^2 \k[\eta, z] = \k[\eta, z] \, ,
	\]
	we write $\beta=\beta_0+\eta\beta_1 +\eta^2\beta_2$ where $\beta_0,\beta_1\in \k$, $\beta_2\in \k[\eta]$ and get 
	\[
		(z,0,\eta s(z),-\eta^2q_2(\eta,z))=(\alpha z,\beta_0,\eta\beta_1,\eta^2\beta_2) \, , 
	\] 
	so $s=\beta_1\in \k$ and $q_2(\eta,z)=-\beta_2\in \k[\eta]$.
	
	\ref{lem.Lin_System_p_in_k[y]_3}\ref{lem.Lin_System_p_in_k[y]_3ii}: 
	We now choose $(\lambda,\mu)=(1+\eta,-\eta^2)$ for some $\eta\in \k$ and obtain
	\[
		\lambda p_1 + \mu p_2 =
		(1 + \eta) y^2 - \eta^2 (y+1) = (y- \eta)((1 + \eta)y+\eta) \, .
	\]
	Due to~\eqref{eq.assertionStar}, for all $\eta \in \k$ the degree of the polynomial
	\[
	(1 + \eta)(\eta s(z) + z) - \eta^2 q_2(\eta, z)
	= z + \eta(s(z) + z) + \eta^2 (s(z) - q_2(\eta, z)) \in \k[z]
	\]
	is $1$. Writing this polynomial as above
	as $\alpha z+\beta_0+\eta\beta_1 +\eta^2\beta_2$ with $\alpha\in \k^*$, $\beta_0,\beta_1\in \k$, $\beta_2\in \k[\eta]$, the decomposition~\eqref{eq.deomp} gives
	\[
		(z,0,\eta (s(z) + z),\eta^2(s(z)-q_2(\eta,z)))=(\alpha z,\beta_0,\eta\beta_1,\eta^2\beta_2) \, ,
	\]
	so $s(z) + z =\beta_1\in \k$ and $s(z)- q_2(\eta, z)=\beta_2 \in \k[\eta]$.
	Choosing $b=\beta_1$ and $r\in \k[y]$ such that $\beta_2=b-r(\eta)$, 
	we obtain $s(z)=-z+b$ and $q_2(y, z) =s(z)-b+r(y)= -z + r(y)$. Since $\deg(q_2) \leq 3$
	it follows that $\deg(r)\leq 3$.
	
	\ref{lem.Lin_System_p_in_k[y]_4}: Let $(\lambda, \mu)=(1,-\eta(\eta+1))$. Then 
	\[
		\lambda p_1 + \mu p_2 = y(y+1)- \eta(\eta + 1) = 
		(y-\eta)(y + \eta +1)
	\]
	Due to~\eqref{eq.assertionStar}, for all $\eta \in \k$, the degree of
	\[
		s(\eta)z + t(\eta) - \eta(\eta + 1) q_2(\eta, z) \in \k[z]
	\]
	is $1$. 
	This implies that the polynomial 
	\[
		h=s(\eta)z - \eta(\eta + 1) q_2(\eta, z)\in \k[\eta,z]
	\] 
	is of the form $\alpha z+\beta$ for some $\alpha \in \k^\ast$ and $\beta\in \k[\eta]$.
	
	When we write $q_2=\sum_{i\ge 0} q_{2, i}(y)z^i$ for $q_{2, i} \in \k[y]$, 
	we obtain $q_{2, i}=0$ for each $i\ge 2$ 
	(as $h$ has degree $1$ in $z$) 
	and $s(y)-y(y+1) q_{2, 1}(y)\in \k^\ast$. 
	As $\deg(s)\le 1$, this yields $q_{2, 1}=0$, and then $s(y)\in \k^*$. Moreover, 
	$q_2= q_{2, 0}(y)\in \k[y]$.
\end{proof}

\begin{lemma}
	\label{Lem:x(y+z^2)+z}
	Let $p, q \in \k[y, z]$ such that $\deg(p) \leq 1$ and $\deg(q) \leq 3$.
	Assume that $(x(y + z^2) + z, x p + q) \colon \AA^3 \to \AA^2$ is an affine linear system
	of affine spaces. Then
	\[
	p \in \k \, , \quad \textrm{$q = a \cdot (y + z^2) + b$ for some $a, b \in \k$} \quad \textrm{and} \quad
	(p, a) \neq (0, 0) \, .
	\]
\end{lemma}
\begin{proof}
	Suppose first that $p\in \k$. When we write 
	$r=q(y-z^2,z)\in \k[y,z]$, we obtain $q=r(y+z^2,z)$.
	For each $\lambda\in \k$, the polynomial 
	\[
		x(y+z^2)+z -\lambda(xp+q)=x(y + z^2 - \lambda p) + z - \lambda r(y+z^2,z)
	\] 
	defines an $\A^2$ in $\A^3$, so the same holds for $xy+z-\lambda r(y+\lambda p,z)$. By Proposition~\ref{Prop:XfibratReminder}, the polynomial $z - \lambda r(\lambda p,z)\in \k[z]$ is of degree $1$ for each $\lambda\in \k$. 
	This implies that the polynomial $r(\lambda p,z)\in \k[\lambda,z]$ lies in $\k[\lambda]$. As $p\in \k$, 
	either $p \neq 0$ and $r(y,z)\in \k[y]$ or $p=0$ and $r(y,z)\in \k+y\k[y,z]$. The first case yields $q\in \k[y+z^2]$, so $q=a \cdot (y + z^2) + b$ for some $a, b \in \k$, since $\deg q\le 3$. In the second case, we write $b=q(0,0)$ and obtain that $q-b$ is irreducible, as it defines the preimage of the hyperplane $y=b$. Hence, $r(y,z)-b\in y\k[y,z]$ is irreducible, so equal to $ay$ for some $a\in \k^*$. As before we get $q=a \cdot (y + z^2) + b$. In both cases $(p, a) \neq (0,0)$.
	
	It remains to see that $p\not\in \k$ is impossible. We write $p=ay+bz+c$ for some $(a,b)\in \k^2\setminus \{(0,0)\}$ and $c\in \k$. If $a=0$, then $b\not=0$ which yields $\k[y+z^2,p]=\k[y+z^2,z]=\k[y,z]$, impossible by Lemma~\ref{Lem:Noptildenoaut}. We may thus assume that $a=1$. We write $q=r+z+\mu$, with $\mu\in \k$ and $r\in \k[y,z]$ such that $r(0,0)=0$. For each $\lambda\in \k$, the polynomial 
	\[
	\lambda (x(y+z^2)+z)+(1-\lambda)(xp+q-\mu)=x(y + \lambda z^2+ (1-\lambda)(bz+c)) + z+(1-\lambda)r
	\]
	defines an $\A^2$ in $\A^3$, so the same holds for $xy+ z+ (1-\lambda)\cdot r(y-\lambda z^2+(\lambda-1)(bz+c),z)$. We again apply Proposition~\ref{Prop:XfibratReminder}, and find that $z+ (1-\lambda)\cdot r(-\lambda z^2+(\lambda-1)(bz+c),z)\in \k[z]$ is of degree $1$ for each $\lambda\in \k$, so the polynomial 
	\[
	R=r(-\lambda z^2+(\lambda-1)(bz+c),z)\in \k[\lambda,z]
	\] 
	is an element of $\k[\lambda]$ (independent of $z$). If $r(y,z) \not\in \k$, then $d := \deg_y(r) \geq 1$ and we may write $r = r_0(z) + r_1(z)y + \ldots + r_d(z) y^d$ 
	where $r_d \neq 0$. 
	Thus we get 
	\[
	R= r(\lambda(bz+c-z^2) - (bz + c),z) =\sum_{i=0}^d \lambda^i q_i
	\] 
	where $q_0,\ldots,q_{d}\in \k[z]$ and $q_d=(bz+c-z^2)^d r_d(z)\in \k[z]\setminus \k$. This contradicts $R\in \k[\lambda]$.
	Hence $r(y,z) \in \k$, so $r=r(0,0)=0$. 
	This proves that $q=z+\mu$. But this is impossible, as the zero locus of the polynomial
	$x(y + z^2) + z-( x p + q-\mu)=x(z^2-bz-c)$ is not isomorphic to $\A^2$ 
	(it is reducible).
\end{proof}

\subsection{Linear systems of affine spaces of degree $\leq 3$
in standard form}

We start with a lemma, which lists the possibilities for the polynomials 
$p_1, \ldots, p_n$ in case of a linear system of affine spaces $\AA^3 \to \AA^n$ of degree $\leq 3$
in standard from where the polynomials $p_1, \ldots, p_n$ lie in $\k[y]$.

\begin{lemma}
	\label{lem:factor_of_x}
	Let $n \geq 1$ and let $p_i \in \k[y]$, $q_i \in \k[y, z]$ for $i = 1, \ldots, n$ such that
	$f = (f_1, \ldots, f_n) = (x p_1 + q_1, \ldots, x p_n + q_n) \colon \AA^3 \to \AA^n$ 
	is a linear system of affine spaces of degree $\leq 3$. Let us assume that
	\[
	V := \Span_{\k}\{p_1, \ldots, p_n \} \subseteq \Span_{\k}\{1, y, y^2\} \, .
	\]
	Then, up to affine coordinate changes in $y$ at 
	the source, one of the following cases holds:
	\[
		\begin{array}{r | l | l}
			\textrm{case} & n & V \\
			\hline
			\hypertarget{factor_of_x_1}{(1)} & 2 \textrm{ or } 3 & \k (y+1) \oplus \k y^2 \\
			\hypertarget{factor_of_x_2}{(2)} & 2 \textrm{ or } 3 & \k \oplus \k y^2 \\
			\hypertarget{factor_of_x_3}{(3)} & 2 \textrm{ or } 3 & \k \oplus \k y(y+1) \\
			\hypertarget{factor_of_x_4}{(4)} & 2 \textrm{ or } 3 & \k \oplus \k y \\
			\hypertarget{factor_of_x_5}{(5)} & 1 \, , 2 \textrm{ or } 3 & \k \\
			\hypertarget{factor_of_x_6}{(6)} & 1 \textrm{ or } 2 & \k p \quad \textrm{where $p \in \{0, y, y^2, y(y+1) \}$}
		\end{array}
	\]
\end{lemma}

\begin{proof}
	We first prove that $\k y\oplus \k y^2$ is not contained in $V$. Indeed, we could then assume that $p_1 = y$ and $p_2 = y^2$, but then $(f_1, f_2)$ is not a linear system
	of affine spaces by Lemma~\ref{lem.Lin_System_p_in_k[y]}\ref{lem.Lin_System_p_in_k[y]_0}. This proves in particular that $\dim V\le 2$.
	
	Suppose now that $\dim V\le 1$. 	 If $n\le 2$, we are in case \hyperlink{factor_of_x_5}{(5)} or \hyperlink{factor_of_x_6}{(6)} up
	to an affine coordinate change in $y$. If $n=3$ and $V=\k$, we obtain case \hyperlink{factor_of_x_5}{(5)}. We then prove that $n=3$ and $V\not=\k$ is impossible. Indeed, otherwise, there is $y_0 \in \k$ with $p_i(y_0) = 0$ for $i=1, 2, 3$
	and the Jacobian of $f$ would be non-invertible in all points $(x, y_0, z)$,
	which contradicts Lemma~\ref{lem.properties_lin_systems_of_aff_spaces}\ref{lem.properties_lin_systems_of_aff_spaces4}.
	
	We may now assume that $\dim V=2$, so $n\in \{2,3\}$.
	After a reordering of $f_1, \ldots, f_n$, we get $V = \k p_1 \oplus \k p_2$.
	If $\deg(p_i) \leq 1$ for $i = 1, 2$ we are in case \hyperlink{factor_of_x_4}{(4)}.
	After a possible exchange of $f_1, f_2$ we may assume that $\deg(p_2) = 2$. 
	After adding a certain multiple of $f_2$ to $f_1$ we may assume that $\deg(p_1) \in \{0, 1\}$. If
	$\deg(p_1) = 0$, then after an affine coordinate change in $y$ at the source, we are in case 
	\hyperlink{factor_of_x_2}{(2)} or \hyperlink{factor_of_x_3}{(3)} depending on whether $p_2$ is a square or not. 
	If $\deg(p_1) = 1$,
	then we may assume after an affine coordinate change in $y$ at the source that
	$p_1 = y$ and $p_2 = a^2 y^2 + by + c^2$ for $a, b, c \in \k$ with
	$ac \neq 0$ (indeed, $0$ is not a common root of $p_1, p_2$, as they are linearly independent, see Lemma~\ref{lem.Lin_System_p_in_k[y]}\ref{lem.Lin_System_p_in_k[y]_0}). 
	After adding $-(2ac +b)f_1$ to $f_2$
	we obtain $p_2 = (ay - c)^2$. Thus after the coordinate change 
	$y \mapsto \frac{c}{a}(y+1)$
	we get $p_2 = c^2y^2$, $p_1 = \frac{c}{a}(y+1)$ and thus we are in case \hyperlink{factor_of_x_1}{(1)}.
\end{proof}

\begin{remark}
	If $\car(\k) \neq 2$, then in case \hyperlink{factor_of_x_2}{(2)} of Lemma~\ref{lem:factor_of_x}, one gets $V = \k \oplus \k (y + \frac{1}{2})^2$. Thus after
	the coordinate change $y \mapsto y - \frac{1}{2}$ we are in case \hyperlink{factor_of_x_3}{(3)}.
\end{remark}

In the case of a linear system of affine spaces of degree $3$ of $\AA^3$ in standard form 
such that one component is of the form $x(y+z^2) + z$, 
the remaining components are almost determined, up to affine automorphisms at the target:

\begin{lemma}
	\label{lem:p_1_is_y+z^2}
	Let $n \in \{ 2, 3 \}$ and let $p_i, q_i \in \k[y, z]$ for $i = 1, \ldots, n$ such that 
	$f = (f_1, \ldots, f_n) = (x(y+z^2) + z, x p_2 + q_2, \ldots, x p_n + q_n)$ is a
	linear system of affine spaces of degree $3$. 
	Then, up to an affine coordinate change at the target we have:
	\begin{enumerate}[leftmargin =*]
		\item \label{lem:p_1_is_y+z^2_n2} $n = 2$ and $f = (x(y+z^2) + z, a (y + z^2) + bx)$
				 for $(a, b) \in \k^2 \setminus \{ 0 \}$ or
		\item \label{lem:p_1_is_y+z^2_n3} $n = 3$ and $f = (x(y+z^2) + z, y + z^2, x)$.
	\end{enumerate}
	
\end{lemma}

\begin{proof}
	For $i=2, \ldots, n$, let
	$p_{i, 2}, q_{i, 3} \in \k[y, z]$ be the homogeneous parts of degree $2$ and $3$ 
	of $p_i$ and $q_i$, respectively. 
	
	We now prove that $p_{i,2}$ is divisible by $z^2$ for each $i\in \{2,\ldots,n\}$. If $q_{i,3}=0$, this follows from Proposition~\ref{prop.standard_form_if_line_at_infinity}\ref{rem.All_b_i.2_are collinear}, applied to the linear system $(f_1(y,x,z), f_i(y,x,z))$. 
	Now, assume $q_{i,3} \neq 0$ and that $p_{i, 2}$ is not a multiple of $z^2$ to derive a contradiction. 
	Since for each $\lambda \in \k$ the polynomial
	$\lambda f_1 + f_i = x(\lambda(y+z^2) + p_i) + (\lambda z + q_i)$ defines an $\AA^2$ in $\AA^3$,
	we get that for general $\lambda \in \k$ the polynomial 
	$\lambda(y+z^2) + p_i \in \k[y, z]$ defines a disjoint union of curves in $\AA^2$ 
	which are isomorphic to $\AA^1$ (see Proposition~\ref{Prop:XxpqisA2thendisjoint}). In particular, 
	for general (and thus for all) $\lambda \in \k$, the polynomial
	$\lambda z^2 + p_{i, 2}$ is a square. Since $p_{i, 2}$ is not a multiple
	of $z^2$ we get that $\car(\k) = 2$ and 
	for general $\lambda \in \k$, the polynomials 
	$\lambda z^2 + p_{i, 2}$ and $q_{i, 3}$ in $\k[y, z]$ have no common non-zero linear factor
	(remember that $q_{i, 3} \neq 0$). This implies that
	the homogeneous part of degree $3$ of $\lambda f_1 +	f_i$, which
	is equal to $x(\lambda z^2 + p_{i, 2}) + q_{i, 3}$, is
	irreducible for general $\lambda \in \k$ and thus $\lambda f_1 + f_i$ does not define
	an $\AA^2$ in $\AA^3$ (see Proposition~\ref{Proposition:XxpqisA2smalldegreelistpqDeg3}),
	contradiction. 
	
	For each $i\in \{2,\ldots,n\}$, we may now add multiples of $f_1$ to $f_i$ and assume that $\deg(p_i) \leq 1$.
	Lemma~\ref{Lem:x(y+z^2)+z} implies that $p_i \in \k$
	and gives the existence of $a_i, b_i \in \k$ such that
	\[
	f_i = x p_i + a_i (y+z^2) + b_i \quad \textrm{and} \quad (p_i, a_i) \neq (0, 0) \, .
	\]
	After applying a translation at the target, we may assume that $b_i = 0$. If $n = 2$, then
	we are in case~\ref{lem:p_1_is_y+z^2_n2}. Hence, we assume $n = 3$.
	Since $f_2$ and $f_3$ are linearly
	independent, it follows that $p_2 a_3 - p_3 a_2 \neq 0$; thus after a linear coordinate change in $y, z$ at the
	target, we may assume that
	\[
		\begin{pmatrix}
			a_2 & p_2\\
			a_3 & p_3
		\end{pmatrix}
		= 
		\begin{pmatrix}
			1 & 0 \\
			0 & 1
		\end{pmatrix} \, .
	\]
	This proves the lemma.
\end{proof}

\begin{lemma}\label{Lemm:xpiqiredky}
	Let $n \geq 1$. For $i \in \{1, \ldots, n\}$, let $f_i = x p_i + q_i$ where 
	$p_i, q_i \in \k[y, z]$ and $\deg(p_i) \leq 2$, $\deg(q_i) \leq 3$.
	If $f = (f_1, \ldots, f_n) \colon \AA^3 \to \AA^n$ is a linear system of affine spaces, then one may apply affine automorphisms at the target and source and reduce to the case where $p_1,\ldots,p_n\in \k[y]$ $($and still have $q_1,\ldots,q_n\in \k[y,z])$.
\end{lemma}
\begin{proof}
	Assume first that $\deg(p_i) \leq 1$ for all $i$. Lemma~\ref{Lem:Noptildenoaut} implies that no two of the linear parts of $p_1,\ldots,p_n$ are linearly independent, 
	so we reduce to the case $p_i \in \k[y]$ for all $i$ by applying an automorphism on $y,z$.
	
	Applying a permutation at the target we may now assume that $\deg(p_1)=2$. 
	
	If $p_1$ is irreducible, we apply an affine coordinate change at the source that fixes $[0:1:0:0]$ and obtain one of the cases of Proposition~\ref{Proposition:XxpqisA2smalldegreelistpq} for $f_1$. The action of this on $p_1$ corresponds to the action of an affine automorphism on $y,z$ and thus does not change the fact that $p_1$ is irreducible; it thus gives Case~\ref{pqPofdegree2} of Proposition~\ref{Proposition:XxpqisA2smalldegreelistpq}, 
	namely $f_1=x(y+z^2)+z$. We apply Lemma~\ref{lem:p_1_is_y+z^2} and obtain two possible cases. Exchanging $x$ and $y$ at the source gives the result.

	We may now assume that for each $(\lambda_1, \ldots, \lambda_n) \in \k^n\setminus \{0\}$, the polynomial
	$\lambda_1 p_1 + \ldots + \lambda_n p_n$ is reducible if it has degree $2$. Indeed, otherwise we reduce to the previous case by applying an affine automorphism at the target. 
	
	We may moreover assume that $\deg(p_i)=2$ for each $i\in\{1,\ldots,n\}$ by adding multiples of $p_1$ to the $p_i$ for $i \geq 2$.
	
	Let $p_{i, j} \in \k[y, z]$ be the homogeneous part of degree $j$ of $p_i$ for 
	$i=1, \ldots, n$, $j = 0, 1, 2$. Let $V=\Span_\k(p_{1,2},\ldots,p_{n,2})$. Applying Proposition~\ref{Proposition:XxpqisA2smalldegreelistpq} to each linear combination $\sum \lambda_i f_i$, we see that each element of $V$ is a square.
	If $\dim(V)=1$, then applying a linear automorphism on $y,z$, we get $p_{i,2}\in \k y^2$ for each $i\in \{1,\ldots,n\}$. For each $i$, the polynomial $p_i\in \k[y,z]$ is reducible, so $p_i\in \k[y]$ as desired.
	
	It remains to see that $\dim(V)\ge 2$ leads to a contradiction. As every element of $V$ is a square, we get $\car(\k)=2$ and $V=\k y^2+\k z^2$.
	For each $(\lambda_1, \ldots, \lambda_n) \in \k^n$, the polynomial
	$x \sum \lambda_i p_{i, 2} + \sum \lambda_i q_{i, 3}$ is reducible
	as it is the homogeneous part of degree $3$ of $\sum \lambda_i f_i$ (Corollary~\ref{cor.Properties_of_hypersurfaces_at_infinity}), so $\sum \lambda_i p_{i, 2}$ and $\sum \lambda_i q_{i, 3}$ have a common
	linear factor.
	Hence, we may apply Lemma~\ref{lem.2nd_power} 
	to $p_{1, 2}, \ldots, p_{n, 2}$ and $q_{1, 3}, \ldots, q_{n, 3}$ and get
	$h \in \k[y, z]_1$ with
	$q_{i, 3} = h p_{i, 2}$ for $i =1, \ldots, n$. 
	After applying the linear automorphism
	$(x - h, y, z)$ at the source, we reduce to the case where $q_{i,3}=0$ for $i=1, \ldots, n$. 
	The vector space generated by the homogeneous parts of degree $3$ of $f_1,\ldots,f_n$ is then equal to $\k xy^2 +\k xz^2$. This is impossible, as Proposition~\ref{prop.standard_form_if_line_at_infinity}\ref{rem.All_b_i.2_are collinear} applied to $(f_1(y,x,z),\ldots,f_n(y,x,z))$ shows.
\end{proof}

\subsection{The proof of Theorem~$\ref{thm:Classification}$}
\label{sec.classific}

In this section, we give a description of all linear systems $\AA^3 \to \AA^n$ of degree $\leq 3$
up to composition of affine automorphisms at the source and target 
and prove in particular Theorem~\ref{thm:Classification}.

\begin{proposition}
	\label{Prop:Deg3_Auto}
	Let $n \geq 2$. For $i \in \{1, \ldots, n\}$, let $f_i = x p_i + q_i$ where 
	$p_i, q_i \in \k[y, z]$ and $\deg(p_i) \leq 2$, $\deg(q_i) \leq 3$.
	If $f = (f_1, \ldots, f_n) \colon \AA^3 \to \AA^n$ is a linear system of affine spaces, then $n\le 3$ and
	$f$ is equivalent to $(g_1, \ldots, g_n) \colon \AA^3 \to \AA^n$ with one of the following 
	possibilities:
	\begin{enumerate}[$(i)$]
		\item \label{Prop:Deg3_Auto_Case1} $(g_1, g_2, g_3) = (x+p(y,z),y+q(z),z)$ 
		where $p \in \k[y, z]$, $q \in \k[z]$;
		\item \label{Prop:Deg3_Auto_Case2}
		$(g_1, g_2, g_3) = (xy + ya(y, z) + z, x + a(y, z) + r(y), y)$ 
		where $a \in \k[y, z]$, $r \in \k[y]$;
		\item \label{Prop:Deg3_Auto_Case3} $(g_1, g_2) = (xy + ya(y, z) + z, y)$
		where $a \in \k[y, z]$;
		\item \label{Prop:Deg3_Auto_Case4} $(g_1, g_2) = (xy^2 + y (z^2 + az + b) + z, y)$ 
		where $a, b \in \k$.
	\end{enumerate}
\end{proposition}

\begin{proof}
	Using Lemma~\ref{Lemm:xpiqiredky}, we may assume that $p_i \in \k[y]$ for all $i$.
	
	We then apply Lemma~\ref{lem:factor_of_x}, and may assume that $p_3=0$ if $n=3$ and that $(p_1,p_2)$ is in one of the following cases:
	\[
		\begin{array}{llll}
			\hypertarget{p1p2_x_1}{\hyperlink{factor_of_x_1}{(1)}} &(y^2,y+1) &
			\hypertarget{p1p2_x_4}{\hyperlink{factor_of_x_4}{(4)}} & (y,1) \\
			\hypertarget{p1p2_x_2}{\hyperlink{factor_of_x_2}{(2)}} & (y^2,1) &
			\hypertarget{p1p2_x_5}{\hyperlink{factor_of_x_5}{(5)}} & (1,0) \\
			\hypertarget{p1p2_x_3}{\hyperlink{factor_of_x_3}{(3)}} & (y(y+1),1) &
			\hypertarget{p1p2_x_6}{\hyperlink{factor_of_x_6}{(6)}} & (p,0) 
			\textrm{ with $p \in \{0, y, y^2, y(y+1) \}$ and $n=2$}
		\end{array}
	\]
	
	We now go through the different cases.

	In Cases~\hyperlink{p1p2_x_1}{(1)}-\hyperlink{p1p2_x_4}{(4)}, if $n=3$ then $f_3=q_3$ is an element of $\k[y]$
	of degree~$1$. This follows from 
	Lemma~\ref{lem.Lin_System_p_in_k[y]}\ref{lem.Lin_System_p_in_k[y]_1} applied to $(f_1,f_3)$, as $p_3=0$ and $p_1\in \k[y]\setminus \k$. One can then, if one needs, replace $f_3$ with $\alpha f_3+\beta$ for some $\alpha,\beta\in \k$, $\alpha\not=0$ and obtain $f_3=y$.
	
	In Cases \hyperlink{p1p2_x_1}{(1)}-\hyperlink{p1p2_x_2}{(2)}, $p_1=y^2$. 
	There is $\alpha\in \Aff(\A^3)$ that fixes $[0:1:0:0]$ such that
	$\alpha^\ast(f_1)$ is one of the cases of Proposition~\ref{Proposition:XxpqisA2smalldegreelistpq}. As $\alpha^*(y^2)$ is the coefficient of $x$ in $\alpha^\ast(f_1)$ up to non-zero scalars, we obtain
	that $\alpha^\ast(f_1)$ is the polynomial of Case~\ref{pqSecondCasexy2Deg3}
	in Proposition~\ref{Proposition:XxpqisA2smalldegreelistpq}
	and $\alpha^\ast(y) \in \k[y]$,
	so we reduce to the case where $p_1=y^2$ and $q_1=ys(z)+z$ for some $s\in \k[z]$
	of degree $\leq 2$.
	
	\hyperlink{p1p2_x_1}{(1)}: Here $p_2=y+1$, so Lemma~\ref{lem.Lin_System_p_in_k[y]}\ref{lem.Lin_System_p_in_k[y]_3}
	\ref{lem.Lin_System_p_in_k[y]_3ii} shows that $s(z) = -z + \mu$ and $q_2 = -z + r(y)$ 
	where $\mu \in \k$ and $r \in \k[y]$ has degre $\leq 3$.
	After performing $(x, y, z) \mapsto (x, y, z + \mu)$ at the source and 
	adding constants at the target we may assume $\mu = 0$. Hence,
	\[
		(f_1, f_2) = (xy^2 -zy+z, x(y+1) -z + r(y)) \, .
	\]
	We apply $(x, y, z) \mapsto (z, y+1, -x)$ at the source and get
	\[
		 (f_1, f_2) = (xy + yz(y+2) + z, x + z(y+2) + r(y+1)) \, . 
	\]
	This gives case~\ref{Prop:Deg3_Auto_Case2} if $n=2$. 
	If $n = 3$, then $f_3$ is still an element of $\k[y]$ of degree $1$ and we can then assume $f_3=y$ to obtain Case~\ref{Prop:Deg3_Auto_Case2}.

	\hyperlink{p1p2_x_2}{(2)}: Here $p_2=1$, so $f_2=x+q_2(y,z)$ and if $n = 3$, then 
	$f_3\in \k[y]$ is of degree $1$, so we may assume $f_3= y$. Lemma~\ref{lem.Lin_System_p_in_k[y]}\ref{lem.Lin_System_p_in_k[y]_3}\ref{lem.Lin_System_p_in_k[y]_3i} gives $q_2\in \k[y]$ and $s\in \k$, thus after a permutation of $x, y, z$ at the source we are in case~\ref{Prop:Deg3_Auto_Case1}.
		
	\hyperlink{p1p2_x_3}{(3)}: Here $p_1=y(y+1)$, so $q_1=a(y,z)y(y+1)+s(y)z + t(y)$ for polynomials $a\in \k[y,z], s, t\in \k[z]$ of degree $\leq 1$ with $s(0)s(-1) \neq 0$ (Proposition~\ref{Prop:XfibratReminder}). Replacing $x$ with $x-a(y,z)$, we may assume that $a=0$. 
	Lemma~\ref{lem.Lin_System_p_in_k[y]}\ref{lem.Lin_System_p_in_k[y]_4} then implies that
	$s(y) \in \k^\ast$ and $q_2(y, z) \in \k[y]$. Hence, 
	\[
		(f_1, f_2) = (xy(y+1)+sz + t(y), x + q_2(y))
	\]
	and if $n=3$, we may assume $f_3 = y$. After a permutation of $x, y, z$ at the source and 
	a rescaling of $f_1$, we are in case~\ref{Prop:Deg3_Auto_Case1}.

	\hyperlink{p1p2_x_4}{(4)}: Here $p_1=y$, so $q_1=\tilde{a}(y,z)y+ \alpha z + \beta$ where $\tilde{a}\in \k[y,z]$, $\alpha\in \k^*$ and $\beta\in \k$ (Proposition~\ref{Prop:XfibratReminder}). Replacing $z$ with $\alpha^{-1}(z-\beta)$, we get $f_1=xy + a(y, z)y + z$ for some $a\in \k[y,z]$.
	By Lemma~\ref{lem.Lin_System_p_in_k[y]}\ref{lem.Lin_System_p_in_k[y]_2} 
	there is $r(y) \in \k[y]$
	with $f_2 = x + a(y, z) + r(y)$. Hence, we are in case~\ref{Prop:Deg3_Auto_Case2}.

	\hyperlink{p1p2_x_5}{(5)}: If $n = 2$, then according to Lemma~\ref{Coro:A1inA2smalldegree} 
	we may apply an affine automorphism in $(y, z)$ at the source in order to get 
	$f_2 = q_2 = y + q(z)$ and thus we are in case~\ref{Prop:Deg3_Auto_Case1}.
	If $n = 3$, then $f = (x + q_1, q_2, q_3)$.
	Since $\AA^3 \to \AA^2$, $(x, y, z) \mapsto (q_2(y, z), q_3(y, z))$
	is an affine linear system of affine spaces, by Lemma~\ref{lem.properties_lin_systems_of_aff_spaces}\ref{lem.properties_lin_systems_of_aff_spaces3} 
	the same holds for $(q_2, q_3) \colon \AA^2 \to \AA^2$.
	By Proposition~\ref{cor.Autos_of_A2_deg3}, we get up to affine automorphisms in $y, z$
	at the source and target that $(q_2, q_3) = (y + q(z), z)$ for some $q \in \k[z]$ 
	and thus we are again in case~\ref{Prop:Deg3_Auto_Case1}.
		
	\hyperlink{p1p2_x_6}{(6)}: Assume first that $p = 0$. Then by Proposition~\ref{cor.Autos_of_A2_deg3} we
	may assume that $f = (y + q(z), z)$ for some $q \in \k[z]$. After replacing $y$ with $x$ and
	$z$ with $y$ we are in case~\ref{Prop:Deg3_Auto_Case1}. In all other cases
	$p \in \k[y] \setminus\k$ and by Lemma~\ref{lem.Lin_System_p_in_k[y]}\ref{lem.Lin_System_p_in_k[y]_1}
	we get that $f_2$ is a polynomial of degree $1$ in $\k[y]$.
	By Proposition~\ref{Proposition:XxpqisA2smalldegreelistpq} there is $\alpha \in \Aff(\AA^3)$
	that fixes $[0:1:0:0]$ such that
	$\alpha^\ast(f_1)$ is one of the polynomials in the cases~\ref{pqCaseOnlyyz}-\ref{pqTwoyDeg3} of
	Proposition~\ref{Proposition:XxpqisA2smalldegreelistpq}. Since up to scalars, $\alpha^\ast(p)$
	is the factor of $x$ in $\alpha^\ast(f_1)$ (when we consider it as a polynomial in $x$ over $\k[y, z]$)
	and since $p \in \{y, y^2, y(y+1)\}$, it follows that $\alpha^\ast(f_1)$ 
	belongs to one of the Cases \ref{pqSecondCasexyDeg3}-\ref{pqTwoyDeg3} of
	Proposition~\ref{Proposition:XxpqisA2smalldegreelistpq} and $\alpha^\ast(y) \in \k[y]$.
	In particular, $\alpha^\ast(f_2)$
	is a polynomial of degree $1$ in $y$. 
	Proposition~\ref{Proposition:XxpqisA2smalldegreelistpqDeg3} then gives $\beta\in \Aff(\A^3)$ such that $\beta^*(y)\in \k[y]$ and such that $\beta^*(f_1)$ is one of the polynomials in cases~\ref{FirstCasekyDeg3}, \ref{SecondCasexyDeg3} or~\ref{SecondCasexy2Deg3} of Proposition~\ref{Proposition:XxpqisA2smalldegreelistpqDeg3}. 
	As $\beta^*(f_2)$ is again a polynomial
	of degree $1$ in $\k[y]$, we may replace it with $y$ and get 
	cases~\ref{Prop:Deg3_Auto_Case1},~\ref{Prop:Deg3_Auto_Case3} or~\ref{Prop:Deg3_Auto_Case4}. 
		
\end{proof}

As an immediate consequence we get

\begin{corollary}
	\label{cor:Trivial_bundles}
	Let $n \geq 1$ and let $f \colon \AA^3 \to \AA^n$ be a linear system of affine spaces of degree $\leq 3$.
	Then $f$ is equivalent to a linear system of affine spaces in standard form if and only if 
	$f$ is a trivial $\AA^{3-n}$ bundle. Moreover, the latter condition is satisfied if
	$\car(\k) \not\in \{2, 3\}$.
\end{corollary}

\begin{proof}
	If $f \colon \AA^3 \to \AA^n$ is a trivial $\AA^{3-n}$-bundle, then $f$ is equivalent 
	to a linear system of affine spaces in standard form by Corollary~\ref{cor.Linear_systems_in_standard_form}. Conversely, we assume that $f$ is a linear system of affine spaces in standard form and prove that $f$ is a trivial $\AA^{3-n}$-bundle. If $n=1$, then $f$ is a variable of $\k[x,y,z]$ (Corollary~\ref{coro:TrivialisationVariables}), so it defines a trivial $\AA^2$-bundle.
	If $n\ge 2$, we go through the four cases of Proposition~\ref{Prop:Deg3_Auto}.
	In case~\ref{Prop:Deg3_Auto_Case1}
	and~\ref{Prop:Deg3_Auto_Case2}, the morphism $(g_1, g_2, g_3) \colon \AA^3 \to \AA^3$ defines
	an automorphism and in case~\ref{Prop:Deg3_Auto_Case3} and~\ref{Prop:Deg3_Auto_Case4}, 
	Proposition~\ref{Prop:XfibratReminder}\ref{XfibratReminderPvariable} gives the existence of $g_3\in \k[x,y,z]$ such that $(g_1,g_2,g_3)\in \Aut(\A^3)$.
	The second claim follows from Proposition~\ref{prop:globally_xp_i+q_i}.
\end{proof}

We now come to the proof of our description of linear systems of 
affine spaces $\AA^3 \to \AA^n$ of degree $\leq 3$:

\begin{proof}[Proof of Theorem~$\ref{thm:Classification}$]
	Let $f_1, \ldots, f_n \in \k[x, y, z]$ such that $f = (f_1, \ldots, f_n) \colon \AA^3 \to \AA^n$ is
	a linear system of affine spaces of degree $\leq 3$. If $f \colon \AA^3 \to \AA^n$ is not a trivial
	$\AA^{3-n}$-bundle, then by Corollary~\ref{cor:Trivial_bundles} 
	and Proposition~\ref{prop:globally_xp_i+q_i}, we are in cases~\ref{thm:Classification_8} or~\ref{thm:Classification_9}.
	Thus we may assume that $f \colon \AA^3 \to \AA^n$ is a trivial $\AA^{3-n}$-bundle. If $n=1$, this means that $f=f_1$ is a variable, and the description of $f$ follows from Proposition~\ref{Proposition:XxpqisA2smalldegreelistpqDeg3}. We may then assume that $n\ge 2$, that $f$ is  in standard form (applying again Corollary~\ref{cor:Trivial_bundles}) and then go through the different cases of Proposition~\ref{Prop:Deg3_Auto}:
	
	\ref{Prop:Deg3_Auto_Case1}:
	$(f_1, f_2) = (x+p(y,z),y+q(z))$ with $p \in \k[y, z]$
	and $q \in \k[z]$, and $f_3=z$ if $n=3$. Since $\deg(f) \leq 3$, we may write $p = \sum_{i=0}^3 p_i(y, z)$
	and $q(z) = \sum_{i=0}^3 q_i z^i$ where $p_i \in \k[y, z]$ is homogeneous of degree $i$ and
	$q_i \in \k$.
	After applying a translation at the target we may assume that $p_0 = 0$ and $q_0 = 0$.
	After composing $f$ with the automorphism $(x-p_1(y-q_1 z, z), y - q_1 z, z)$ at the source
	we are either in case~\ref{thm:Classification_4} or~\ref{thm:Classification_10}.
	
	\ref{Prop:Deg3_Auto_Case2} and \ref{Prop:Deg3_Auto_Case3}: 
	There exist $a \in \k[y, z]$ of degree $\leq 2$ and $r \in \k[y]$
	of degree $\leq 3$ such that $g=(xy + ya(y, z) + z, x + a(y, z) + r(y), y)$ satisfies:
	$f$ is either equal to $g$,
	or $f$ is equal to $\pi\circ g$ where $\pi\colon \A^3\to \A^2$ is one of the projections $(x,y,z)\mapsto(x,z)$ or $(x,y,z)\mapsto(x,y)$.
	Write $r(y) = r_0 + r_1y + r_2 y^2 + r_3 y^3$
	and $a = a_0 + a_1(y, z) + a_2(y, z)$ where $r_i \in \k$ and $a_i \in \k[y, z]$ is homogeneous of degree $i$. After adding constants at the target, we may assume $r_0 = 0$.
	After applying $(x - a_0 - a_1(y, z), y, z)$ 
	at the source, we may further assume that $a = a_2$ is homogeneous of degree $2$.
	After applying the permutation of the coordinates $(x, y, z) \mapsto (y, z, x)$ at the source, we have replaced $g$ with $g=(yz + za_2(z, x) + x, y + a_2(z, x) + r_1z+r_2z^2+r_3z^3, z)$.
	
	If $a_2(z, x) \in \k[z]$ and $f_2 \neq z$, then
	after applying $(x, y-r_1 z, z)$ at the source we are
	in case~\ref{thm:Classification_4} or case~\ref{thm:Classification_10}. If $a_2(z, x) \in \k[z]$ and
	$f_2 = z$, then after exchanging $y$ and $z$ at the source
	we are again in case~\ref{thm:Classification_4}. Thus we may assume that $a_2(z, x) \not\in \k[z]$.
	If $n = 2$, then we are in case~\ref{thm:Classification_5} or~\ref{thm:Classification_6} and if
	$n = 3$, then we are in case~\ref{thm:Classification_11} 
	after applying $(x, y - r_1z, z)$ at the target.
	
	\ref{Prop:Deg3_Auto_Case4}: This is case~\ref{thm:Classification_7}.
\end{proof}

Next, we will show that the cases in Theorem~\ref{thm:Classification} are all
pairwise non-equivalent. For this we need the following lemma.

\begin{lemma}
\label{Lemm:TwoCasesOfHypersurfacesNotequal}
For each $r_2\in \k[y,z]\setminus \k[y]$, homogeneous of degree $2$, it is not possible to find $p\in \k[y,z]$, $\lambda\in \k$ and $\alpha\in \Aff(\A^3)$ such that 
\[
	\alpha^*(xy+yr_2(y,z)+z)=\lambda x+p(y,z) \, .
\]
\end{lemma}
\begin{proof}
Suppose for contradiction that $p,\lambda,\alpha$ exist. We may assume that $\alpha\in \GL_3$, as a translation sends $\lambda x+p(y,z)$ onto $\lambda x+\tilde{p}(y,z)$ for some $\tilde{p}\in \k[y,z]$. Hence, the homogeneous part of degree $2$ of $\alpha^*(xy+yr_2(y,z)+z)$ is $\alpha^*(xy)\in \k[y,z]$. This implies that $\alpha^*(x),\alpha^*(y)$ are linearly independent elements of $\k y+\k z$, as $\k[y, z]$
is factorially closed in $k[x,y,z]$. Replacing $\alpha$ by its composition with an element of $\GL_2$ acting on $y,z$ (which simply replaces $p$ with another polynomial in $\k[y, z]$), we may assume that $\alpha^*(x)=z$ and $\alpha^*(y)=y$. Hence, $\alpha^*(z)=a x+b y+ cz$ for some $a,b,c\in \k$, $a\not=0$. This gives
\[
	\lambda x+p(y,z)=\alpha^*(xy+yr_2(y,z)+z)= yz+yr_2(y,a x+b y+ cz)+a x+b y+ cz \, ,
\]
impossible as $r_2\in \k[y,z]\setminus \k[y]$ and $a\not=0$, 
so the coefficient of $x$ of the right hand side is not constant.
\end{proof}

\begin{proposition}
	\label{prop.Cases_distinct}
	The eleven families in Theorem~$\ref{thm:Classification}$ define disjoint sets of equivalence classes of affine linear systems of affine spaces, 
	i.e.~if $(k), (l) \in \{$\ref{thm:Classification_1}, \ref{thm:Classification_2}, \ldots, \ref{thm:Classification_11}$\}$, and $f,g\colon \AA^3 \to \AA^n$ are equivalent 
	affine linear systems of affine spaces as in family $(k)$ and $(l)$ of
	Theorem~$\ref{thm:Classification}$, respectively, then $(k)=(l)$. 
\end{proposition}

\begin{proof}
	 If $f$ or $g$ is a non-trivial $\A^1$-fibration, then both are. As $\car(\k)=2$ in \ref{thm:Classification_8} and $\car(\k)=3$ in \ref{thm:Classification_9}, we obtain $(k)=(l)=$\ref{thm:Classification_8} or $k=l=$\ref{thm:Classification_9}. We may now assume that $(k)$ and $(l)$ are both contained in one of the sets 
	$\{$\ref{thm:Classification_1}, \ref{thm:Classification_2}, \ref{thm:Classification_3}$\}$, 
	$\{$\ref{thm:Classification_4}, \ref{thm:Classification_5}, \ref{thm:Classification_6}, \ref{thm:Classification_7}$\}$ or $\{$\ref{thm:Classification_9}, \ref{thm:Classification_10}, \ref{thm:Classification_11}$\}$.
	
	We write $f=(f_1,\ldots,f_n)$ and $g=(g_1,\ldots,g_n)$.

	Assume that $f_1 = xy^2 + y(z^2 + az + b) + z$ for some $a, b \in \k$, i.e.~
	$(k)\in \{$\ref{thm:Classification_3}, \ref{thm:Classification_7}$\}$.
	Then for general $(\lambda_1, \ldots, \lambda_n)$, the homogeneous part of 
	degree $3$ of $\sum \lambda_i f_i$
	does not factor into linear polynomials. This has to be the same for the homogeneous part of degree $3$ of $\sum \lambda_i g_i$, 
	so $(k)=(l)\in \{$\ref{thm:Classification_3}, \ref{thm:Classification_7}$\}$
	by inspecting the cases that are different from \ref{thm:Classification_3}, \ref{thm:Classification_7}. The same holds when $(l)\in \{$\ref{thm:Classification_3}, \ref{thm:Classification_7}$\}$, so we may exclude these two cases.
	
	Assume now that $f_1 = x + r_2(y, z) + r_3(y, z)$ for homogeneous polynomials $r_2, r_3 \in \k[y, z]$
	of degree $2$ and $3$, respectively, i.e.~$(k)\in\{$\ref{thm:Classification_1}, \ref{thm:Classification_4},
	\ref{thm:Classification_10}$\}$. 
	For each $(\lambda_1, \ldots, \lambda_n)$, 
	the polynomial $\sum \lambda_i f_i$ is equal to $\lambda x+p(y,z)$ for some $\lambda\in\k$ and $p\in \k[y,z]$.
	Lemma~\ref{Lemm:TwoCasesOfHypersurfacesNotequal} implies that $g_1$ is not equivalent to $xy+y a_2(y,z)+z$ for some $a_2 \in \k[y,z]\setminus \k[y]$, homogeneous of degree $2$, so $(l)\notin \{$\ref{thm:Classification_2}, \ref{thm:Classification_5}, \ref{thm:Classification_6}, \ref{thm:Classification_11}$\}$. 
	This yields $(k)=(l)\in\{$\ref{thm:Classification_1}, \ref{thm:Classification_4}, \ref{thm:Classification_10}$\}$. As before, we may now exclude the cases
	\ref{thm:Classification_1}, \ref{thm:Classification_4} and \ref{thm:Classification_10}.
	
	It remains to see that $(k)=$\ref{thm:Classification_5} and 
	$(l)=$\ref{thm:Classification_6} are not equivalent. 
	We take homogeneous polynomials $a_2, b_2 \in \k[x, z] \setminus \k[z]$ of degree $2$
	and $r_1, r_2, r_3 \in\k$ such that
	\begin{align*}
		f = (f_1, f_2) & = (yz + za_2(x, z) + x, y + a_2(x, z) + r_1 z + r_2 z^2 + r_3 z^3) \\
		g = (g_1, g_2) & = (yz + zb_2(x, z) + x, z) \, .
	\end{align*}
	and prove that $f,g$ are not equivalent.
	For $i =1, 2$, denote by $f_{i, 3}, g_{i, 3} \in \k[x, y, z]$ the homogeneous part of degree $3$ of $f_i$
	and $g_i$, respectively. If $r_3 \neq 0$, then $f_{1, 3}, f_{2, 3}$ are linearly independent
	as $a_2 \not\in \k[z]$, but $g_{1, 3}, g_{2, 3}$ are not, 
	so $f$ and $g$ are not equivalent. If $r_3=0$, as
	$a_2 \not\in \k[z]$, we get that
	$\deg(\lambda_1 f_1 + \lambda_2 f_2) \in \{2, 3\}$ 
	for each $(\lambda_1, \lambda_2) \neq (0, 0)$. As $\deg(g_2) = 1$, $f$ and $g$ are not equivalent.
\end{proof}

\begin{corollary}
	\label{cor.deg3_autos_tame}
	Every automorphism of degree $\leq 3$ of $\AA^3$ is tame.
\end{corollary}

\begin{proof}
	As for each $a \in \k[x, z]$ and each $r \in \k[z]$ we have the decomposition
	\[
		(x + yz + za(x, z), y + a(x, z) + r(z), z) = h_1 \circ \iota \circ h_2 \circ \iota
	\]
	where $h_1 = (x + yz, y + r(z), z) \in \Triang_{\k}(\AA^3)$, 
	$h_2 = (x + a(y, z), y, z) \in \Triang_{\k}(\AA^3)$ 
	and $\iota = (y, x, z) \in \Aff_{\k}(\AA^3)$,
	it follows from Theorem~\ref{thm:Classification} 
	that all automorphisms of degree $\leq 3$ of $\AA^3$ are tame. 
\end{proof}

\section{Dynamical degrees of automorphisms of $\AA^3$ of degree at most $3$ }
\label{sec.DDeg}

As an application of our description
of automorphisms of $\AA^3$ of degree $\leq 3$
(see Theorem~\ref{thm:Classification}), 
we list in this section all possible dynamical degrees of these automorphisms. 
Recall that the dynamical degree satisfies $\lambda(f) \leq \deg(f)$ 
and that $\lambda(f) = \lambda(g)$ if 
$f, g$ are conjugated automorphisms in $\Aut(\AA^n)$ and more generally if $f, g$ are only conjugated
in the bigger group $\Bir(\AA^n)$ of birational maps of $\AA^n$. 

\subsection{Affine-triangular automorphisms}
We say that an element $f\in \Aut(\AA^n)$ is \emph{affine-triangular} if $f=\alpha\circ \tau$, where $\alpha\in \Aff(\A^n)$ is an affine automorphism and $\tau\in \Triang_{\k}(\AA^n)$ is a triangular automorphism.
Note that an element $g\in \Aut(\A^n)$ is equivalent to a triangular automorphism if and only if it is conjugate to an affine-triangular automorphism by an affine automorphism. The dynamical degrees of affine-triangular automorphisms of $\AA^3$ can be computed, 
using a simple algorithm described in \cite{BlSa2022Dynamical-degrees-}. In particular, one has the following result.

\begin{theoremsimple}\cite[Theorem 1]{BlSa2022Dynamical-degrees-}\label{TheoremDynAffTri}
For each field $\kany$ and each integer $d\ge 2$, the set of dynamical degrees of all affine-triangular automorphisms of $\A^3$ of degree $\le d$ is equal to
\[
	\left.\left\{\frac{a+\sqrt{a^2+4bc}}{2}\right| (a,b,c)\in \NN^3, a+b\le d, c\le d\right\}\setminus \{0\}.
\]
Moreover, for all $a,b,c\in \NN$ such that $\lambda=\frac{a+\sqrt{a^2+4bc}}{2}\not=0$, the dynamical degree of the automorphism
\[
	(z + x^ay^b,y+x^c,x)
\] 
is equal to $\lambda$.
\end{theoremsimple}

\begin{corollary}\label{Cor:DynDegSmallAffTriang}
	For each $d\ge 1$ and each field $\kany$, let us denote by $\Lambda_{d,\kany}\subset\mathbb{R}$ the set of dynamical degrees of all automorphisms of $\A^3_{\kany}$ of degree $d$. We then have
	\[
		\begin{array}{rcl}
	\Lambda_{1,\kany}&=&\{1\}\\
	\Lambda_{2,\kany}&=&\{1,\sqrt{2},(1+\sqrt{5})/2,2\}\\
	\Lambda_{3,\kany}&\supseteq&\{1 , \sqrt{2} , \ \frac{1+\sqrt{5}}{2} , \ \sqrt{3} , 2 , \ \frac{1 + \sqrt{13}}{2} , \
		1 + \sqrt{2} , \ \sqrt{6} , \ \frac{1 + \sqrt{17}}{2} , \ 1+ \sqrt{3} , \ 3 \}
		\, .
	\end{array}
	\]
	Moreover, if $f\in \Aut(\A^3_{\kany})$ is conjugated in $\Aut(\A^3_{\overline{\kany}})$ to an affine triangular automorphism of degree $\le 3$ $($where
	$\overline{\kany}$ is a fixed algebraic closure of $\kany)$, then 
	\[
		\lambda(f)\in \{1 , \sqrt{2} , \ (1+\sqrt{5})/2 , \ \sqrt{3} , 2 , \ (1 + \sqrt{13})/2 , \
		1 + \sqrt{2} , \ \sqrt{6} , \ (1 + \sqrt{17})/2 , \ 1+ \sqrt{3} , \ 3 \} \, .
	\]
\end{corollary}

\begin{proof} Let us write 
\[
	L_d=\left.\left\{ \frac{a+\sqrt{a^2+4bc}}{2}\right | (a,b,c)\in \NN^3, a+b\le d, c\le d\right\}\setminus \{0\} \quad
	\textrm{for each $d\ge 1$} \, .
\] 
This gives then 
	\[
		\begin{array}{rcl}
	L_1&=&\{1\}\\
	L_2&=&\{1,\sqrt{2},(1+\sqrt{5})/2,2\}\\
	L_3&=&\{1 , \sqrt{2} , \ \frac{1+\sqrt{5}}{2} , \ \sqrt{3} , 2 , \ \frac{1 + \sqrt{13}}{2} , \
		1 + \sqrt{2} , \ \sqrt{6} , \ \frac{1 + \sqrt{17}}{2} , \ 1+ \sqrt{3} , \ 3 \}
		\, .
	\end{array}
	\]

For each $d \in \{1,2, 3\}$ holds: If $f\in \Aut(\A^3_{\kany})$ is conjugated in $\Aut(\A^3_{\overline{\kany}})$ to an affine triangular automorphism of degree $\le d$, then Theorem~\ref{TheoremDynAffTri} implies that $\lambda(f)\in L_d$. In particular, $\Lambda_{1,\kany}\subseteq L_1$ and $\Lambda_{2,\kany}\subseteq L_2$, as every element of $\Aut(\A^3_{\kany})$
of degree $\leq 2$ is equivalent to a triangular automorphism and is thus conjugate in $\Aut(\A^3_{\overline{\kany}})$ to an affine triangular automorphism (Theorem~\ref{thm:Classification}).

It remains to see that $L_d\subseteq \Lambda_{i,\kany}$ for $d=1,2,3$, by giving explicit examples. For $d=1$, we simply take the identity. For $d\in \{2,3\}$, we use elements of the form
\[
	f_{a,b,c}=(z + x^ay^b,y+x^c,x)\in \Aut(\A^3_\k)
\]
whose dynamical degrees are equal to 
$\lambda(f_{a,b,c})=(a+\sqrt{a^2+4bc})/2$ when this number is not zero (Theorem~\ref{TheoremDynAffTri}).

For $d=2$, we use $f_{1,0,2}$, $f_{0,1,2}$, $f_{1,1,1}$ and $f_{1,1,2}$, which all have degree $2$ and dynamical degrees $1,\sqrt{2},(1+\sqrt{5})/2,2$ respectively.

For $d=3$, we first use  $f_{1,0,3}$, $f_{0,1,3}$, $f_{2,0,3}$, $f_{1,1,3}$, $f_{2,1,1}$, $f_{0,2,3}$, $f_{1, 2, 2}$, 
	$f_{2,1,2}$ and  $f_{0,3,3}$ which all have degree $3$ and dynamical degrees 
$1  $, $ \sqrt{3} $, $ 2 $, $(1 + \sqrt{13})/2$, $ 1 + \sqrt{2} $, $ \sqrt{6} $, $(1 + \sqrt{17})/2$, $ 1+ \sqrt{3}$ and $3$, respectively.
In order to obtain the values
$\sqrt{2}$ and $(1+\sqrt{5})/2$, we conjugate  $f_{0,1,2}=(z + y,y+x^2,x)$ and
$f_{1,1,1}=(z + xy,y+x,x)$  by $(x,y+z^3,z)$ and $(x,y+z^2,z)$, respectively, to get two automorphisms of $\A^3$ of degree $3$ having dynamical degree equal to to $\lambda(f_{0,1,2})=\sqrt{2}$ and $\lambda(f_{1,1,1})=(1+\sqrt{5})/2$, respectively.
\end{proof}
 
\subsection{List of dynamical degrees of all automorphisms of degree $3$}

An automorphism $f \in \Aut(\AA^n)$ is called 
\emph{algebraically stable}, if $\deg(f^r) = \deg(f)^r$ for all $r > 0$. In this case, $\lambda(f) = \deg(f)$.
Now, let $\iota \colon \AA^n \to \PP^n$ be the standard embedding, i.e.~$\iota(x_1, \ldots, x_n) = [1: x_1: \cdots: x_n]$.
Note that $f$ is algebraically stable, if and only if the extension of $f$ to a birational map 
$\bar{f} \colon \PP^n \bir \PP^n$ via $\iota$ satisfies the following: $\bar{f}^r$ maps the hyperplane
at infinity $H_{\infty} = \PP^n \setminus \iota(\AA^n)$ not into the base locus of $\bar{f}$ for each $r > 0$ (follows for instance from \cite[Proposition 1.4.3]{Sibony} or \cite[Lemma 2.14]{BlancConj}).

The computation of the dynamical degrees in Theorem~\ref{thm.Dynamical_degrees} is
 heavily based on the results of~\cite{BlSa2022Dynamical-degrees-}. Let us recall the notations and
 results that we need here. 
 
\begin{definition}
	Let	$\mu = (\mu_1, \ldots, \mu_n) \in (\RR_{\geq 0})^n$ and $r \in \RR_{\geq 0}$.
	For a polynomial
	$p = \sum p_{i_1, \ldots, i_n} x_1^{i_1} \cdots x_n^{i_n} \in \k[x_1, \ldots, x_n]$ 
	(where $p_{i_1, \ldots, i_n} \in \k$) its \emph{$\mu$-homogeneous part of degree $r$}
	is the polynomial
	\[
		\sum_{i_1 \mu_1 + \ldots + i_n \mu_n = r} p_{i_1, \ldots, i_n} x_1^{i_1} \cdots x_n^{i_n}
		\in \k[x_1, \ldots, x_n] \, .
	\]
	For each $p\in \k[x_1,\ldots,x_n]\setminus \{0\}$, we define $\deg_\mu(p)$ to be the maximum of the real numbers $r\in \RR_{\ge 0}$ such that the $\mu$-homogeneous part of degree $r$ of $p$ is non-zero. We then set $\deg_\mu(0)=-\infty$.
\end{definition} 

\begin{definition}
	Let $f = (f_1, \ldots, f_n) \in \Aut(\AA^n)$ and let 
	$\mu = (\mu_1, \ldots, \mu_n) \in (\RR_{\geq 0})^n$.
	We define the \emph{$\mu$-degree} of $f$ by
	\[
		\deg_\mu(f) = \inf\set{\theta \in \RR_{\geq 0}}{ 
		\textrm{$\deg_\mu(f_i) \leq \theta \mu_i$ 
		for each $i \in \{1, \ldots, n\}$}} \, .
	\]
	In particular, $\deg_\mu(f) = \infty$ if the above set is empty.
	If $\theta = \deg_{\mu}(f) < \infty$, then for each
	$i \in \{1, \ldots, n\}$, let $g_i \in \k[x_1, \ldots, x_n]$ be the $\mu$-homogeneous part 
	of degree $\theta \mu_i$ of  $f_i$. Then
	$g = (g_1, \ldots, g_n) \in \End(\AA^n)$ 
	is called the \emph{$\mu$-leading part of $f$}.
\end{definition}

The following result from~\cite{BlSa2022Dynamical-degrees-} 
will serve as the main technique to compute dynamical degrees.

\begin{proposition} \cite[Proposition~A]{BlSa2022Dynamical-degrees-}
	\label{prop.Comp_DDeg}
	Let $f \in \Aut(\AA^n)$ and let
	$\mu = (\mu_1, \ldots, \mu_n) \in (\RR_{> 0})^n$ be such that $\theta=\deg_\mu(f) \in \RR_{> 1}$. 
	If the $\mu$-leading part
	$g \colon \AA^n \to \AA^n$ of $f$ satisfies $g^r \neq 0$ for each $r > 0$, then the dynamical degree $\lambda(f)$
	is equal to  $\theta$.
\end{proposition}

\begin{proposition}\label{Prop:DynDeg3NotTriang} Let $f=(f_1,f_2,f_3)=\alpha \circ g \in \Aut(\A^3)$, where $\alpha\in \Aff(\A^3)$,
\[ g=(x+yz + za(x, z)+\xi z, y + a(x, z) + r(z), z),\]
 $\xi\in \k$, $a(x,z)=a_0 x^2+a_1 xz+a_2 z^2+a_3x+a_4z\in \k[x,z]$, $a_0,\ldots,a_4\in \k$, $r\in \k[z]$ has degree $\le 3$ and $(a_0,a_1)\not=(0,0)$.

If $\alpha^*(z)\in \k[z]$, then $\lambda(f)=\deg_x(a)\in \{1,2\}$. Otherwise, either $f$ is algebraically stable $($in which case $\lambda(f)=3)$ or $f$ is conjugate by an element of $\Aut(\AA^3)$ to an affine-triangular automorphism of degree $\le 3$, or we can conjugate $f$ by an affine automorphism and reduce to one of the following cases: 
\begin{enumerate}
\item\label{DynamDegCaser3}
$\deg(r)=3$, $\alpha^*(x)\in \k[z]$ and the coefficient of $z^3$ in $f_3$ is zero;
\item\label{DynamDegCase2a}
$\deg(r) \leq 2$, $\alpha^*(y)\in \k[z]$ and $\alpha^*(z)\in \k[y,z]$;
\item\label{DynamDegCase2b}
$\deg(r)  \leq 2$, $\alpha^*(y)\in \k[z]$, $\alpha^*(x)\in \k[y,z]$ and $a_2 = 0$;
\item\label{DynamDegCase2c}
$\deg(r) \leq 2$, $\alpha^*(x)\in \k[z]$, $\alpha^*(y)\in \k[y,z]$, $a_1\not=0$ and $a_2 = 0$.
\end{enumerate}
\end{proposition}

\begin{proof}
	$(A)$ Suppose first that $\alpha^*(z)\in \k[z]$. Since the dynamical degree of the automorphism 
	$z \mapsto \alpha^*(z)$ of $\AA^1$
	is $1$, by~\cite[Lemma~2.3.1]{BlSa2022Dynamical-degrees-} the dynamical degree of $f$ is given by
	$\lambda(f) = \lim_{r \to \infty} \deg_{x, y}(f^r)^{\frac{1}{r}} $.  If $\deg_x(a)=1$, then $\deg_{x,y}(f^r)=1$ for each $r\ge 1$, so $\lambda(f)=1$. We then suppose that $\deg_x(a)=2$ and prove that $\lambda(f)=2$. Choosing $\mu=(1,1,0)$, we find $\deg_{x,y}(p)=\deg_\mu(p)$ for each $p\in \k[x,y,z]$. As $za(x,z)$ and $a(x,z)$ are $\k$-linearly independent, one finds $\deg_\mu(f_1)=\deg_\mu(f_2)=2$ and $\deg_\mu(f_3)=0$. Hence, $\deg_\mu(f)=2$ and the $\mu$-leading part of $f$ is $g=(g_1,g_2,g_3)$, where $g_3=f_3 \in \k^\ast z + \k$ and 
	$g_1,g_2\in (\k x^2z+\k x^2) \setminus \{0\}$. This implies by induction on $r$
	that no component  of $g^r$ is zero, for each $r\ge 1$, which implies that $\lim_{r \to \infty} \deg_{\mu}(f^r)^{\frac{1}{r}} =2$ \cite[Lemma~2.6.1(5)]{BlSa2022Dynamical-degrees-}. This gives $\lambda(f) =2$.
	
	We may thus assume that $\alpha^*(z)\not\in \k[z]$ in the sequel. We denote by $\overline{f},\overline{g}\in \Bir(\p^3)$ and $\overline{\alpha},\overline{\tau}\in \Aut(\p^3)$ the extensions of $f,g$ and $\alpha,\tau$, via the standard embedding $\A^3\hookrightarrow \p^3$, $(x,y,z)\mapsto [1:x:y:z]$ and denote as usual by $H_\infty$ the hyperplane $\p^3\setminus \A^3$ given by $w=0$
	where $w, x, y, z$ denote the homogeneous coordinates of $\PP^3$.
	 Denoting by $f_{i,j}$ the homogeneous part of $f_i$ of degree $j$, the restriction of $\overline{f}$ to $H_\infty$ is given by $[0:x:y:z]\mapsto [0: f_{1,3}(x,y,z):f_{2,3}(x,y,z): f_{3,3}(x,y,z)]$. 

$(B)$  Suppose now that $\deg(r)=3$. This implies that $\Span_\k(f_{1,3},f_{2,3}, f_{3,3})\subset \k[x,z]_3$ has dimension $2$. Hence, the image by $\overline{f}$ of $H_\infty$ is a line $\ell\subset H_\infty$ (as $(a_0, a_1) \neq (0, 0)$) and the base-locus of $f$ is the 
 line $\ell_z \subset H_\infty$ given by $z = 0$. 
As $\overline{g}(H_\infty)$ is the line $\ell_z$ 
and as $\alpha^*(z) \not\in \k[z]$, the line $\ell=\overline{\alpha}(\ell_z)\subset H_\infty$ satisfies $\ell\not=\ell_z$.  If $\overline{f}$ restricts to a dominant rational map $\ell\dasharrow \ell$, then $f$ is algebraically stable, and the same holds if 
 $\overline{f}(\ell \setminus \ell_z)$  
is a point of  $\ell \setminus \ell_z$.  
We may thus assume that 
$\overline{f}(\ell \setminus \ell_z) = \ell \cap \ell_z \in H_\infty$. 
The fact that $\overline{f}(\ell \setminus \ell_z)$
and thus also $\overline{g}(\ell \setminus \ell_z)$
is a point implies that $\ell=\overline{\alpha}(\ell_z)$ passes through the point $[0:0:1:0]$
and thus $\ell$ is given by $x=\mu z$ for some $\mu\in \k$. We may conjugate $f$ with $\kappa=(x-\mu z,y,z)\in \Aff(\A^3)$ (this replaces $\alpha$ with $\kappa\circ \alpha$  and $g$ with $g\circ \kappa^{-1}$ so does not change the form of $g$)  and assume that $\mu=0$.

 Since  $\overline{f}(\ell \setminus \ell_z) = \ell \cap \ell_z = [0:0:1:0]$, 
 the coefficient of $z^3$ of $f_3$ (and of $f_1$) is equal to zero. As $\overline{\alpha}(\ell_z)$ is the line $x=0$, we get $\alpha^*(x)\in \k[z]$. We are thus in Case~\ref{DynamDegCaser3}.

(C): We may now assume that $\deg(r)<3$ (and still $\alpha^*(z)\not\in \k[z]$).
We write
\[
	\alpha = 
	(\alpha_{11}x+\alpha_{12}y+\alpha_{13}z+\beta_1,
	\alpha_{21}x+\alpha_{22}y+\alpha_{23}z+\beta_2,
	\alpha_{31}x+\alpha_{32}y+\alpha_{33}z+\beta_3)
\]
where $\alpha_{ij}\in \k$ and $\beta_{i}\in \k$ for all $i,j\in \{1,2,3\}$. 
As $\deg(r)<3$ the vector space $\Span_\k(f_{1,3},f_{2,3}, f_{3,3})\subset \k[x,z]_3$ has dimension $1$. The image of $H_\infty$ by $\overline{f}$  
is the point $q=[0:\alpha_{11}:\alpha_{21}:\alpha_{31}]\in H_\infty$ and the base-locus of $f$ is the union of three lines (maybe with multiplicity).  If $q$ is not in the base-locus, then $f$ is algebraic stable. We may thus assume that $f_{i,3}(q)=0$ for each $i$. We distinguish the possible cases, depending on whether $\alpha_{11}$ and $\alpha_{31}$ are zero or not.

(C1): Assume first that $\alpha_{11}=\alpha_{31}=0$. 
As $\alpha^*(z)\not\in \k[z]$, we get $\alpha_{32}\not=0.$ Conjugating by $\kappa=(x-\alpha_{12}/\alpha_{32}z,y,z)$ (this replaces $\alpha$ with $\kappa\circ \alpha$  and $g$ with $g\circ \kappa^{-1}$ so does not change the form of $g$), we may assume that $\alpha_{12}=0$. 

As $g=(x+yz+\xi z, y + r(z), z)\circ (x,y+a(x,z),z)$, we find
\[
\begin{array}{rcl}
	h  &=& (h_1, h_2, h_3) = (x,y+a(x,z) + r(z),z) \circ f\circ (x,y-a(x,z)-r(z),z) \\
	    &=& (x, y + a(x, z) + r(z), z) \circ \alpha \circ (x+(y-r(z))z+\xi z, y, z)
\end{array}
\]
with
\[
\begin{array}{rcl}
	h_1 &=& \alpha_{13} z + \beta_1 \\
	h_3 &=& \alpha_{32}y + \alpha_{33}z + \beta_3 \\
	h_2 &=& \alpha_{21}(x + (y-r(z))z + \xi z) + \alpha_{22}y + \alpha_{23}z + \beta_2 +a(h_1, h_3) + r(h_3)
\end{array}
\]
We see that $h$ is affine-triangular of degree $\leq 3$ 
and thus $f$ is conjugate to an affine triangular automorphism of degree $\leq 3$.

(C2): Assume now that $\alpha_{11}\not=0$ and $\alpha_{31}=0$. The equality $\alpha_{31}=0$ corresponds to $\alpha^*(z)\in \k[y,z]$. As $\alpha^*(z)\not\in \k[z]$, we have $\alpha_{32}\not=0$. Conjugating by $\kappa=(x,y-\alpha_{21}/\alpha_{11}x,z)$ we may assume that $\alpha_{21}=0$ (as before, this replaces $g$ with $g\circ \kappa^{-1}$ and thus does not change the form of $g$). We then conjugate by $(x,y-\alpha_{22}/\alpha_{32} z,z)$ and may assume that $\alpha_{22}=0$, so $\alpha^*(y)\in \k[z]$. We are thus in Case~\ref{DynamDegCase2a}.

(C3): Assume now that $\alpha_{31}\not=0$. Conjugating by $\kappa=(x-\alpha_{11}/\alpha_{31}z,y-\alpha_{21}/\alpha_{31}z,z)$, we may assume that $\alpha_{11}=\alpha_{21}=0$, so $\alpha^*(x),\alpha^*(y)\in \k[y,z]$ and $q = [0:0:0:1]$. 
As $f_{3, 3}(q) = 0$ and as the coefficient of $x$ in $\alpha^\ast(z)$ is non-zero, 
we get $a_2 = 0$.
If $\alpha_{12}\not=0$, we conjugate by $(x,y-\alpha_{22}/\alpha_{12} x,z)$ and may assume that $\alpha_{22}=0$, so  $\alpha^*(y)\in \k[z]$, giving  Case~\ref{DynamDegCase2b}. If $\alpha_{12}=0$ 
and $a_1\not=0$, we get Case~\ref{DynamDegCase2c}. We may thus assume that $\alpha_{11}=\alpha_{12}=\alpha_{21}=0$ and $a_1=a_2=0$.
This gives $\alpha^*(x)\in \k[z]$, $\alpha^*(y)\in \k[y,z]$ and $a(x,z)=a_0 x^2+a_3x+a_4z$, with 
$a_0\not=0$. 
Then, 
\[
 	\begin{array}{rcl}
 		h&=&(h_1,h_2,h_3)=(x,y+a_3x+a_0x^2,z)\circ  f\circ (x,y-a_3x-a_0x^2,z)\\
 		  &=& (x,y+a_3x+a_0x^2,z)\circ\alpha\circ (x+yz + a_4z^2+\xi z, y + a_4z + r(z), z)
 	\end{array}
\] 
is such that $h_1\in \k[z]$, $h_2\in \k[y,z]$ and $h_3\in \k[x,y,z]$ are of degree $\le 2$. Hence, $f$ is conjugate by an element of $\Aut(\AA^3)$ to an affine-triangular automorphism of degree $\le 2$.
 \end{proof}
\smallskip 
 \begin{proposition}\label{Prop:DynDeg3NotTriangDynDeg} 
 The dynamical degree of any $f = \alpha \circ g$ as in the four Cases~\ref{DynamDegCaser3}-\ref{DynamDegCase2a}-\ref{DynamDegCase2b}-\ref{DynamDegCase2c} of Proposition~$\ref{Prop:DynDeg3NotTriang}$, is given as follows:
\begin{enumerate}
\item[\ref{DynamDegCaser3}]
$\lambda(f)=\left\{\begin{array}{ll}
1+\sqrt{2}&\text{ if }a_1\not=0;\\
(1+\sqrt{13})/2&\text{ if }a_1=0.\end{array}\right.$
\item[\ref{DynamDegCase2a}]
$\lambda(f)=\left\{\begin{array}{ll}
1+\sqrt{3}&\text{ if }a_0\not=0;\\
1+\sqrt{2}&\text{ if }a_0=0.\end{array}\right.$
\item[\ref{DynamDegCase2b}]
Writing the coefficient of $z^2$ in $f_1$ as $\varepsilon$, we obtain 
\[
	\lambda(f)=\left\{\begin{array}{llll}
1+\sqrt{3}&\text{ if }a_1\neq 0& \text{ and }&\varepsilon\neq 0;\\
(3 + \sqrt{5})/2&\text{ if }a_1\neq 0& \text{ and }&\varepsilon=0;\\
 (1 + \sqrt{17})/2&\text{ if }a_1= 0& \text{ and }&\varepsilon \neq 0;\\
2&\text{ if }a_1= 0& \text{ and }&\varepsilon = 0.\end{array}\right.
\]
\item[\ref{DynamDegCase2c}] $\lambda(f)=1+\sqrt{2}$.
\end{enumerate}
\end{proposition}
\begin{proof}
\ref{DynamDegCaser3}: We have $\deg(r)=3$, $\alpha^*(x)\in \k[z]$ and the coefficient of $z^3$ in $f_3$ is zero. This gives $f_1=f_{1,0}+f_{1,1}\in \k[z]$ and implies that the coefficient of $z^3$ in $f_2$ is not zero. Let $\theta$ be in the open intervall $(2, 3)$ and choose $\mu= (1,3, \theta)$. 
The $\mu$-degree of $z^3$ is bigger than any other monomial that occurs in $f_1, f_2$ or $f_3$, as $\theta > 2$.
We get $\deg_\mu(f_1)=\theta$, $\deg_\mu(f_2)=3\theta$, with $\mu$-leading parts equal to $\zeta_1 z$ and $\zeta_2 z^3$ for some $\zeta_1,\zeta_2\in \k^*$,
respectively.
As the coefficient of $z^3$ in $f_3$ is zero, the monomial $yz$ occurs in $f_3$. Hence, the $\mu$-leading part 
of $f_3$ belongs to $(\k yz+\k xz^2)\setminus \{0\}$. Indeed, as $\deg_\mu(y)>\deg_\mu(z)>\deg_\mu(x)$,   $\deg_\mu(yz)=3+\theta$ is the biggest $\mu$-degree of the monomials of degree $\le 2$ appearing in $f$; moreover $\deg_\mu(yz)>\deg_\mu(x^2z)=2+\theta$.
 
If  $a_1\not=0$, the coefficient of $xz^2$ in $f_3$ is not zero, 
 so $t\in \k xz^2 $ (since $\theta > 2$). We choose $\theta=1+\sqrt{2}$ and observe that $\theta^2=2\theta +1$. Thus we obtain $\deg_\mu(f)=\theta$, with $\mu$-leading part $g=(\zeta_1 z, \zeta_2 z^3,\zeta_3 xz^2)$, where $\zeta_3\in \k^*$.
 
If $a_1=0$, then 
$t\in \k yz$. We choose $\theta=(1+\sqrt{13})/2$ and observe that
$\theta^2=\theta +3$. Thus we obtain  $\deg_\mu(f)=\theta$, with $\mu$-leading part $g=(\zeta_1 z, \zeta_2 z^3,\zeta_3 yz)$, where $\zeta_3\in \k^*$.

As $g$ is monomial, we have $g^r\not=0$ for each $r\ge 1$, so  $\lambda(f)$ is equal to $\theta$ in both cases (Proposition~\ref{prop.Comp_DDeg}).
 
\ref{DynamDegCase2a}: 
We have $\deg(r) \leq 2$,  $\alpha^*(y)\in \k[z]$ and  $\alpha^*(z)\in \k[y,z]$. This gives
\[
	\begin{array}{rcl}
		f_1&=&f_{1,0}+f_{1,1} + f_{1,2} +\zeta_1 z(a_0 x^2+a_1 xz + a_2 z^2),\\
		f_2&=&f_{2,0}+\zeta_2 z,\\
 		f_{3}&=&f_{3,0}+f_{3,1}+\zeta_3 (a_0 x^2+a_1 xz)+\varepsilon_3 z^2,
 	\end{array} 
 \]
where $\zeta_1,\zeta_2,\zeta_3\in \k^*$, $\varepsilon_3\in \k$.

If $a_0\not=0$, we choose $\theta=1+\sqrt{3}$, $\mu=(\theta +1,1,\theta)$ and observe that
$\theta^2=2\theta+2$. Then, $\deg_\mu(f)=\theta$, with $\mu$-leading part $(\zeta_1 a_0 x^2z, \zeta_2 z, \zeta_3 a_0x^2)$.  This gives $\lambda(f)=\theta$ by Proposition~\ref{prop.Comp_DDeg}.

If $a_0=0$, then $a_1\not=0$. We choose $\theta=1+\sqrt{2}$, $\mu=(\theta +1,1,\theta)$
and observe that $\theta^2=2\theta+1$. Then, $\deg_\mu(f)=\theta$, with $\mu$-leading part $(\zeta_1 a_1 xz^2, \zeta_2 z, \zeta_3 a_1xz)$.  This gives $\lambda(f)=\theta$ by Proposition~\ref{prop.Comp_DDeg}.

\ref{DynamDegCase2b}:
We have $\deg(r) \leq 2$,  $\alpha^*(y)\in \k[z]$, $\alpha^*(x)\in \k[y,z]$ and $a_2 = 0$. 
This gives
\[
	\begin{array}{rcl}
		f_1&=&f_{1,0}+f_{1,1}+\zeta_1 (a_0 x^2+a_1 xz)+\varepsilon_3 z^2, \\
		f_2&=&f_{2,0}+\zeta_2 z,\\
		f_3&=&f_{3,0}+f_{3,1} + f_{3,2} +\zeta_3 z(a_0 x^2+a_1 xz),
	\end{array} 
\]
where $\zeta_1,\zeta_2,\zeta_3\in \k^*$, $\varepsilon_3\in \k$.

If $a_1 \neq 0$ and $\varepsilon_3 \neq 0$, then we choose $\theta = 1 + \sqrt{3}$, $\mu = (2, 1, \theta)$
and observe that $\theta^2 = 2 \theta + 2$. Then,
$\deg_{\mu}(f) = \theta$, with $\mu$-leading part
$(\varepsilon_3 z^2, \zeta_2 z, \zeta_3 a_1 xz^2)$. This gives $\lambda(f) = \theta$ by Proposition~\ref{prop.Comp_DDeg}.

If $a_1 \neq 0$ and $\varepsilon_3 = 0$, then we choose $\theta = (3 + \sqrt{5})/2$, 
$\mu = (1, \theta -2, \theta - 1)$ and observe that $\theta^2 = 3 \theta -1$. 
Then $\deg_{\mu}(f) = \theta$, with $\mu$-leading part
$(\zeta_1 a_1 xz, \zeta_2 z, \zeta_3 a_1 x z^2)$. This gives $\lambda(f) = \theta$ by Proposition~\ref{prop.Comp_DDeg}.

If $a_1 = 0$ and $\varepsilon_3 \neq 0$, then $a_0 \neq 0$ and we choose  $\theta = (1 + \sqrt{17})/2$,
$\mu = (2, 1, \theta)$. Observe that $\theta^2 = \theta + 4$. 
Then $\deg_{\mu}(f) = \theta$, with $\mu$-leading part
$(\varepsilon_3 z^2, \zeta_2 z, \zeta_3 a_0 x^2 z)$. This gives $\lambda(f) = \theta$ by Proposition~\ref{prop.Comp_DDeg}.

If $a_1 = \varepsilon_3 = 0$, then $a_0 \neq 0$ and we choose $\theta = 2$, $\mu = (1, 1, \theta)$.
Then $\deg_{\mu}(f) = \theta$, with $\mu$-leading part
$q = (\zeta_1 a_0 x^2 + \xi_1 z, \zeta_2 z, \zeta_3 a_0 x^2 z +\xi_{3} z^2)$ 
for some $\xi_1$, $\xi_3 \in \k$.
Let $\hat{q} \colon \AA^2 \to \AA^2$,
$(x, z) \mapsto (\zeta_1 a_0 x^2 + \xi_1 z,  \zeta_3 a_0 x^2 z +\xi_{3} z^2)$
and observe that $\hat{q}$ is dominant (as $\zeta_1 a_0$ and $\zeta_3 a_0$ are both non-zero).
As $\pi \circ q = \hat{q} \circ \pi$ for $\pi \colon \AA^3 \to \AA^2$, $(x, y, z) \mapsto (x, z)$, 
it follows that $q^r \neq 0$ for each $r \geq 1$. This gives $\lambda(f) = \theta$ by Proposition~\ref{prop.Comp_DDeg}.

\ref{DynamDegCase2c}: 
We have $\deg(r) \leq 2$, $\alpha^*(x)\in \k[z]$, 
$\alpha^*(y)\in \k[y,z]$, $a_1 \neq 0$ and $a_2 = 0$. This gives
\[
 \begin{array}{rcl}
	f_1&=&f_{1,0}+\zeta_1 z,\\
	f_2&=&f_{2,0}+f_{2,1}+\zeta_2 (a_0 x^2+a_1 xz)+\varepsilon_2 z^2, \\
 	f_3&=&f_{3,0}+f_{3,1} + f_{3, 2} +\zeta_3 z(a_0 x^2+a_1 xz),
 \end{array} 
\]
where $\zeta_1,\zeta_2,\zeta_3 \in \k^*$, $\varepsilon_2 \in \k$.
We choose $\theta = 1 + \sqrt{2}$, $\mu=(1,2,1 + \sqrt{2})$ and observe that $\theta^2 = 2\theta + 1$.
Then $\deg_\mu(f)=\theta$ with $\mu$-leading part $(\zeta_1 z, \varepsilon_2 z^2, \zeta_3 a_1 x z^2)$. 
As $a_1 \neq 0$ and $\zeta_1, \zeta_3 \neq 0$, this  gives $\lambda(f)=\theta$ by Proposition~\ref{prop.Comp_DDeg}.
\end{proof} 

\begin{example}\label{Exa:DiffPropDynDeg3NotTriang}
We illustrate the different cases \ref{DynamDegCaser3}-\ref{DynamDegCase2c} of Proposition~\ref{Prop:DynDeg3NotTriang} 
and Proposition~\ref{Prop:DynDeg3NotTriangDynDeg}, by giving a simple example in each possible case
and we give examples for the two cases where $\alpha^*(z)=z$ and the case where $f$ is algebraically stable.
All of them are of the form $\alpha\circ g$, where $\alpha\in \Aff(\A^3)$,  $g=(x+yz + za(x, z), y + a(x, z) + r(z), z)$, $a=a_0x^2+a_1xz+a_2z^2\in \k[x,z]\setminus \k[z]$ is homogeneous of degree $2$ and $r\in \k[z]$ is of degree $\le 3$. 
\[
	\begin{array}{|c|c|c|c|c|}
	\hline
	\text{Case} & a & r& f\in \Aut(\A^3) & \lambda(f)\\
	\hline
	& xz & 0 & (x+yz+xz^2,y+xz,z) & 1\\
	& x^2 & 0 & (x+yz+x^2z,y+x^2,z) & 2\\
	& xz & z^3& (x+yz+xz^2,z,y+xz+z^3) & 3\\
	\ref{DynamDegCaser3} & xz & z^3 & (z, y+xz+z^3, x+yz+xz^2) & 1+\sqrt{2}\\
	\ref{DynamDegCaser3} & x^2& z^3&  (z,y+x^2+z^3,x+yz+x^2z) & (1+\sqrt{13})/2\\
	\ref{DynamDegCase2a} & x^2 & 0 & (x+yz + zx^2,z, y + x^2) & 1+\sqrt{3}\\
	\ref{DynamDegCase2a} & xz & 0 & (x+yz + xz^2,z,y + xz) & 1+\sqrt{2}\\
	\ref{DynamDegCase2b}& xz&  z^2 & (y + xz + z^2, z,x+yz + xz^2) & 1+\sqrt{3}\\
	\ref{DynamDegCase2b}& xz&  0 & (y + xz, z,x+yz + xz^2) & (3 + \sqrt{5})/2\\
	\ref{DynamDegCase2b}& x^2&  z^2 & (y + x^2 + z^2, z,x+yz + x^2z) &  (1 + \sqrt{17})/2\\
	\ref{DynamDegCase2b}& x^2&  0 & (y + x^2, z,x+yz + x^2z) &2\\
	\ref{DynamDegCase2c}& xz &  0 & (z, y + xz,x+yz + xz^2) & 1+\sqrt{2} \\
	\hline
	\end{array}
\]
\end{example}

\begin{proof}[Proof of Theorem~$\ref{thm.Dynamical_degrees}$]
Corollary~\ref{Cor:DynDegSmallAffTriang} gives the values of $\Lambda_{1,\kany}$ and $\Lambda_{2,\kany}$, proves that $\Lambda_{3,\kany}$ contains $L_3=\{1 , \sqrt{2} , \ \frac{1+\sqrt{5}}{2} , \ \sqrt{3} , 2 , \ \frac{1 + \sqrt{13}}{2} , \
1 + \sqrt{2} , \ \sqrt{6} , \ \frac{1 + \sqrt{17}}{2} , \ 1+ \sqrt{3} , \ 3 \}$ and that for each $f\in \Aut(\A^3_{\kany})$ which is 
conjugated in $\Aut(\A^3_{\overline{\kany}})$ to an affine triangular automorphism of degree $\le 3$ 
$($where $\overline{\kany}$ is a fixed algebraic closure of $\kany)$, we have $\lambda(f)\in L_3$.
	
	Moreover, the element $(y + xz, z,x+z(y + xz))\in \Aut(\A^3_{\kany})$ has dynamical degree $(3 + \sqrt{5})/2$ (follows from Proposition~\ref{Prop:DynDeg3NotTriangDynDeg} as it belongs to Case~\ref{DynamDegCase2b} with $a_1\not=0$ and $\varepsilon=0$, see also Example~\ref{Exa:DiffPropDynDeg3NotTriang}). 
	
It remains then to see that each element  $f\in \Aut(\A^3_{\kany})$ of degree $3$ has a dynamical degree which is either equal to $(3 + \sqrt{5})/2$ or belongs to $L_3$. 
By Theorem~\ref{thm:Classification},  $f$ is conjugate in $\Aut(\A^3_{\overline{\kany}})$ either to an affine-triangular automorphism or to 
$f = \alpha \circ (yz + za(x, z) + x, y + a(x, z) + r(z), z)$ where 
$a\in \k[x, z] \setminus \k[z]$ is homogeneous of degree $2$  
and $r \in \k[z]$ is of degree $\le 3$. In the first case, $\lambda(f)\in L_3$ by Corollary~\ref{Cor:DynDegSmallAffTriang}. In the second case, Propositions~\ref{Prop:DynDeg3NotTriang} and \ref{Prop:DynDeg3NotTriangDynDeg} show that either  $\lambda(f)=(3 + \sqrt{5})/2$ or $\lambda(f)\in L_3$. This achieves the proof.
\end{proof}

\par\bigskip

\renewcommand{\MR}[1]{}

\providecommand{\bysame}{\leavevmode\hbox to3em{\hrulefill}\thinspace}
\providecommand{\MR}{\relax\ifhmode\unskip\space\fi MR }
\providecommand{\MRhref}[2]{%
	\href{http://www.ams.org/mathscinet-getitem?mr=#1}{#2}
}
\providecommand{\href}[2]{#2}

\end{document}